\documentclass[a4paper, 11pt]{amsart}

\usepackage{graphicx}

\usepackage[left=3cm, right=3cm]{geometry}
\graphicspath{ {figures} }
\usepackage[]{amsmath}
\usepackage{overpic}
\usepackage{xcolor}
\usepackage{caption}
\usepackage{enumerate}
\usepackage{paralist}
\usepackage{tikz}
\usepackage{amssymb}
\usepackage{textcomp} %\textbullet, \textopenbullet for bi-colored vertices
\usepackage{ wasysym } %\fullmoon, \newmoon for larger bi-colored vertices
\usetikzlibrary{intersections, calc}
\usepackage[implicit=false]{hyperref}

\usepackage{stackengine}
\usepackage{tikz}
\usepackage{harmony}
\usepackage{scalerel}

\usepackage{tikz}
\usepackage{tkz-euclide}
\usepackage{setspace}
\usepackage{appendix}
\usepackage{caption}
\let\OLDthebibliography\thebibliography
\renewcommand\thebibliography[1]{
	\OLDthebibliography{#1}
	\setlength{\parskip}{0pt}
	\setlength{\itemsep}{.8pt}
} % change biblography spacing

\newcommand{\Z}{{\mathbb Z}}
\newcommand{\R}{{\mathbb R}}

\newcommand{\G}{{\mathcal G}}
\newcommand{\Rt}{{\mathbb R}^{3}}
\newcommand{\Rto}{{\mathbb R}^{2, 1}}
\newcommand{\Rfo}{{\mathbb R}^{4, 1}}
\newcommand{\Rttwo}{{\mathbb R}^{3, 2}}

\newcommand{\RP}{\mathbb{R}\mathrm{P}}

\newcommand{\black}{b}

\renewcommand{\S}{{\mathcal S}}
\newcommand{\Sz}{{\mathcal S}_z}

\newcommand{\sn}{\ensuremath\operatorname{sn}}
\newcommand{\cn}{\ensuremath\operatorname{cn}}
\newcommand{\dn}{\ensuremath\operatorname{dn}}

\newcommand{\lorsca}[1]{\left<#1\right>} % lorentz sp
\newcommand{\lmsca}[1]{\left<#1\right>_{3, 2}} % lorentz mob sp

\newtheorem{theorem}{Theorem}[section]
\newtheorem{proposition}[theorem]{Proposition}
\newtheorem{lemma}[theorem]{Lemma and Definition}

\newtheorem{definition}[theorem]{Definition}
\newtheorem{corollary}[theorem]{Corollary}

\newenvironment{remark}{\par\noindent{\em Remark.} % \hskip\labelsep%
\normalfont}{\par}

\newlength{\subfigureheight}
\makeatletter
\@fpsep\textheight
\makeatother

\title[Constant mean curvature surfaces from ring patterns]{Constant mean curvature surfaces from ring patterns: Geometry from combinatorics}

\author{Alexander I.~Bobenko, Tim Hoffmann and Nina Smeenk}

\address{\hspace{-1.3\parindent} Alexander I.~Bobenko, Nina Smeenk
		\newline
		Institute of Mathematics, Secr. MA 8-3, TU Berlin, 10623 Berlin, Germany \newline
		Email: {\normalfont \{bobenko,smeenk\}@math.tu-berlin.de}
		 \newline \newline
		Tim Hoffmann 
		\newline
		Dept. of Mathematics, TU Munich, 85748 Garching, Germany\newline
		Email: {\normalfont hoffmant@ma.tum.de}}

  
  \date{\today}
  
\begin{document}
\maketitle

\begin{abstract}
	We define discrete constant mean curvature (cmc) surfaces in the three-dimensional Euclidean and Lorentz spaces in terms of sphere packings with orthogonally intersecting circles. These discrete cmc surfaces can be constructed from orthogonal ring patterns in the two-sphere and the hyperbolic plane. We present a variational principle that allows us to solve boundary value problems and to construct discrete analogues of some classical cmc surfaces. The data used for the construction is purely combinatorial - the combinatorics of the curvature line pattern. In the limit of orthogonal circle patterns we recover the theory of discrete minimal surfaces associated to Koebe polyhedra all edges of which touch a sphere. These are generalized to two-sphere Koebe nets, i.e., nets with planar quadrilateral faces and edges that alternately touch two concentric spheres.
\end{abstract}
%\tableofcontents

\section{Introduction}
\label{sec:introduction}

In the last two decades, the field of discrete
differential geometry emerged on the border of differential
and discrete geometry, see the books \cite{bobenko2008discrete, DDG-OberwolfachBook, DDG-SFBBook, DDG-Crane, DDG-Gu}. 
Whereas classical differential geometry investigates
smooth geometric shapes (such as surfaces), and discrete geometry studies
geometric shapes with a finite number of elements (polyhedra), discrete differential 
geometry aims at a development of discrete equivalents of the
geometric notions and methods of surface theory. The latter appears then
as a limit of the refinement of the discretization. Current progress in this
field is to a large extent stimulated by its relevance for computer graphics
and visualization.

One of the central problems of discrete differential geometry is to find
proper discrete analogues of special classes of surfaces, such as minimal,
constant mean curvature, isothermic, etc. Usually, one can suggest various
discretizations with the same continuous limit which have quite different
geometric properties. The goal of discrete differential geometry is to find a
discretization which inherits as many essential properties of the smooth
geometry as possible.

Our discretizations are based on quadrilateral meshes, i.e., we discretize
parametrized surfaces. For the discretization of a special class of surfaces,
it is natural to choose an adapted parametrization. In this paper, we
consider discretizations in terms of circles and spheres, which are treated as 
conformal curvature line discretizations of surfaces, and develop a method 
to investigate a discrete model for constant mean curvature (cmc) surfaces. 

\begin{figure}[htbp]
	\begin{minipage}{.30\linewidth}
		\includegraphics[width=\linewidth]{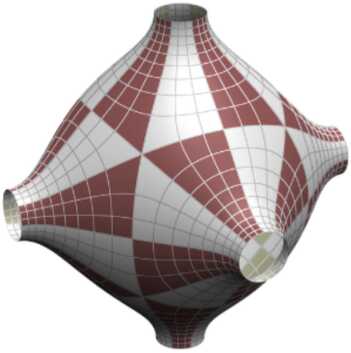}
	\end{minipage}
	\begin{minipage}{.30\linewidth}
		\includegraphics[width=\linewidth]{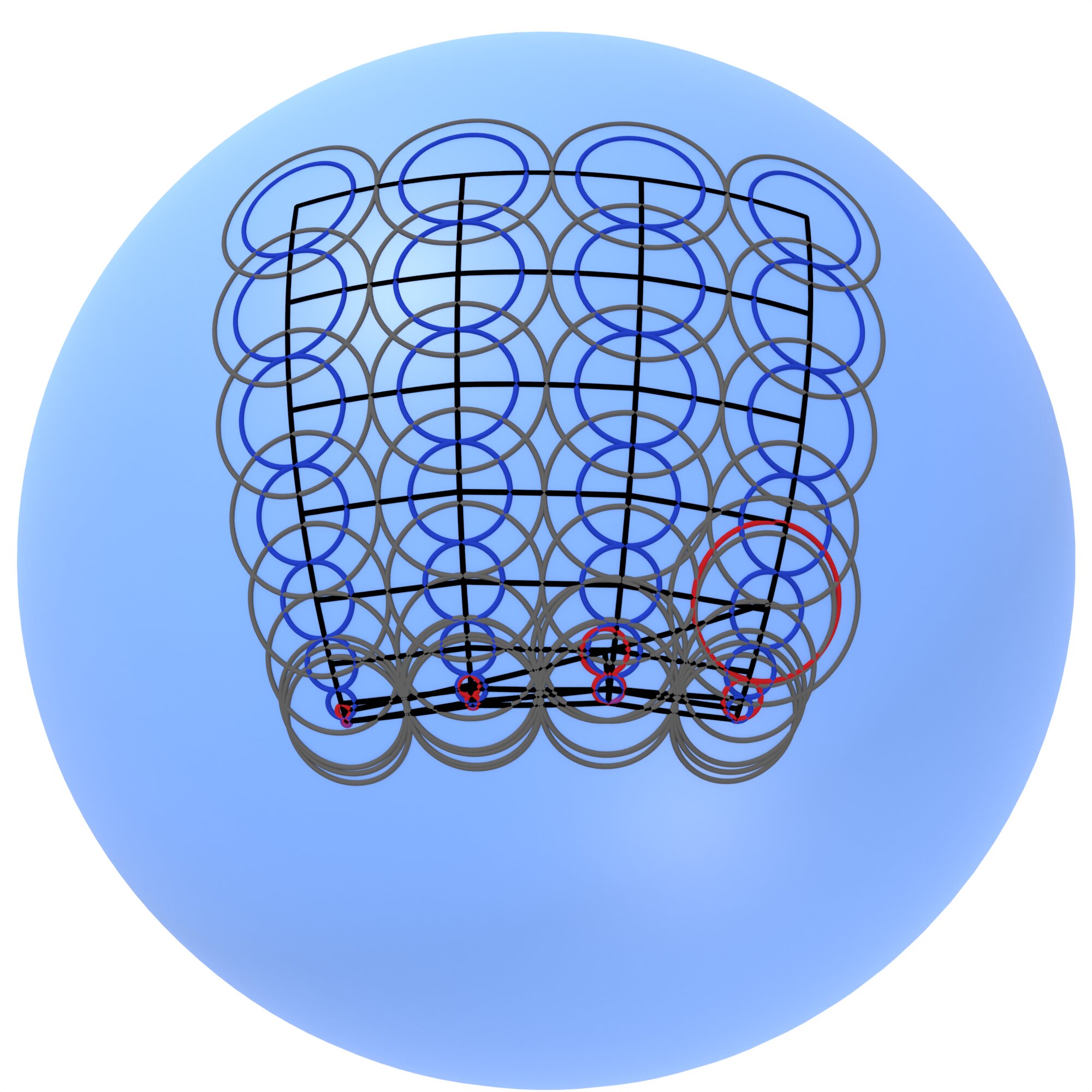}
	\end{minipage}
	%	\begin{minipage}{.325\linewidth}
		%		\includegraphics[width=\linewidth]{figures/Rt_Example_1_cmc}
		%	\end{minipage}
	\begin{minipage}{.30\linewidth}
		\includegraphics[width=\linewidth]{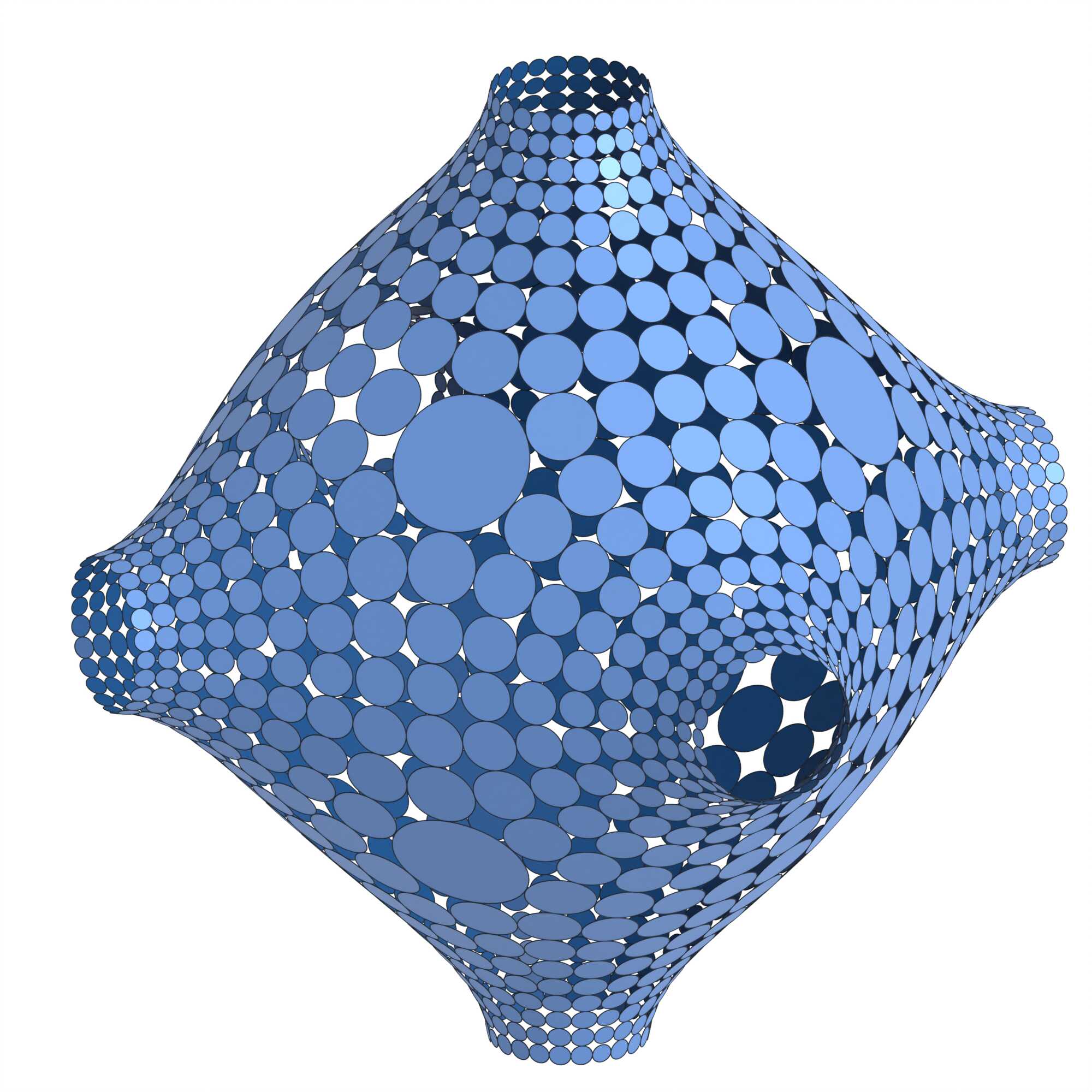}
	\end{minipage}
	
	\vspace{.5cm}
	
	\begin{minipage}{.30\linewidth}
		\includegraphics[width=\linewidth]{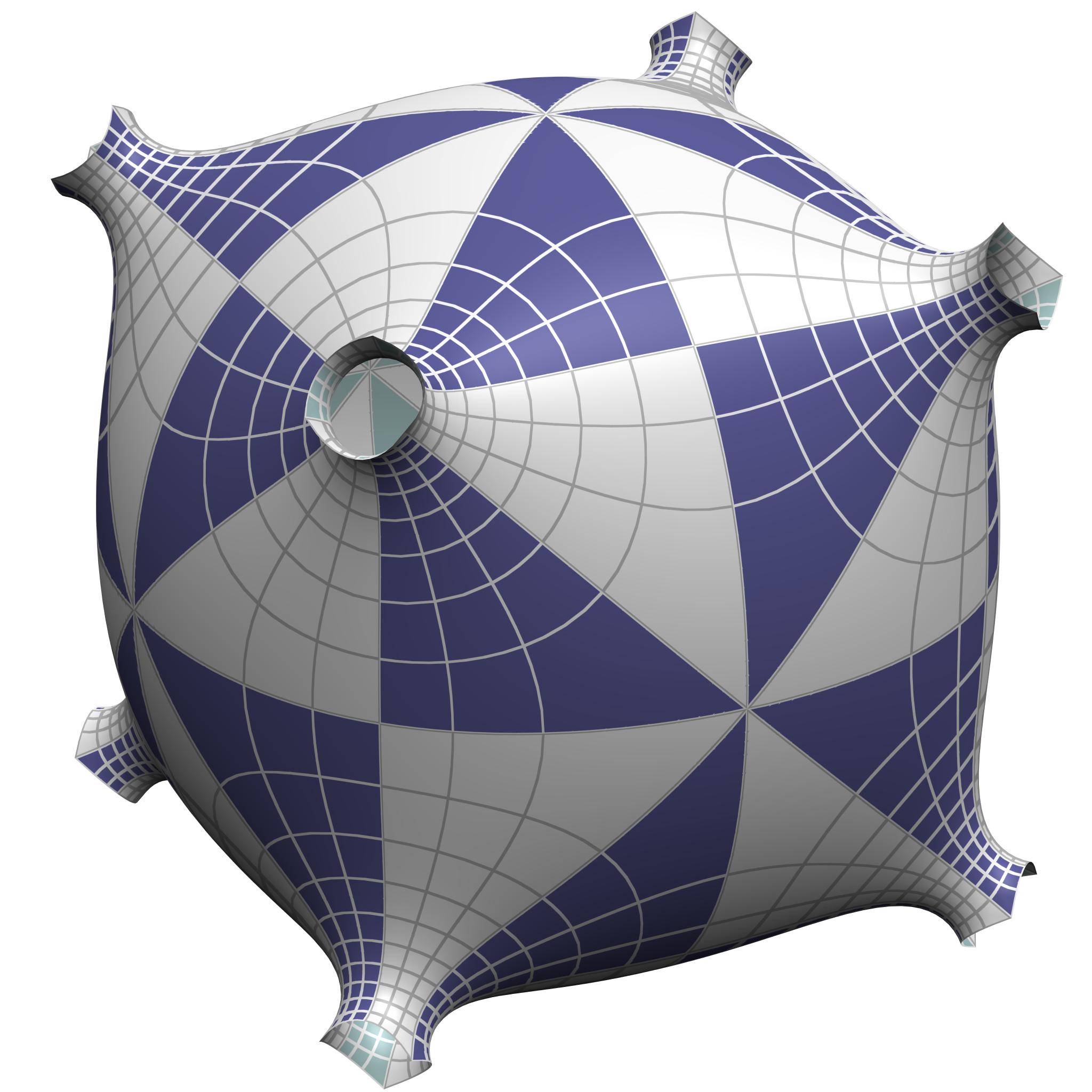}
	\end{minipage}
	\begin{minipage}{.30\linewidth}
		\includegraphics[width=\linewidth]{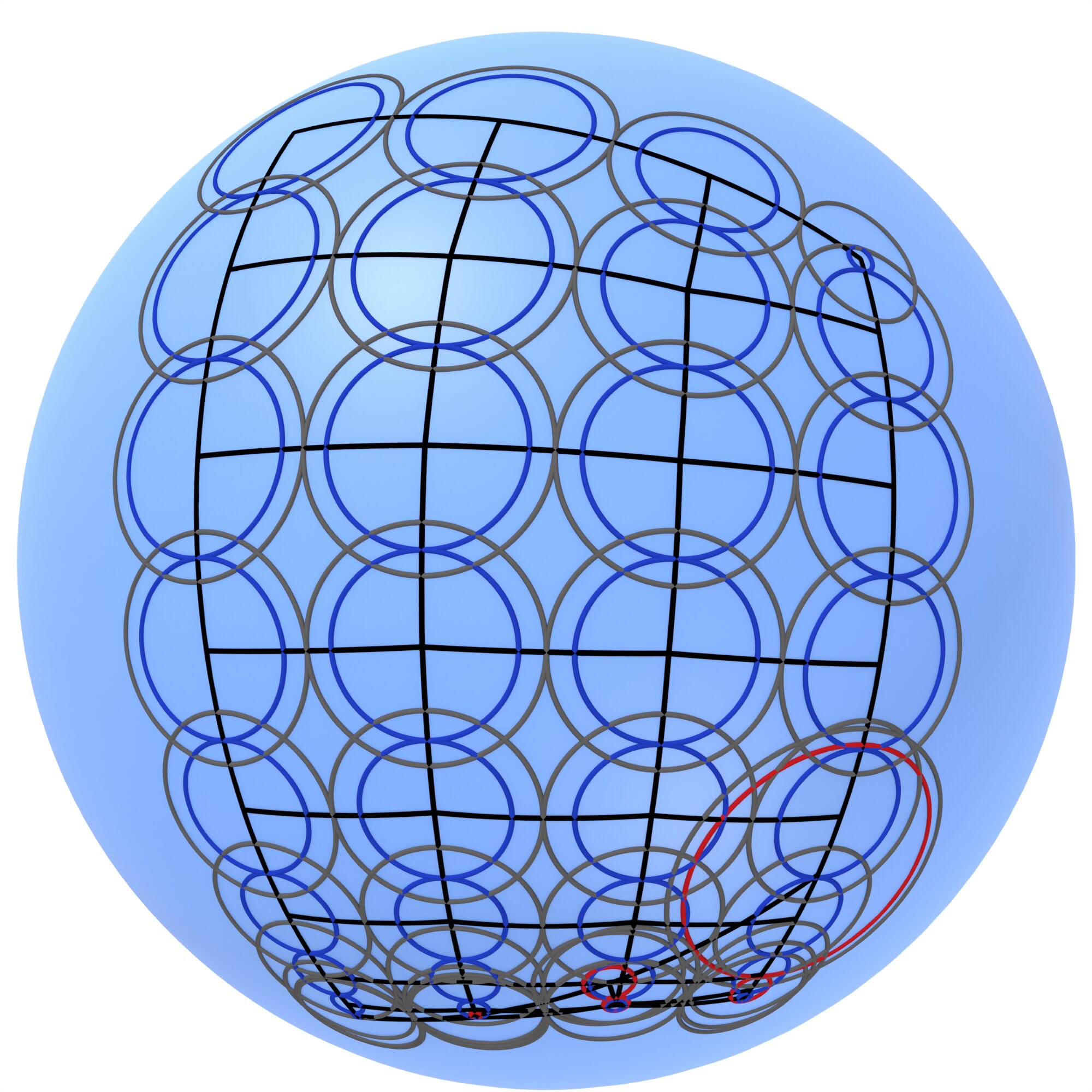}
	\end{minipage}
	%\begin{minipage}{.325\linewidth}
	%	\includegraphics[width=\linewidth]{figures/Rt_Example_2_cmc}
	%\end{minipage}7
	\begin{minipage}{.30\linewidth}
		\includegraphics[width=\linewidth]{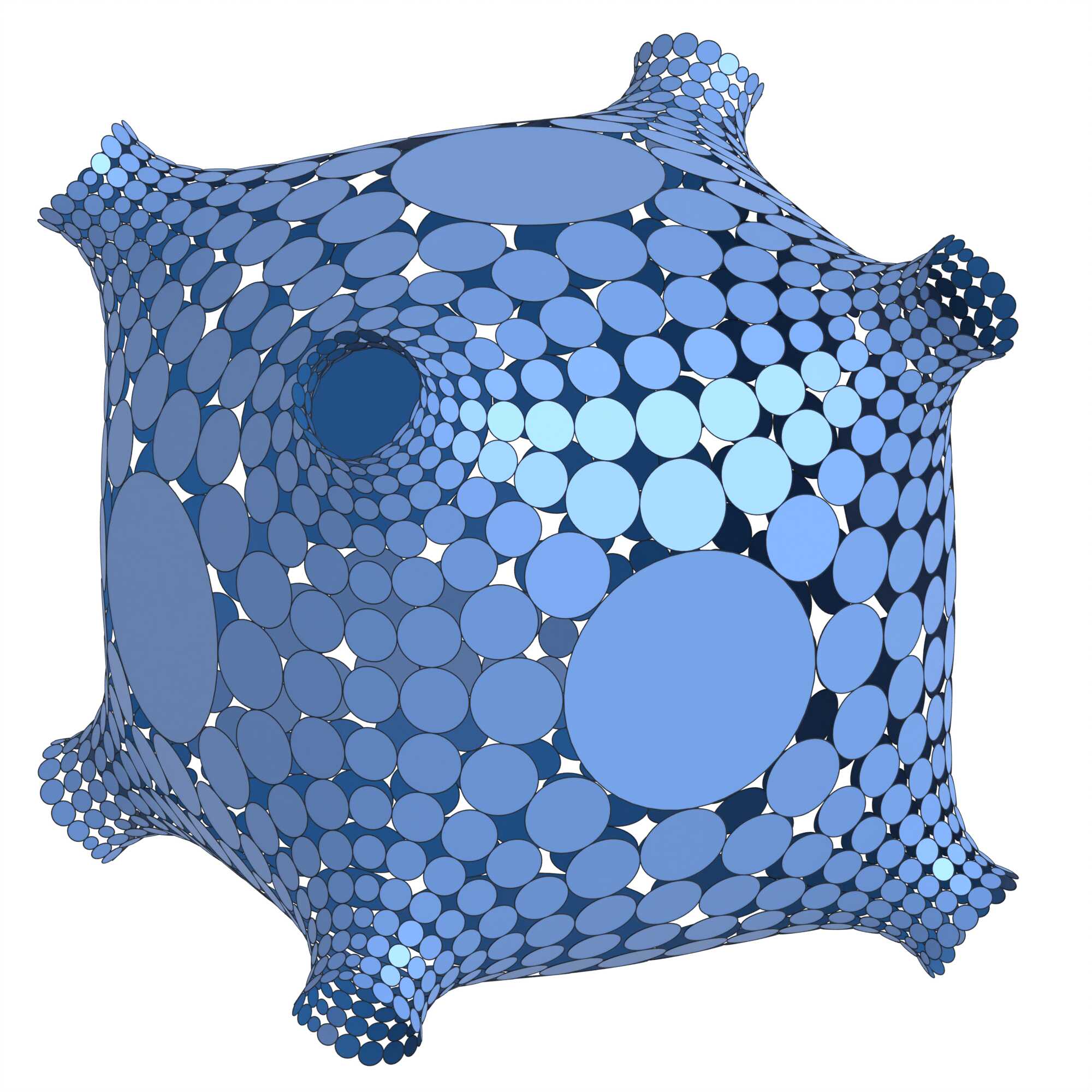}
	\end{minipage}
	\caption{Triply periodic cmc surfaces in $\Rt$: Schwarz's P surface (first line) and Schoen's I-WP surface (second line). The left column presents the smooth cmc surfaces from \cite{bobenko2021constant} obtained by the loop group method from the theory of integrable systems. The right column are the corresponding discrete cmc surfaces built by touching disks. Their spherical orthogonal ring patterns are shown in the middle.}
	\label{Fig:Introduction}
\end{figure}

We consider cmc surfaces as a subclass of isothermic surfaces. The
analogous discrete surfaces, {\em discrete S-isothermic surfaces} \cite{bobenko1999discretization, bobenko2008discrete}
consist of touching spheres, and of circles which intersect the spheres
orthogonally in their points of contact, see Section~\ref{sec:Rt_discrete_s_isothermic_surfaces} and Fig.~\ref{Fig:Rt_S_isothermic}. Isothermic surfaces are described by integrable equations, and thus belong to integrable differential geometry \cite{CiGoSy95-isothermic, bobenko1996discrete}.  They allow Darboux transformations and a duality transformation, called the Christoffel transformation. Smooth cmc surfaces $s$ are characterized among
isothermic surfaces by the property that their Christoffel dual $s^*$ is simultaneously a Darboux transform \cite{hertrich1997remarks}. In this case these surfaces are parallel, and define a common Gauss map $n=s^*-s$. 
The transformations and the characterization of cmc surfaces
carry over to the discrete domain. Thus, one arrives at the notion of {\em
  discrete S-isothermic cmc surfaces} \cite{hoffmann2010darboux, bobenko2016s}, or {\em discrete cmc surfaces}
for short. This definition is in complete agreement with the notion of discrete curvatures, see \cite{bobenko2010curvature}, defined for polyhedral surfaces using a discrete version of the classical Steiner formula. Indeed, Theorem~\ref{Thm:Rt_H=1} claims that the discrete mean curvature of the pair $(s,n)$ is constant.

As we show in Section~\ref{sec:Rt_discrete_s_isothermic_cmc_surfaces}, the role of the Gauss maps $n$ is played by a so called two-sphere Koebe net, a polyhedral surface with planar quadrilateral faces, all edges of which alternatively touch two concentric spheres $S^2_\pm$. This is a generalization of a classical Koebe polyhedron, all edges of which touch a common sphere $S^2$.
Koebe polyhedra are in one to one correspondence with orthogonal circle patterns on $S^2$. Using this identification,  Koebe \cite{Ko36} has shown that any 3-dimensional combinatorial convex polytope can be (essentially uniquely)
realized as a Koebe polyhedron. For generalizations and proofs of this theorem see \cite{BrSch93, schramm1992cage, Ri96, bobenko2004variational}, the latter also for a variational proof. In Section \ref{sec:Rt_two_spheres_Koebe_q_nets_and_spherical_orp} we show that two-sphere Koebe nets are in one to one correspondence to orthogonal ring patterns on a sphere. The latter were described analytically in \cite{bobenko2024rings}, including a variational description generalizing the description of orthogonal circle patterns. We present it in Section~\ref{sec:analytic_orp}. The variational principle allows us to construct particular orthogonal ring patterns determined by their combinatorics and also by some geometric boundary conditions. 

This leads to a construction method
for discrete S-isothermic cmc surfaces from orthogonal ring patterns.
Our general method to construct discrete cmc surfaces is
schematically shown in the following diagram, see also
Section~\ref{sec:Rt_constructing_discrete_s_isothermic_cmc_surfaces}:
\begin{center}
  continuous cmc surface \\
  {\large$\Downarrow$} \\
  image of curvature lines under Gauss map\\
  {\large$\Downarrow$} \\
  cell decomposition of (a branched cover of) the sphere\\
  {\large$\Downarrow$} \\
  spherical orthogonal ring pattern \\
  {\large$\Downarrow$} \\
  two-sphere Koebe Q-net\\
  {\large$\Downarrow$} \\
  discrete cmc surface
\end{center}

As usual in the theory of minimal and cmc surfaces \cite{HoKa97}, one starts
constructing such a surface with a rough idea of how it should look. To use
our method, one should understand the symmetries of its Gauss map and the {\em combinatorics}\/
of the curvature line pattern. The image of the curvature line pattern under
the Gauss map provides us with a cell decomposition of (a part of) $S^2$ or a
covering.  From these data we obtain a spherical orthogonal ring pattern with the prescribed combinatorics.  
Finally, a direct construction, see Theorem~\ref{Thm:Rt_Koebe_to_cmc}, recovers the pair $s, s^*$ of Christoffel dual cmc surfaces from their common Gauss map $n=s^*-s$.

Let us emphasize that our data, besides possibly boundary conditions, are
purely combinatorial -- the combinatorics of the curvature line pattern. All
faces are quadrilaterals and typical vertices have four edges. There may exist
distinguished vertices (corresponding to the ends or umbilic points of a
cmc surface) with a different number of edges.

The most nontrivial step in the above construction is the third one listed in
the diagram. It is based on the variational description of ring patterns. 
Under some constraints (see Sections~\ref{sec:Rt_constructing_discrete_s_isothermic_cmc_surfaces} and \ref{sec:Rt_discrete_cmc_surfaces_as_deformations_of_minimal_surfaces}) it implies the existence and 
uniqueness for the discrete cmc S-isothermic surface under
consideration, but not only this. This principle
provides us with a variational description of discrete cmc S-isothermic
surfaces and makes possible a solution of some Plateau problems as well.
It is also an effective tool for numerically constructing these surfaces.

In the limit of the zero mean curvature $H=0$, the two spheres coincide $S^2_\pm=S^2$, the ring patterns become circle patterns, and we recover the theory of discrete minimal surfaces developed in \cite{BHS_2006}. We describe this limit in Section~\ref{sec:Rt_discrete_cmc_surfaces_as_deformations_of_minimal_surfaces} and show that a generic discrete minimal surface possesses a one parameter deformation family of discrete cmc surfaces, parametrized by (small) $H$. This corresponds to a deformation of a circle pattern to ring patterns with the same combinatorial and boundary data and varying ``thickness'' parameter.  

In Section~\ref{sec:Rt_examples} we apply our method to construct some examples of periodic, reflectionally symmetric discrete cmc surfaces. In particular, we construct discrete cmc analogues of minimal surfaces of Schwarz and Schoen, shown in Fig.~\ref{Fig:Introduction}. The surfaces are composed of touching disks. The touching coins lemma (see \cite{BHS_2006}) ensures the existence of orthogonal spheres passing through their touching points.  
Without changes our method can be applied to construct cmc analogs of numerous discrete minimal surfaces found in \cite{BBuSe17}. For constructing one should simply replace the variational functional for circle patterns by the corresponding functional for ring patterns (see Section~\ref{sec:analytic_orp} and \cite{bobenko2004variational}), with the same convexity properties. Note that rotationally symmetric discrete cmc surfaces (discrete Delaunay surfaces) were previously constructed in \cite{bobenko2016s} by unrolling the billiard in an ellipse. The isometric deformation of discrete cmc surfaces, known as the associated family, was geometrically described in \cite{Ko18}. In particular the associated family of discrete Delaunay surfaces was constructed.  Discrete cmc surfaces have been applied in architectural geometry  for construction of doubly-curved building envelopes \cite{tellier2018discrete}.

The convergence of discrete minimal S-isothermic surfaces to smooth minimal surfaces was proven in \cite{BHS_2006}, see also a visualization \cite{BNeTe19}. The proof is based on Schramm's approximation result for circle patterns with the combinatorics of the square grid~\cite{schramm1997circle}. Later it was shown in \cite{lan2009c} that the convergence of minimal surfaces is in fact $C^\infty$, i.e., with all derivatives. Our numerical experiments with discrete cmc surfaces show an astonishingly good convergence, demonstrated in Fig.~\ref{Fig:Introduction}. It would be desirable to give a mathematical proof of this fact. The first step here is the approximation of (harmonic) Gauss maps by infinitesimal spherical orthogonal ring patterns shown in \cite{bobenko2024rings}. 

In Sections \ref{sec:Rto_discrete_s_isothermic_surfaces}-\ref{sec:Rto_discrete_cmc_surfaces_as_deformations_of_minimal_surfaces} we define and investigate spacelike discrete cmc surfaces in the Lorentz space $\R^{2,1}$. The considerations and results are quite similar to the Euclidean case. The main difference is that in this case the Gauss map $n$ connecting $s$ and $s^*$ forms a spacelike two spheres Koebe net, which can be identified with an orthogonal ring pattern in the hyperbolic plane. Unlike the spherical case, the corresponding variational functional is convex, see Section~\ref{sec:Rto_two_spheres_Koebe_q_nets_and_hyperbolic_orp} and \cite{bobenko2024rings}. This allows us to prove the existence and uniqueness of spacelike discrete cmc surfaces in $\R^{2,1}$ satisfying classical boundary conditions. Discrete maximal surfaces, which correspond to $H=0$ and to hyperbolic orthogonal circle patterns, are considered in Section~\ref{sec:Rto_discrete_cmc_surfaces_as_deformations_of_minimal_surfaces}.
Discrete spacelike maximal surfaces were recently linked to models in statistical mechanics \cite{ADMPS24a, ADMPS24b}.

It was shown in \cite{bobenko2024rings} that spherical and hyperbolic orthogonal ring patterns are described by the integrable Q4-equation \cite{AdSu04}. The latter is an equation with an elliptic spectral parameter, and is the master discrete integrable equation in the classification of \cite{AdBSu03}. This leads to an analytic description of discrete cmc surfaces in $\R^3$ and $\R^{2,1}$ in terms of loop groups. 
Known investigation methods of smooth cmc surfaces from an integrable point of view are mostly analytic. All integration methods, including the DPW factorization approach \cite{dorfmeister1998weierstrass}, algebro-geometric solutions \cite{Hi90, B91b}, are based on the loop group representations of the corresponding geometries. Although the DPW method was very productive in construction of new examples of cmc surfaces \cite{HeHeSch18}, precise mathematical proofs of their existence turned out to be a difficult issue. Recent progress here was achieved for n-noids and surfaces of high genus \cite{Tr20, HeHeTr23}. We do not explore the loop group description in the present paper and plan to address it in future publications. Here we expect a nice interplay of smooth and discrete theories enriching both.

Cmc-surfaces in $\R^3$ are isometric to minimal surfaces in $S^3$. This fact, known as the Lawson correspondence \cite{lawson1970complete}, can be lifted to a relation for the frames of the corresponding surfaces \cite{KGBPo97}. The Lawson correspondence is a powerful method to investigate embedded cmc-surfaces, by constructing the corresponding minimal surfaces in $S^3$, see for example \cite{GBKuSu03}. Analogously, cmc-surfaces in $\R^{2,1}$ are isometric to minimal surfaces in three-dimensional anti-de Sitter space $AdS_3$, which are popular models in the string theory.
It is a natural important problem to find a discrete Lawson correspondence, which should be a geometric construction of discrete minimal surfaces in $S^3$ and $AdS_3$ consisting of touching disks, from orthogonal ring patterns in the sphere and the hyperbolic plane.  

\section*{Acknowledgements} 
This research was supported by the DFG Collaborative Research Center TRR 109 ``Discretization in Geometry and Dynamics''.

\section{Discrete S-isothermic surfaces in $\Rt$}
\label{sec:Rt_discrete_s_isothermic_surfaces}

Motivated by the combinatorics of the curvature lines on smooth surfaces we consider quad graphs with interior vertices of even valence, typically four. We denote these graphs by $\G$ and their vertices, edges and faces by $V(\G), E(\G)$ and $F(\G)$ respectively. Recall that a quad graph is a quadrilateral cell decomposition of a surface, meaning all faces in $F(\G)$ are quadrilaterals. A special case is a cell complex $\G \subset \Z^2$ defined by a subset of elementary squares of the $\Z^2$ lattice in $\R^2$.

The edges of a graph $\G$ can be combinatorially divided into  'horizontal' and 'vertical' edges, which we consistently label by $i$ and $j$, see Figure \ref{Fig:graph_g}. We denote the vertices of an elementary quadrilateral of $\G$ by $v, v_i, v_{ij}$ and $v_j$, where $v$ is a vertex of $\G$, and $v_i$ and $v_j$ are its horizontal and vertical neighboring vertices, respectively. The vertex $v_{ij}$ is the fourth vertex of the elementary quadrilateral.

\begin{figure}[htbp]
	\centering
	\includegraphics[width=0.3\linewidth]{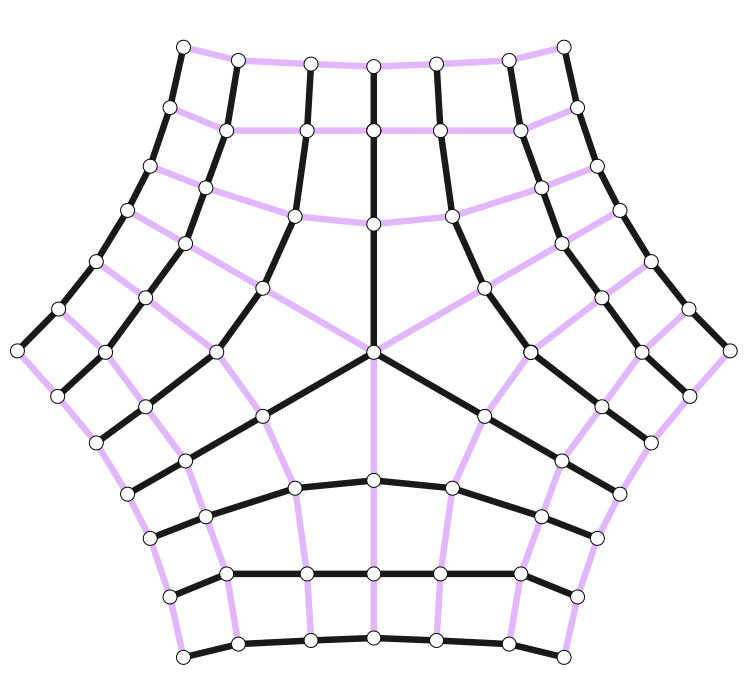}
	\caption{A quad graph $\G$ with interior vertices of even valency. The edges are divided into 'horizontal' $i$--edges (black) and 'vertical' $j$--edges (pink), such that in each quadrilateral the edges alternately change color.}
	\label{Fig:graph_g} 
\end{figure}
To proceed with the Möbius geometric characterization of S-isothermic nets, we recall some basic facts about Möbius geometry, see \cite{hertrich2003introduction} and \cite[Chapter 9]{bobenko2008discrete}.

Let $\R^{4, 1}$ be the five-dimensional Lorentz-Minkowski space equipped with the non-degenerate bilinear form 
\begin{align}
	\label{eq:Rt_msca}
\langle x, y \rangle_{4, 1} = x_1y_1 + x_2y_2 + x_3y_3 + x_4y_4 - x_5y_5.
\end{align}
 Let $e_1, ..., e_5$ denote the standard basis of $\R^{4,1}$ and let $e_0:=\frac{1}{2}(e_5-e_4)$ and $e_\infty:=\frac{1}{2}(e_4+e_5)$.

Points in $\R^3 \cup \{ \infty\}$ can be identified with points in the \emph{light cone} $$\hat{\mathbb{S}}_0 := \{\hat{x} \in \Rfo \ | \ \langle \hat{x}, \hat{x} \rangle_{4, 1} = 0\}$$ via the identification 
\begin{align}
\label{eq:Rt_point_lift}
\Rt \ni x \leftrightarrow & \ \hat{x} = x + e_0 + ||x||^2e_\infty \in \hat{\mathbb{S}}_0.
\end{align}
In this context, $x$ on the right hand side is understood as $x_1e_1 + x_2 e_2 + x_3 e_3 \in \Rfo$, and points of the form $\hat{x} = x + e_0 + ||x||^2e_\infty \in \hat{\mathbb{S}}_0$ are normalized such that $\langle \hat{x}, e_\infty \rangle_{4, 1} = - \frac{1}{2}$.
The point $\infty \in \Rt \cup \{ \infty\}$  is identified with $e_\infty$.

If $\Rfo$ is interpreted as the space of homogeneous coordinates of $\R P ^4$, points in $\R^3 \cup \{ \infty\}$  can be identified with the stereographic projection of points on
$$\mathcal{M} := \{ [\hat{x}] \in \R P ^4 | \langle \hat{x}, \hat{x} \rangle_{4, 1} = 0 \}, $$
which is called the \emph{M\"obius quadric}. A point of the form \eqref{eq:Rt_point_lift} in $\hat{\mathbb{S}}_0$ represents a special choice of homogeneous coordinates for the corresponding projective point in $\mathcal{M}$.
Similarly, points outside the Möbius quadric, i.e., 
$$[\hat{x}] \in \mathcal{M}_+ := \{ [\hat{x}] \in \R P ^4 | \langle \hat{x}, \hat{x} \rangle_{4, 1}  > 0 \}, $$
can be identified with non-oriented spheres in $\Rt$. Hyperplanes are considered as spheres with infinite radius. 
 A sphere with center $c$ and radius $d$ is represented by the projective point
 \begin{align*}
 	[\hat{s}] = [c + e_0 + (||c||^2-d^2)e_\infty] \in \mathcal{M}_+.
 \end{align*}
Its homogeneous coordinates can be normalized to
\begin{align*}
	%\label{eq:Rt_sphere_lift}
	%\Rt \supseteq  s_{c, d}  \leftrightarrow & \ 
	\hat{s} = \frac{1}{d} \left( c + e_0 + (||c||^2-d^2)e_\infty \right)
\end{align*} on the Lorentz unit sphere
	\begin{align}
		\label{eq:Rfo_unit_sphere}
		\hat{\mathbb{S}}_1 := \{\hat{x} \in \Rfo \ | \ \langle \hat{x}, \hat{x} \rangle_{4, 1} = 1\} \subset \Rfo.
	\end{align}
The sign of the radius $d$ encodes the orientation of the sphere.
Note that the spheres with different orientations $\pm d$ correspond to different points on $\hat{\mathbb{S}}_1$, but are represented by the same point in $\mathcal{M}_+$.

 General S-isothermic surfaces were introduced in \cite{bobenko2008discrete} as two-dimensional T-nets in $\Rfo$.

\begin{definition}
	\label{def:Rt_s_isothermic}
	A map
	$$s: V(\G) \rightarrow  \{\text{oriented spheres in } \Rt \}$$
	is called a \emph{discrete S-isothermic surface}
	if the corresponding map  ${\hat{s}: V(\G) \rightarrow \hat{\mathbb{S}}_1 \subset \Rfo}$, that maps oriented spheres to 
	$$\hat{s}=\frac{1}{d}\left( c + e_0 +(||c||^2-d^2)e_\infty\right),$$
	satisfies the discrete Moutard equation 
	\begin{align*}	%	\label{eq:Moutard_s_iso}
		\hat{s}_{v_{ij}} - \hat{s}_v  = a_{ij}(\hat{s}_{v_j} - \hat{s}_{v_i}).
	\end{align*}
	for some $a_{ij}: F(\G) \rightarrow \R$. 
\end{definition}

\begin{remark} Instead of considering the Lorentz unit sphere \eqref{eq:Rfo_unit_sphere} a sphere centered at the origin with an arbitrary radius 
$$\hat{\mathbb{S}}_\kappa := \{\hat{x} \in \Rfo \ | \ \langle \hat{x}, \hat{x} \rangle_{4, 1} = \kappa\}$$
can be used to identify its points with oriented spheres in $\Rt$. In this case the scaling factor of the homogeneous lift of a sphere $\hat{s}$ is given by  $\frac{\kappa}{r}$. In the limit $\kappa \rightarrow 0$ one obtains the original discrete isothermic surfaces introduced in \cite{bobenko1996discrete}. Since we do not consider this convergence in the present paper we restrict ourselves to the fixed value $\kappa=1$. \\
\end{remark}

Note that $\hat{s}$ is a \emph{Q-net}, i.e., the vertices $\hat{s}_v, \hat{s}_{v_i}, 
\hat{s}_{v_{ij}}, \hat{s}_{v_j}$ are coplanar. 
The elementary quadrilaterals of S-isothermic surfaces are called \emph{S-isothermic quadrilaterals}. 

\begin{corollary}
	\label{Cor:Rt_Q_congruences}
	Let $s_v, s_{v_i}, s_{v_{ij}}$ and $s_{v_j}$ be four the spheres of an S-isothermic quadrilateral. Then 
	\begin{equation}
		\label{eq:Rt_labelling_property}
		\begin{aligned} 
			\langle \hat{s}_v,  \hat{s}_{v_i} \rangle_{4, 1}  &= \langle \hat{s}_{v_j}, \hat{s}_{v_{ij}} \rangle_{4, 1} =:\alpha_i\\
			\langle \hat{s}_v, \hat{s}_{v_j} \rangle_{4, 1}  &= \langle \hat{s}_{v_i}, \hat{s}_{v_{ij}} \rangle_{4, 1} =:\alpha_j.
		\end{aligned}
	\end{equation}
Further they satisfy one of the three conditions shown in Figure \ref{Fig:Rt_Q_congruences}, namely
	\begin{enumerate}[(i)]
		\item They have a common orthogonal circle;
		\item They intersect in exactly two points;
		\item They intersect in exactly one point.
	\end{enumerate}
\end{corollary}
\begin{figure}
	\center
	\begin{overpic}[width=\linewidth]
		{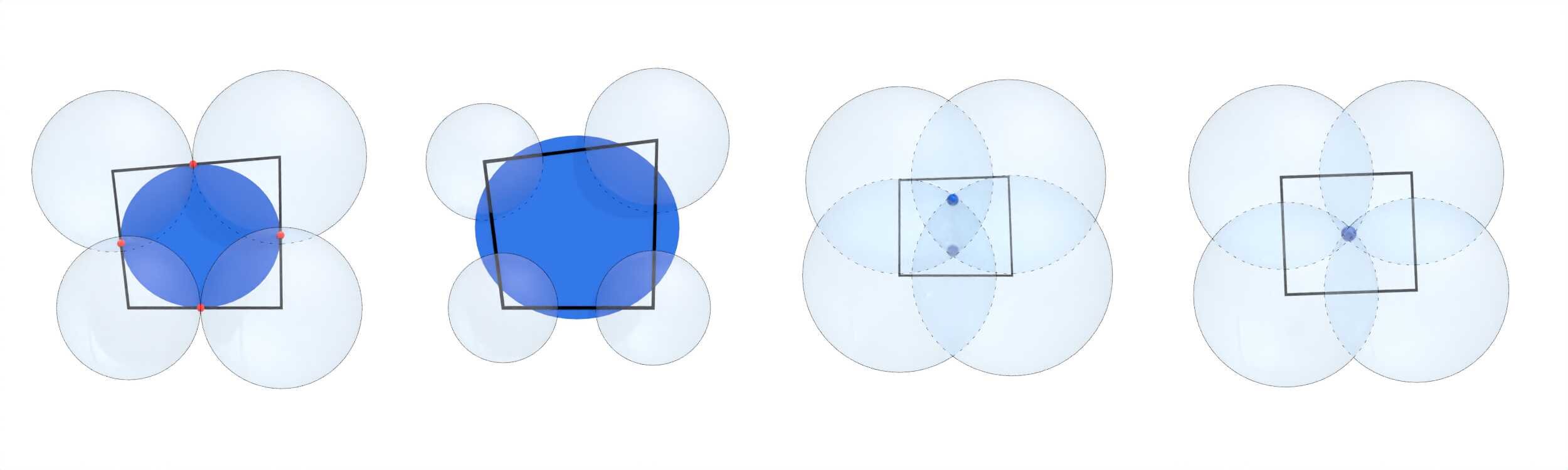}
		\put(5,2.5){\small touching case}
		\put(35.5,2.5){\small $(i)$}
		\put(59.5,2.5){\small $(ii)$}
		\put(84.5,2.5){\small $(iii)$}
	\end{overpic}
	\caption{Three types of S-isothermic quadrilaterals with vertex spheres $(i)$ possessing a common orthogonal circle, $(ii)$ intersecting in two points, or $(iii)$ intersecting in one point. The first image shows the special case of type $(i)$ where vertex spheres touch.}
	\label{Fig:Rt_Q_congruences}
\end{figure}
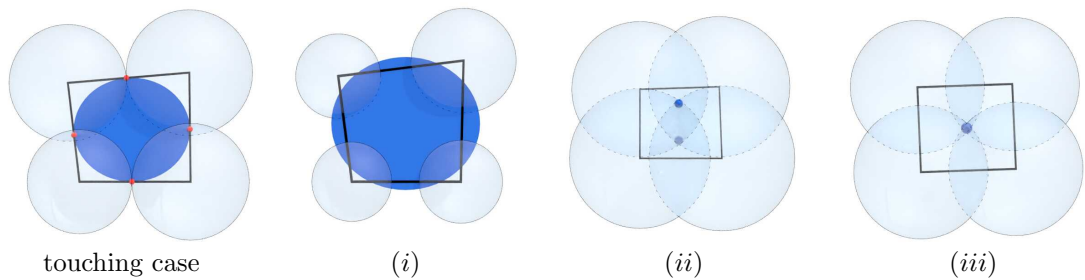
\begin{proof}
	As a Moutard net in the quadric $\hat{\mathbb{S}}_1$, the map $\hat{s}$ admits the labeling property  \eqref{eq:Rt_labelling_property} \cite{bobenko2008discrete}. The planarity of the faces of $\hat{s}$ corresponds to the spheres around a face in $\Rt$ satisfying one of the conditions $(i), (ii)$ or $(iii)$. Let us consider the three-dimensional linear subspaces in $\Rfo$ containing a face of $\hat{s}$. The induced metric on its orthogonal complements is either of signature $(2, 0)$, $(1, 1)$ or degenerate. These three cases correspond to the conditions $(i), (ii)$ and $(iii)$ respectively. For more details see \cite{bobenko2008discrete}. 
\end{proof}

\noindent S-isothermic surfaces are discrete K{\oe}nigs nets and invariant under Möbius transformations. 
The edge labels $\alpha_i$ and $\alpha_j$ have the meaning of cosines of the intersection angles of the neighboring spheres, resp. of their so-called inversive distance, if they do not intersect.

 If all quadrilaterals of an S-isothermic surface $s$ are of the same type, $s$ is said to be of type $(i)$, $(ii)$, or $(iii)$.
A special case of type $(i)$, introduced in \cite{bobenko1999discretization}, are S-isothermic surfaces where neighboring vertex spheres touch, called \emph{S$_1$-isothermic surfaces}, see Figure \ref{Fig:Rt_Q_congruences} (left).
In this case the orthogonal circles form inscribed circles of the quadrilaterals and neighboring circles touch in the same points as the spheres, see Figure  \ref{Fig:Rt_S_isothermic}. For the edge labels \eqref{eq:Rt_labelling_property} of an S$_1$-isothermic surface it is $\alpha_i, \alpha_j = \pm 1$. To ensure that  $\alpha_i \neq \alpha_j$ for each quadrilateral, one chooses an orientation of spheres such that neighboring spheres are in non-oriented contact along horizontal edges and in oriented contact along vertical edges, 
\begin{equation*}
%\label{eq:Rt_sphere_orientation}
\begin{aligned} 
\alpha_i &=\langle \hat{s}_v,  \hat{s}_{v_i} \rangle_{4, 1}  = \langle \hat{s}_{v_j}, \hat{s}_{v_{ij}} \rangle_{4, 1} =  -1\\
\alpha_j &= \langle \hat{s}_v, \hat{s}_{v_j} \rangle_{4, 1}  = \langle \hat{s}_{v_i}, \hat{s}_{v_{ij}} \rangle_{4, 1} =  1.
\end{aligned}
 \end{equation*}
Let $$c: V(\G) \rightarrow \Rt, \ v \mapsto c_v$$ denote the center net of the spheres of an S-isothermic net and let $d: V(\G) \rightarrow \R$ denote the corresponding signed radii.
Any S-isothermic surface possesses a dual S-isothermic surface \cite{bobenko2008discrete}.

\begin{figure}
	\centering
		\includegraphics[width=.6\linewidth]{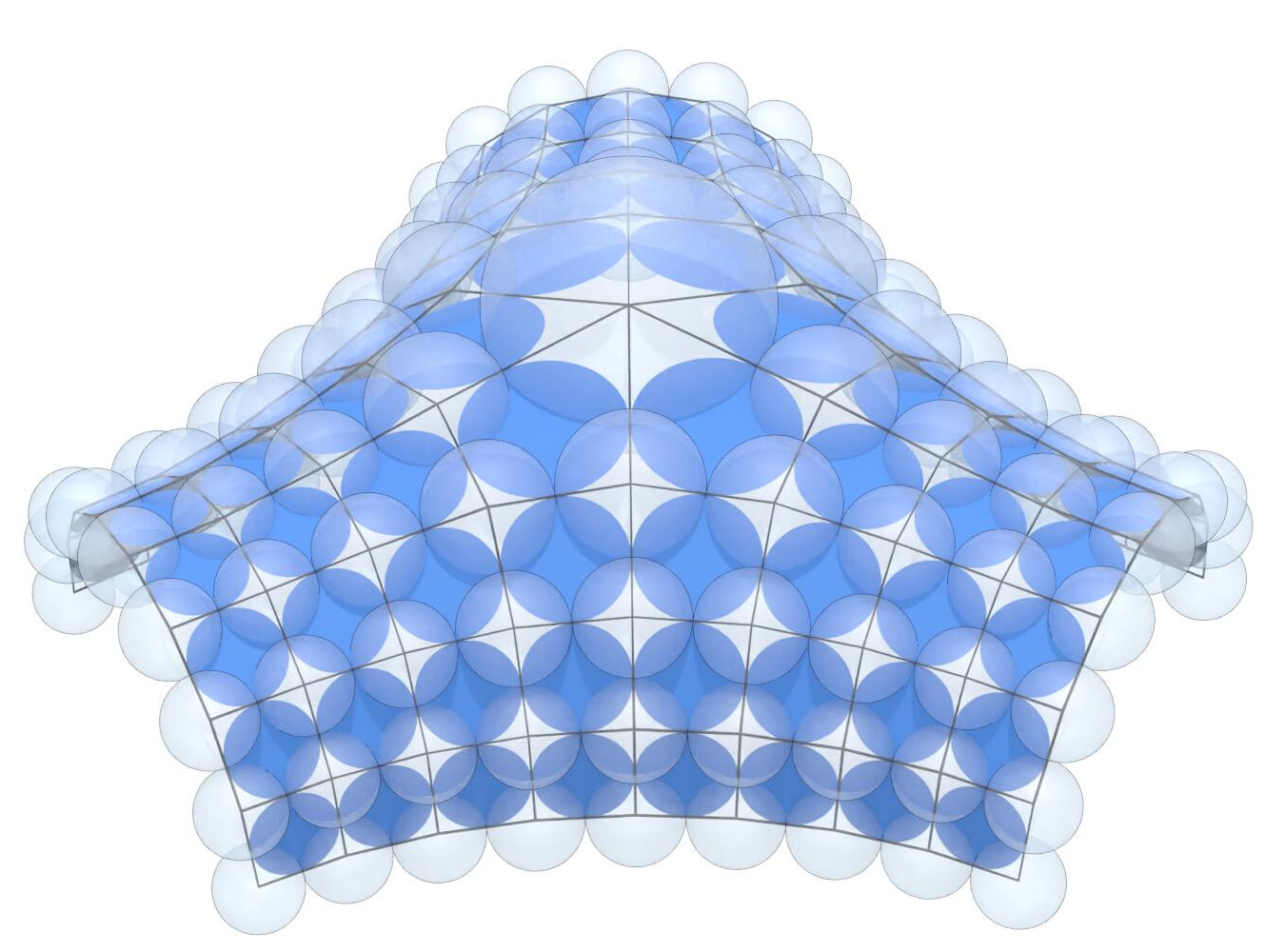}
	\caption{An S$_1$-isothermic surface. Each quad possesses an incircle that intersects adjacent spheres orthogonally and passes through  their touching points. The one large vertex sphere corresponds to an umbilic point, an inner vertex of $\mathcal{G}$ with valency greater than four. The figure shows a piece of the doubly periodic S$_1$-isothermic cmc surface $\psi(U_{2, 2})$  presented in Section \ref{sec:Rt_examples}.}
	\label{Fig:Rt_S_isothermic}
\end{figure}

\begin{lemma}
	\label{Def:Rt_Christoffel}
	Let $s$
	be an S-isothermic surface. Then the $\Rt$-valued discrete one-form $\partial c^*$ defined by
	\begin{align}
	\label{eq:Rt_one_form}
	\partial_{(v, v')}c^* = \frac{\partial_{(v, v')} c}{d_v d_{v'}},
	\end{align}
	with $(v, v')\in E(\G)$, is exact. Its integration defines (up to a translation) a function $$c^* : V(\G) \rightarrow \Rt, \ v \mapsto c^*_v.$$ 
	Define also $$d^* : V(\G) \rightarrow \R, \ v \mapsto d^*_v =\frac{1}{d_v}.$$ Then the map 
	\begin{align*}
	s^*: V(\G) \rightarrow \{\text{oriented spheres in } \Rt \}
	\end{align*} with spheres with 
	centers $c^*$ and radii $d^*$ is an S-isothermic surface, called \emph{Christoffel dual} to $s$.
\end{lemma}

 With the chosen orientation of spheres, the definition of the one-form (\ref{eq:Rt_one_form}) implies that horizontal edges maintain their orientation, while vertical edges reverse their orientation in the Christoffel dual.

\begin{corollary}
	Christoffel duality preserves edge labels. In particular, it preserves the class of S$_1$-isothermic surfaces.
\end{corollary}

When extending the combinatorics of an S$_1$-isothermic surface by introducing its face circle centers and its touching points as new vertices (referred to as the \emph{central extension}), one obtains a related discrete isothermic surface \cite{bobenko1999discretization}. In this extension, all faces of the isothermic surface form kites with one  pair of opposite right angles. Dualizing the isothermic surface yields a dual isothermic surface that is itself the central extension of the Christoffel dual of the S$_1$-isothermic surface. The following corollary is a direct consequence when considering the dual of the central extension.

\begin{corollary}
	The radii $d_f$ and $d_f^*$ of the orthogonal circles of an S$_1$-isothermic surface and its Christoffel dual are related by \begin{align*}
	d_f^* = \frac{1}{d_f}.
	\end{align*}
\end{corollary}

 Sometimes it is more convenient to define the Christoffel up to scaling (and translation). In general, there exists an arbitrary global constant $\lambda \in \mathbb{R}$ such that 
\begin{equation}
	\label{eq:scaling_christoffel}
	d_vd_v^* = d_f d_f^* = \lambda. 
\end{equation}
In Definition \ref{Def:Rt_Christoffel} we have fixed the scaling by choosing $\lambda = 1$.

\section{Discrete cmc surfaces in $\Rt$}
\label{sec:Rt_discrete_s_isothermic_cmc_surfaces}

S-isothermic nets belong to integrable discrete differential geometry \cite{bobenko2008discrete}. S-isothermic nets with underlying $\Z^2$ combinatorics can be extended consistently to the $\mathbb{Z}^n$ lattice so that all its two-dimensional coordinate subnets are S-isothermic. The generalization to the combinatorics of $\mathcal{G}$ is straight forward. This property can be interpreted
as a transformation of two-dimensional S-isothermic nets called \emph{Darboux transform} \mbox{\cite{hoffmann2010darboux, bobenko2008discrete}}. 

\begin{definition}
	\label{Def:Rt_Darboux}
	Two S-isothermic surfaces $s$ and $s^+$ are called a \emph{Darboux pair} if the corresponding  Moutard nets $\hat{s}, \hat{s}^+ : V(\G) \rightarrow \Rfo $ are related by a Moutard transformation. In particular the transformation faces $(\hat{s}_v, \hat{s}_{v'}, \hat{s}^+_{v'}, \hat{s}^+_{v})$ fulfill the Moutard equation 
	\begin{align*}
		\hat{s}^+_{v'}- \hat{s}_v  = a_+(\hat{s}^+_{v} - \hat{s}_{v'})
	\end{align*}
	with some $a_+: E(\G) \rightarrow \R$. 
In this case, one surface is called a \emph{Darboux transform} of the other.
\end{definition}
A Darboux transform $s^+$ of an S-isothermic net $s$ is uniquely determined by the choice of one of the spheres of $s^+$. A Darboux transform has an associated constant parameter $\alpha$, which arises from the edge labeling property of the Moutard equation:
\begin{align}
	\label{eq:alpha}
	-2\alpha  := -2\langle \hat{s}_v,  \hat{s}_{v}^+ \rangle_{4, 1} = ||c_v - c_v^+||^2 - \left( d_v^2 + {d_v^+}^2\right).
\end{align}

\begin{corollary}
	\label{Cor:Rt_Darboux}
	Darboux transformations preserve edge labels. 
	In particular, they preserve the class of S$_1$-isothermic surfaces.
\end{corollary}

Smooth cmc surfaces can be characterized through their Darboux and Christoffel transformations \cite{hertrich1997remarks}. 
\begin{theorem}
	\label{Thm:Hetrich_Pedit}
	A smooth isothermic surface $f$ is  a surface of constant mean curvature $H\neq 0$ if and only if the (correctly scaled and positioned) Christoffel transform $f^*$ is also a Darboux transform. In this case, $f^*$ is the parallel surface at distance $\frac{1}{H}$ .
\end{theorem}

This characteristic property was used to define discrete isothermic surfaces of constant mean curvature \cite{hertrich1999discrete} and S-isothermic cmc surfaces \cite{hoffmann2010darboux}.

\begin{definition}
	\label{Def:Rt_cmc}
	An S-isothermic surface $s$ is an \emph{S-isothermic surface of constant mean curvature}, or \emph{S-cmc surface} for short, if its Christoffel dual $s^*$ is simultaneously a Darboux transform (after appropriate scaling and translation).
\end{definition}

The pair $s, s^*$ of an S-cmc surface and its Christoffel dual is called an \emph{S-cmc pair}.

\subsection{The Gauss map of S$_1$-cmc surfaces}
\label{sec:Gauss_map}

We now restrict our considerations of S-cmc surfaces to the case of touching spheres, i.e., to S$_1$-cmc surfaces.

The Gauss map of a smooth surface is a map to the unit sphere $S^2$. Natural discretizations of Gauss maps  are polyhedra
 with vertices on $S^2$, faces tangent to $S^2$, or edges tangent to $S^2$.
 These discretizations lead to circular surfaces \cite{bobenko2010curvature}, conical surfaces \cite{liu2006geometric, pottmann2008focal}, and surfaces of Koebe type \cite{BHS_2006}, respectively. Polyhedra with all edges touching $S^2$ are called \emph{Koebe Polyhedra}, or \emph{Koebe nets} when considering only local pieces. They are used to describe the Gauss map of S$_1$-minimal surfaces \cite{BHS_2006}. A generalization of Koebe polyhedra to the case when all edges touch a general convex body was considered in \cite{schramm1992cage}. 
 
 For our purposes, we propose another generalization of Koebe nets by introducing two spheres concentric with the unit sphere and by considering nets that are alternately tangent to these two spheres, see Figure \ref{Fig:Rt_Koebe}.

\begin{definition}
	\label{def:Rt_two-sphere_Koebe}
	A Q-net $k:V(\G) \rightarrow \Rt$ is called a \emph{two-sphere Koebe net} if its edges alternately touch two spheres $S_+^2$ and $S_-^2$ (concentric with the unit sphere $S^2$) whose radii satisfy the relation $r_+r_- = 1$.
\end{definition}

The latter condition ensures that the two spheres $S_+^2$ and $S_-^2$ and the discrete net $k$ remain close to the unit sphere. Indeed, $S_-^2$ can be obtained by a sphere inversion of $S_+^2$ in $S^2$, and vice versa.

\begin{theorem}
	\label{Rt:Thm_s_cmc_koebe}
	Let $s$ and $s^*$ be a (suitably scaled) S$_1$-cmc pair.
	The Gauss map  \begin{align}
		\label{eq:Rt_vertex_normals}
		n: V(\G) \rightarrow \Rt,\ v \mapsto n_v := c^*_v - c_v, 
	\end{align}
between the sphere centers of $s$ and $s^*$, forms a two-sphere Koebe net.
\end{theorem}
\begin{figure}[bp]
	\centering
	
	\begin{overpic}[scale=.13]
		{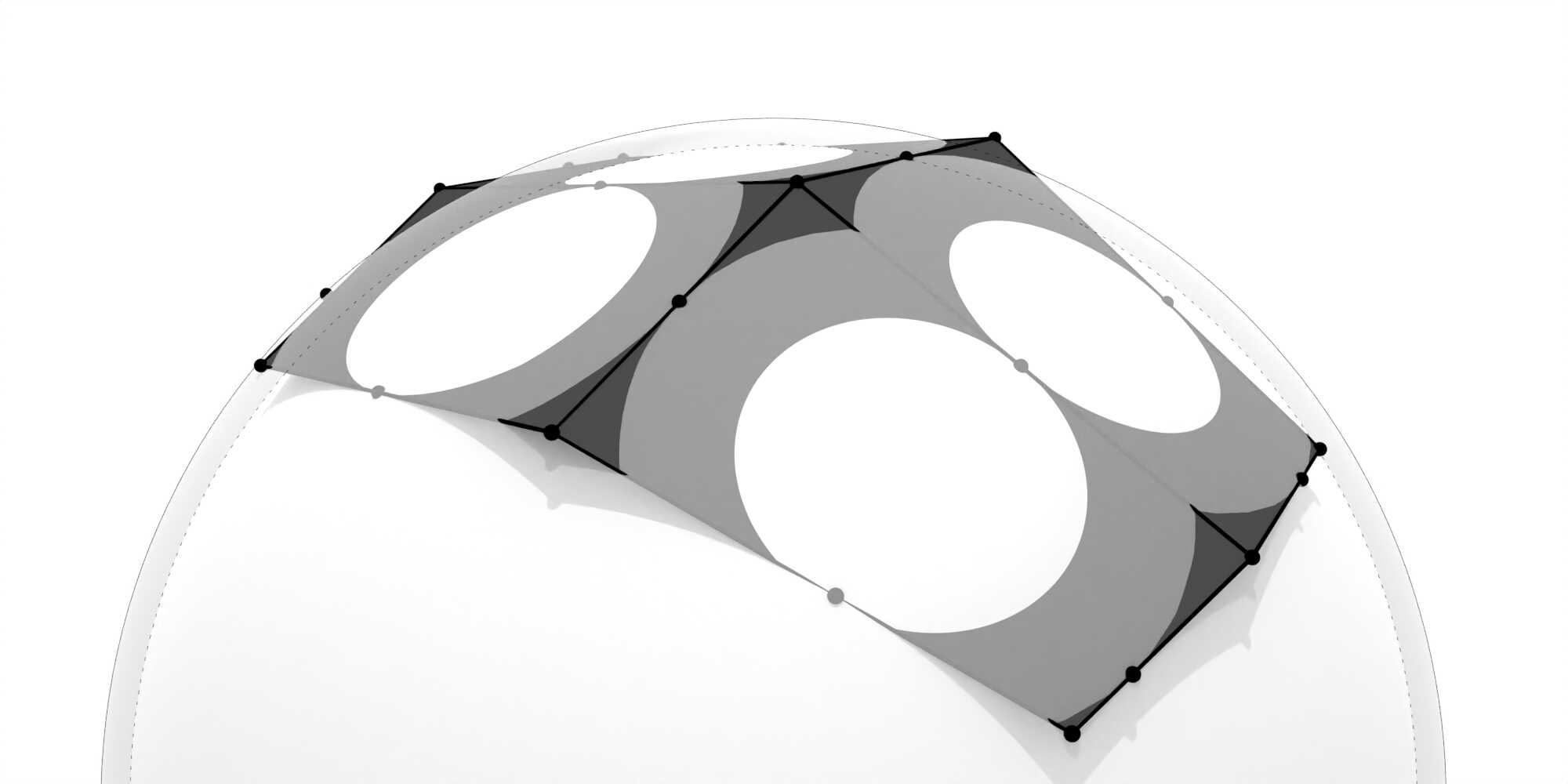}
		
		\put(92,1){ $S_+^2$}
		\put(83,1){ $S_-^2$}
	\end{overpic}
	\caption{A two-sphere Koebe net. Its edges alternately touch the larger, transparent gray sphere  $S_+^2$ and the smaller, white sphere $S_-^2$.}
	\label{Fig:Rt_Koebe}
\end{figure}
Before proving this theorem, we will investigate the properties of edge and face normals of an S$_1$-cmc pair. To that end, let $t, t^*: E(\G) \rightarrow \Rt$ denote the points of contact between neighboring spheres on primal and dual edges, as illustrated in the Figures \ref{Fig:Rt_side_faces} and \ref{Fig:Rt_fundamental_hex}.

\begin{proposition}
	\label{Prop:Rt_side_faces}
	The \emph{edge normals}
	\begin{align}
	\label{eq:edge_normals}
	l: E(\G) \rightarrow \Rt,\ (v, v') \mapsto l_{(v, v')} := t^*_{(v, v')} - t_{(v, v')}, 
	\end{align}
	which connect the points of contact on primal and corresponding Christoffel dual edges, are orthogonal to both primal and dual edges,
	\begin{equation}
		\label{eq:edge_normals_orthogonality}
	\begin{aligned}
	l_{(v, v')} \perp c_{v'}-c_v \quad 
	\text{ and } \quad 
	l_{(v, v')} \perp c^*_{v'}-c^*_v. \\
	\end{aligned}
	\end{equation}
	Additionally, the edge normals have constant length:
	\begin{align}
		\label{eq:edge_normals_length}
		||l_{(v, v')}||^2 &= 
		\begin{cases}
			-2\alpha +2\lambda  & \text{ for embedded faces}\\
			-2\alpha -2\lambda & \text{ for non-embedded faces}.
		\end{cases}
	\end{align} Here $\alpha$ denotes the parameter of the Darboux transform \eqref{eq:alpha} and $\lambda$ represents the global constant \eqref{eq:scaling_christoffel}.
\end{proposition}

\begin{figure}
	\centering
	\begin{minipage}{.383\linewidth}
		\centering
		\begin{overpic}[width=\linewidth]
			{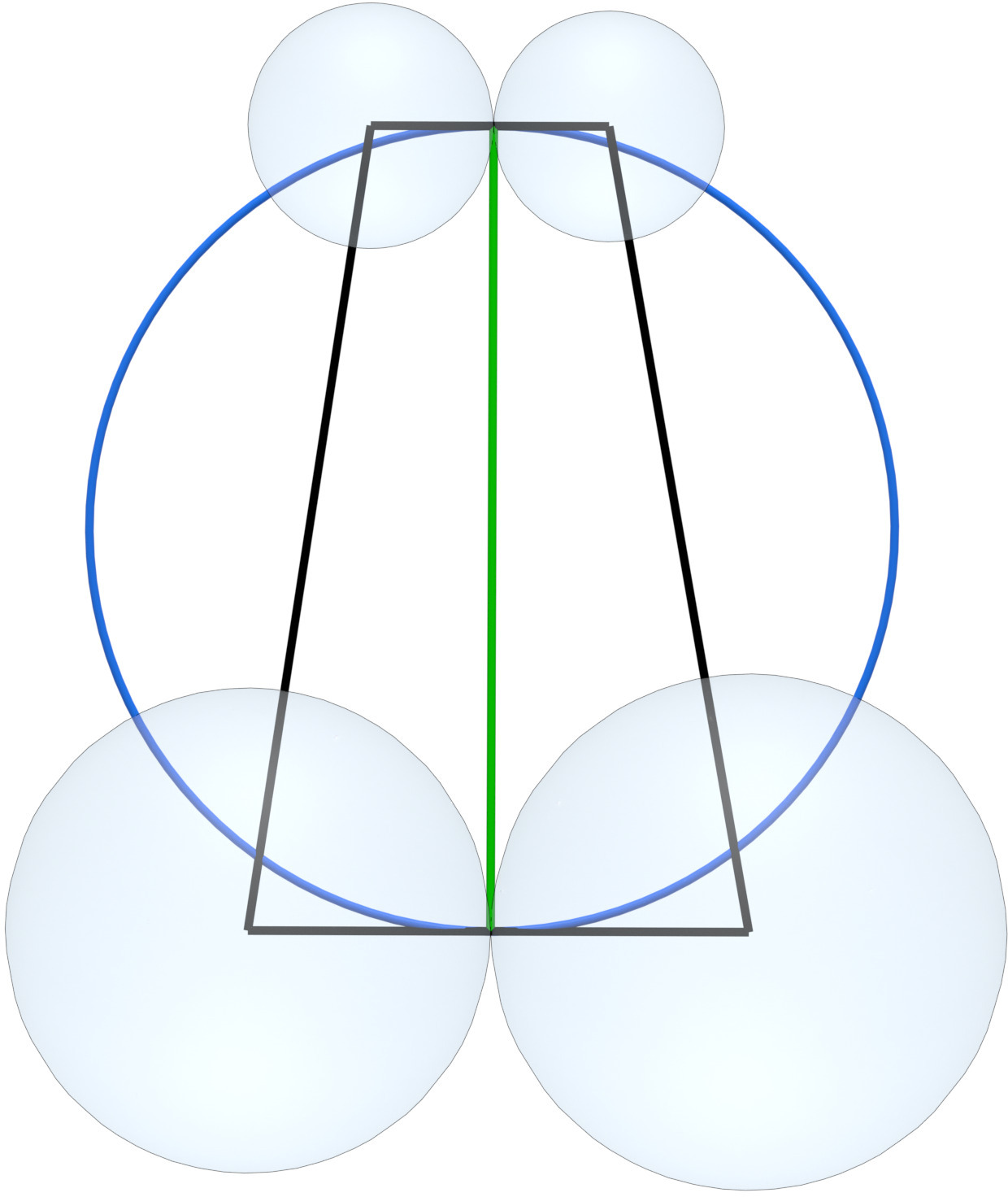}
			
			\put(25,91){$c_v^*$}
			\put(52,91){$c_{v_i}^*$}
			\put(37,93){$t_{(v, v_i)}^*$}
			
			\put(19,65){$n_v$}
			\put(57,65){$n_{v_i}$}
			\put(42,46){$l_{(v, v_i)}$}
			
			\put(17,17){$c_v$}
			\put(64,17){$c_{v_i}$}
			\put(37,17){$t_{(v, v_i)}$}
		\end{overpic}
	\end{minipage}
	\hspace{1.5cm}
	\begin{minipage}{.4\linewidth}
		\centering
		\begin{overpic}[width=\linewidth]
			{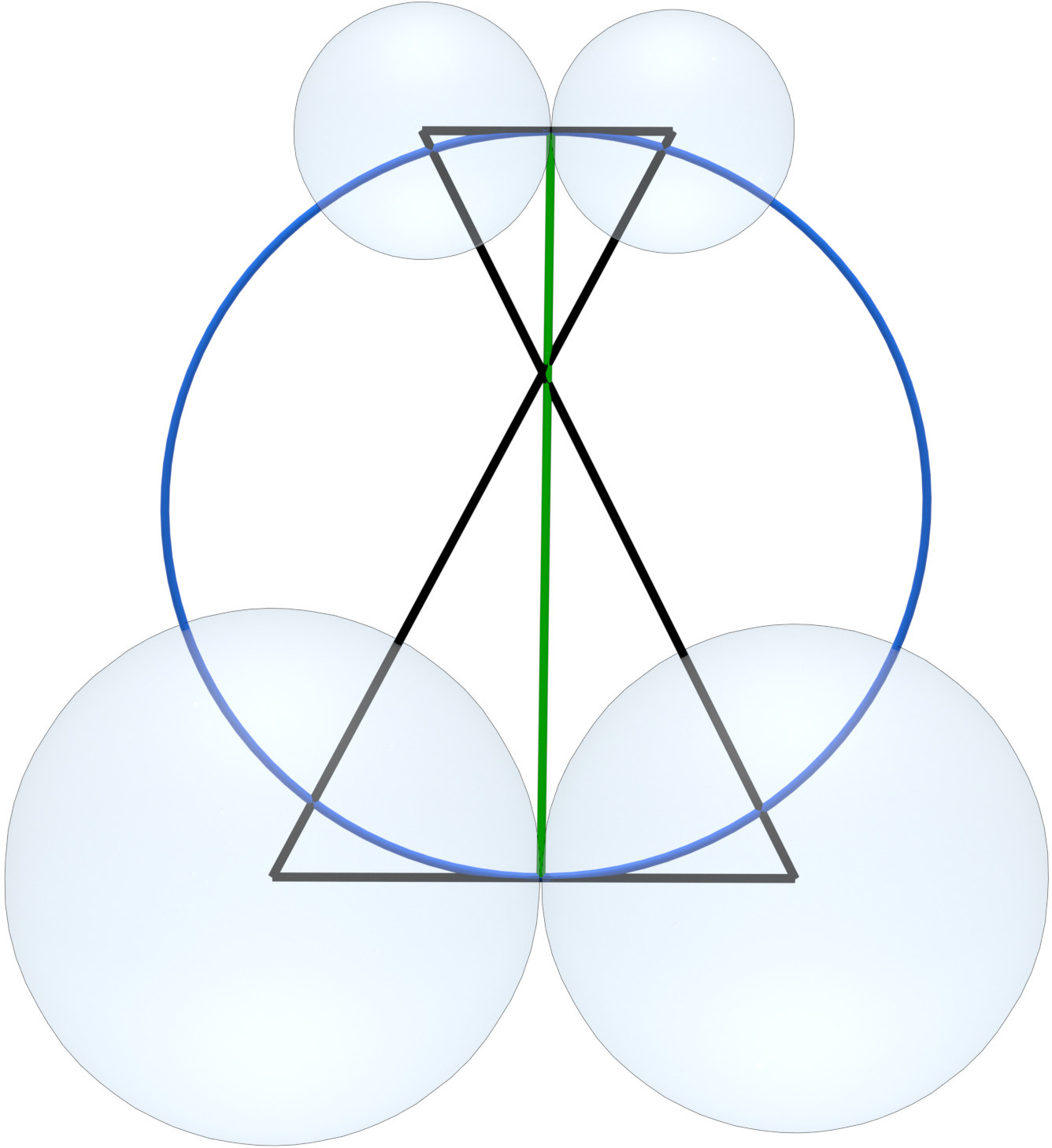}
			
			\put(30,91){$c_{v_j}^*$}
			\put(60,90){$c_v^*$}
			\put(42,93){$t_{(v, v_j)}^*$}
			
			\put(34,59){$n_v$}
			\put(54,59){$n_{v_j}$}
			\put(45,44){$l_{(v, v_j)}$}
			
			\put(19,19){$c_v$}
			\put(70,19){$c_{v_j}$}
			\put(42,19){$t_{(v, v_j)}$}
		\end{overpic}
	\end{minipage}
	\caption{The two types of transformation faces of an S$_1$-cmc pair together with vertex normals (black), edge normals (green) and the circle orthogonal to the four adjacent vertex spheres (blue). Embedded faces correspond to horizontal edges, non-embedded faces correspond to vertical edges.}
	\label{Fig:Rt_side_faces}
\end{figure}
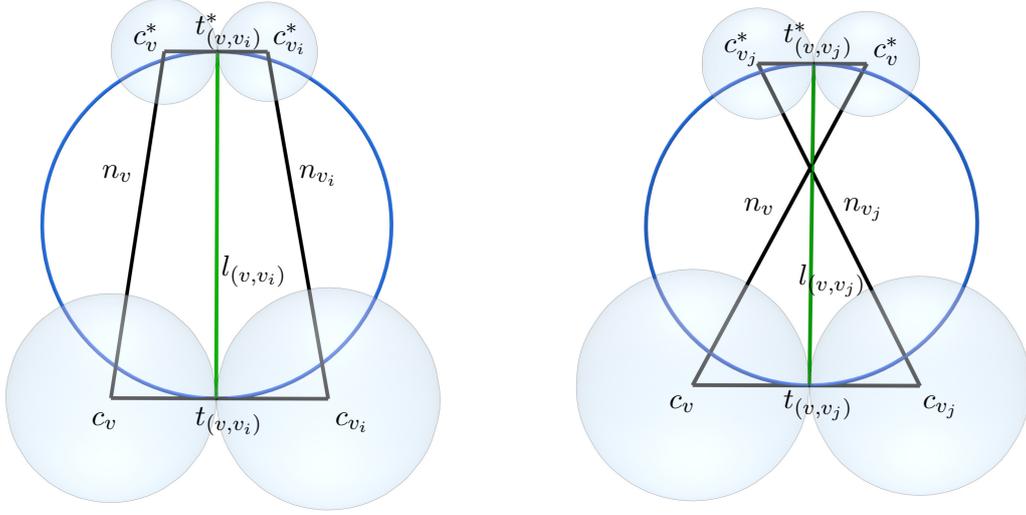

\begin{proof}
	The transformation faces of an S$_1$-cmc pair form S-isothermic trapezoids: S-isothermic quadrilaterals with an orthogonal circle and one pair of parallel edges, see Figure \ref{Fig:Rt_side_faces}.
	The parallel edges, $c_{v'}-c_v$ and $c^*_{v'}-c^*_v$, are tangent to the orthogonal circle, with the points of tangency at $t_{(v, v')}$ and $t^*_{(v, v')}$. This implies \eqref{eq:edge_normals_orthogonality}.
	
	To prove the second part of the proposition, we observe that the edge normals are also tangent to both the primal and dual spheres at the points $t_{(v, v')}$ and $t^*_{(v, v')}$. Using the orthogonality \eqref{eq:edge_normals_orthogonality}, we can compute the squared length of the edge normals as follows:
	\begin{align*}
		&||c_{v_i}-c_{v_i}^*||^2 - \left( d_{v_i} - {d^{*}_{v_i}} \right)^2 = ||c_{v_i}-c_{v_i}^*||^2 - \left(d_{v_i}^2 + {d_{v_i}^{*2}} \right) + 2\lambda, \\
		&||c_{v_j}-c_{v_j}^*||^2  - \left(d_{v_j} + {d^{*}_{v_j}}\right)^2 = ||c_{v_j}-c_{v_j}^*||^2 - \left(d_{v_j}^2 + {d_{v_j}^{*2}}\right) - 2\lambda,
	\end{align*}
	for embedded and non-embedded transformation faces, respectively.
	A comparison with \eqref{eq:alpha} leads to \eqref{eq:edge_normals_length}.
\end{proof}

\begin{figure}
	\begin{overpic}[width=.35\linewidth,]
		{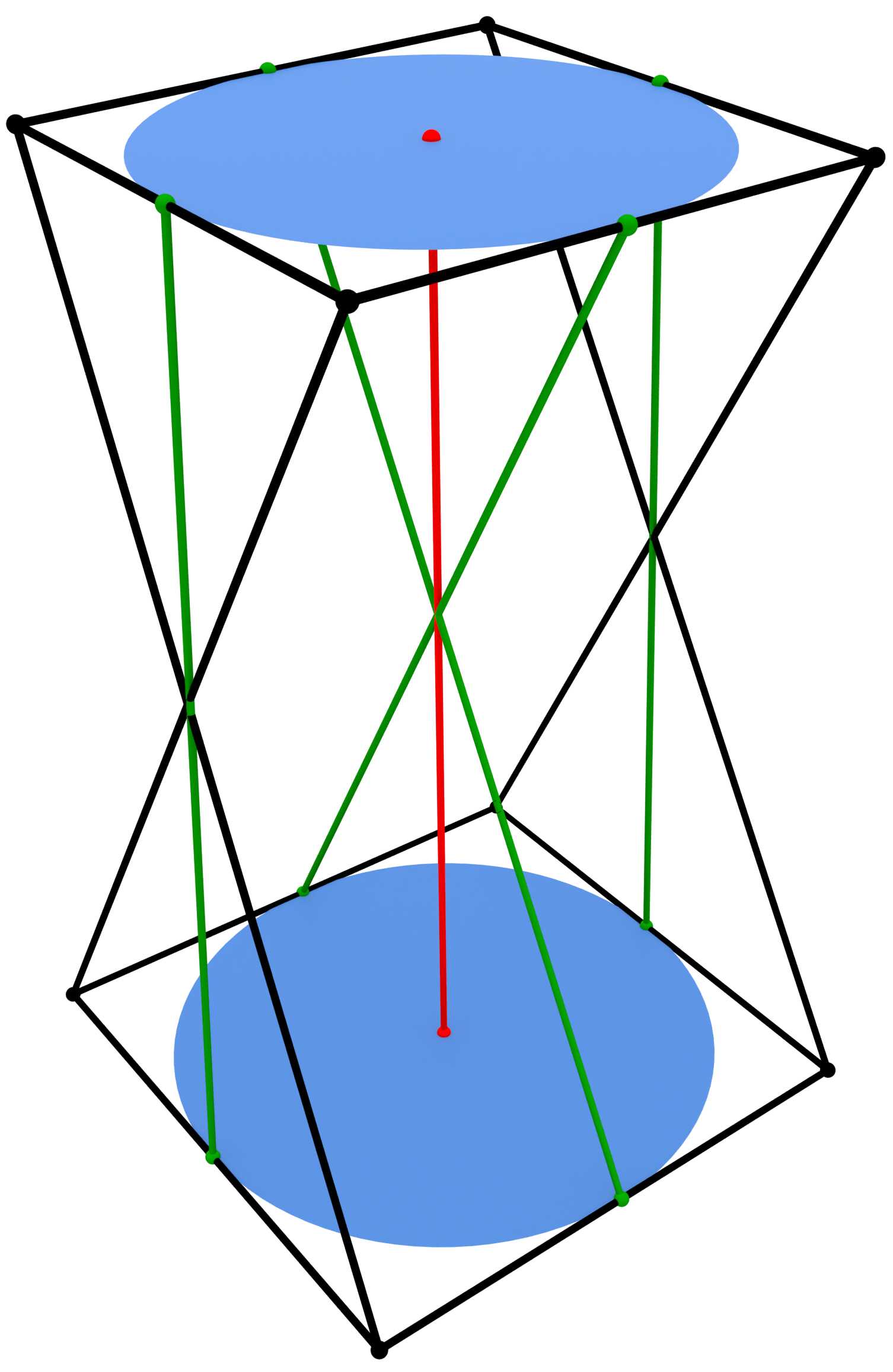}
		\put(22,0){$c_v$}
		\put(61.5,22){$c_{v_i}$}
		\put(-2,28){$c_{v_j}$}
		\put(38,40){$c_{v_{ij}}$}
		\put(33,22){$c_{f}$}
		
		\put(-4,91){$c^*_v$}
		\put(35,100){$c^*_{v_i}$}
		\put(18,76){$c^*_{v_j}$}
		\put(65,88){$c^*_{v_{ij}}$}
		\put(13, 98.5){$t^*_{(v, v_i)}$}
		\put(13, 86.5){$t^*_{(v, v_j)}$}
		\put(33,88){$c^*_{f}$}

		\put(46, 09){$t_{(v, v_i)}$}
		\put(03, 15){$t_{(v, v_j)}$}
	\end{overpic}
	\caption{A fundamental hexahedron of an S$_1$-cmc pair, formed by a pair of primal and Christoffel dual faces and corresponding vertex normals (black). The orthogonal face circles, shown in blue, touch the faces at the points $t_{(v, v')}$ and $t^*_{(v, v')}$. The edge normals, shown in green, connect these points on corresponding primal and dual edges. The face normal (red) connects the centers of primal and dual orthogonal circles and is orthogonal to both faces.}
	\label{Fig:Rt_fundamental_hex}
\end{figure}
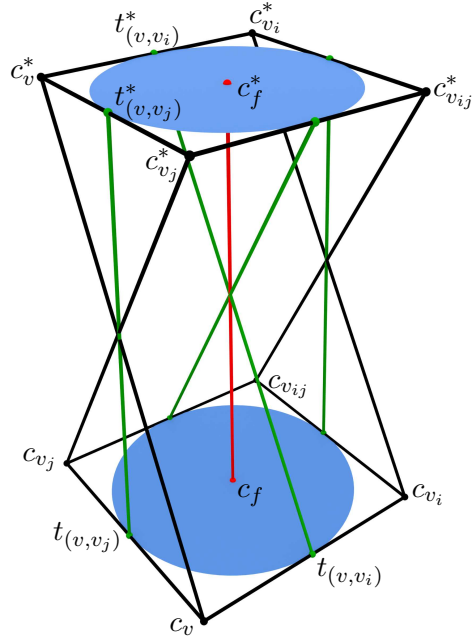

\begin{proposition}
	The axes of primal and Christoffel dual orthogonal circles coincide.
	\label{Prop:Rt_face_normals}
	The axes are given by the \emph{face normals}
	\begin{align}
	\label{eq:face_normals}
	m: F(\G) \rightarrow \Rt,\ f \mapsto m_f := c_f^*-c_f, 
	\end{align}
	which connect centers of primal and Christoffel dual orthogonal circles, as shown in \mbox{Figure \ref{Fig:Rt_fundamental_hex}}.
\end{proposition}

\begin{proof}
	The circles lie in parallel planes and they lie on a common sphere. This follows form the fact that the linear subspace containing the lift of a fundamental cmc hexahedron (as illustrated in Figure \ref{Fig:Rt_fundamental_hex}) to $\Rfo $ is four-dimensional.  Its orthogonal complement represents a sphere that is orthogonal to all eight vertex spheres and, in particular, contains the two circles.
	Consequently, their circle axes coincide and are given by $m_f$.
\end{proof}

We use the geometric observations to prove Theorem \ref{Rt:Thm_s_cmc_koebe}.

\begin{proof}[Proof of Theorem \ref{Rt:Thm_s_cmc_koebe}]	The Gauss map \eqref{eq:Rt_vertex_normals} and the center nets $c$ and $c^*$ have parallel edges 
	\begin{equation}
		\label{eq:Rt_parallel_edges}
		c_v - c_{v'} \parallel n_v - n_{v'} \parallel c^*_v - c^*_{v'}.
\end{equation}
	In particular, the Gauss map shares the same face normals \eqref{eq:face_normals},
	\begin{align*}
		m_f \perp
		[n_v, n_{v_i}, n_{v_{ij}}, n_{v_j}],
	\end{align*}
	and therefore is a Q-net.
	Using the squared length of the edge normals \eqref{eq:edge_normals_length}, we define two global constants:
	\begin{align*}
		\Delta_+^2 : =  -2\alpha +2\lambda, \quad  \Delta_-^2 : =  -2\alpha -2\lambda.
	\end{align*}
	The Gauss images of the edge normals \eqref{eq:edge_normals} lie on two concentric spheres $S_\pm^2$ of radii $\Delta_\pm$, as illustrated in Figure \ref{Fig:Rt_fundamental_hex_gauss}. A combination of \eqref{eq:edge_normals_orthogonality}  and \eqref{eq:Rt_parallel_edges} yields
	\begin{align*}
		l_{(v, v')} \perp n_v - n_{v'}.
	\end{align*}
	Thus, the edges $n_v - n_{v'}$ are alternately tangent to $S_+^2$ and $S_-^2$ at the points $l_{(v, v')}$.
	By a suitable scaling of the initial S$_1$-cmc pair, we can adjust $\Delta_+$ and $\Delta_-$ to new radii $r_+$ and $r_-$ that satisfy the relation $r_+r_-=1$. Consequently, the Gauss map \eqref{eq:Rt_vertex_normals} indeed forms a two-sphere Koebe net.
\end{proof}
\begin{figure}
	\begin{overpic}[scale=.13]
		{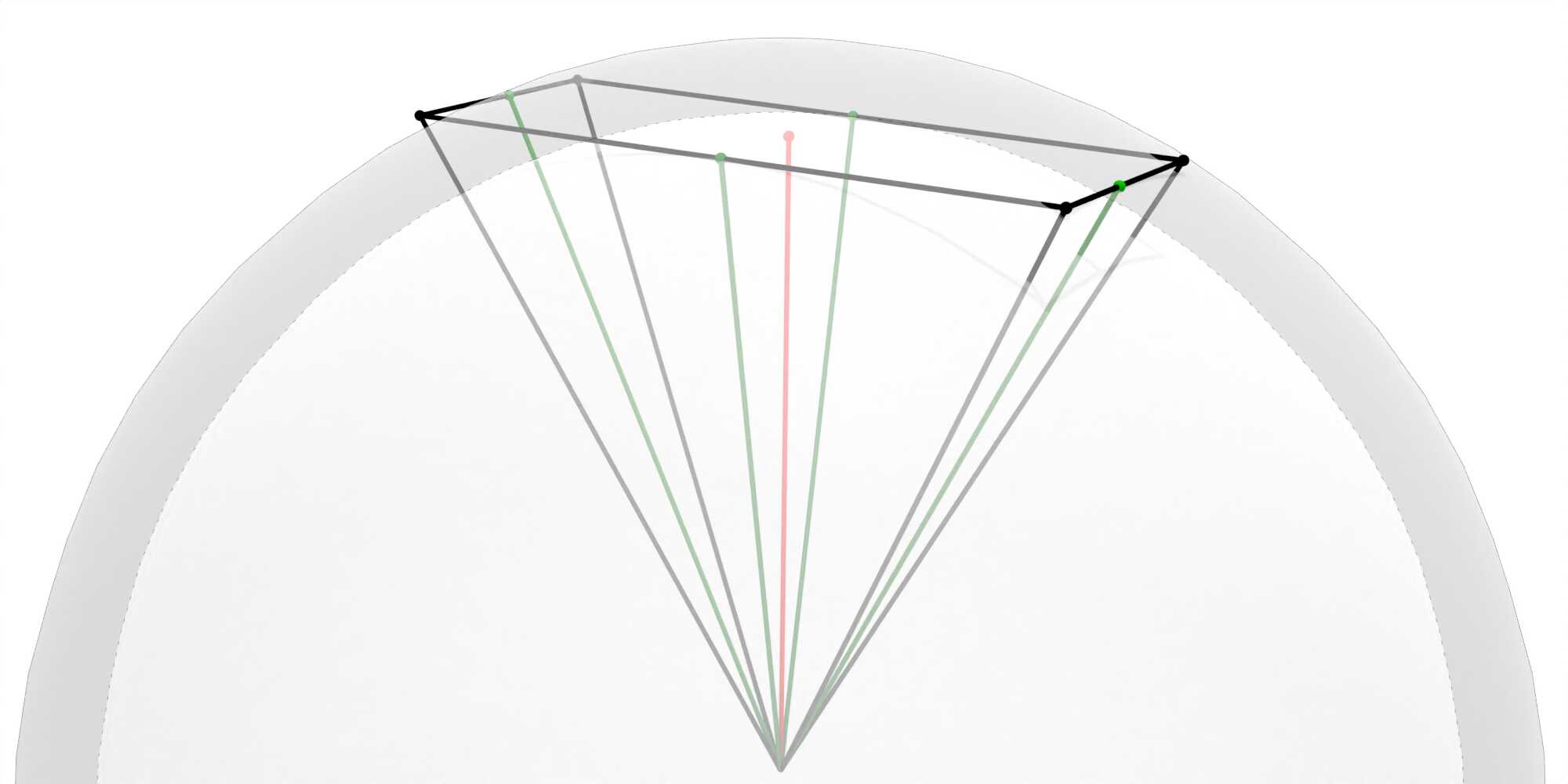}
		
		\put(22,42){$n_v$}
		\put(36,46){$n_{v_i}$}
		\put(61,35){$n_{v_j}$}
		\put(75,41){$n_{v_{ij}}$}
		
		\put(51,40.5){$m_f$}
		
		\put(41,37){$l_{(v, v_j)}$}
		
		\put(25,46){$l_{(v, v_i)}$}
		
		\put(99,1){$S_+^2$}
		\put(87,1){$S_-^2$}
	\end{overpic}
	\caption{The Gauss image of the fundamental S$_1$-cmc hexahedron of Figure \ref{Fig:Rt_fundamental_hex}, which forms a face of a two-sphere Koebe net. It alternately touches the two spheres $S^2_+$ and $S^2_-$, with the points of tangency given by the edge normal vectors $l_{(v, v')} \in S_\pm^2$ (green). Horizontal edges always touch $S^2_+$ and vertical edges always touch $S^2_-$. The vector $m_f$ (red) is orthogonal to the planar face. }
	\label{Fig:Rt_fundamental_hex_gauss}
\end{figure}
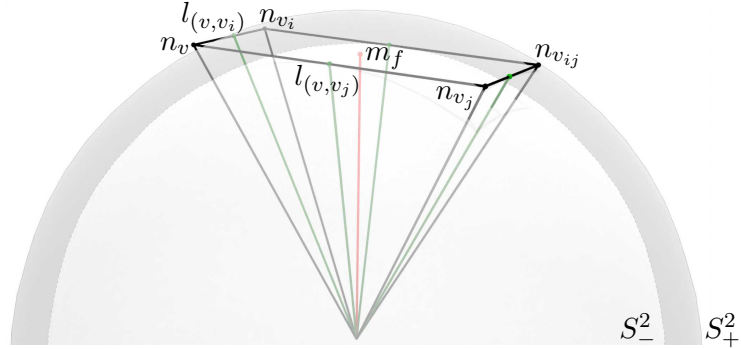

The edges of the Gauss map may change direction relative to the edges of $c$. From Figures \ref{Fig:Rt_side_faces} and \ref{Fig:Rt_fundamental_hex}, we can observe that the change in orientation can only occur along horizontal edges.

\subsection{Discrete curvatures}
\label{sec:Rt_discrete_curvatures}
We have seen that the center nets $c$ and $c^*$ of two S$_1$-isothermic surfaces $s$ and $s^*$ and the Gauss map $n$ are Q-nets with parallel edges. We interpret the pair $(c, n)$ as a Q-net with a parallel Gauss map $n$. Discrete curvatures for such pairs were defined in \cite{bobenko2010curvature} using a discrete version of the classical Steiner formula. Consider the one parameter family of parallel surfaces \mbox{$c_t := c + tn$} for $t \in \R$. Let $c(f)$ and $n(f)$ be the quadrilaterals (with parallel edges) corresponding to the same combinatorial face $f$. The area of the quadrilateral $c_t(f)$ is a quadratic polynomial of $t$
\begin{align*}
	A(c_t(f)) & = \left( 1-2 H_f t+ K_f t^2\right) A(c(f)), 
\end{align*}
where 
\begin{align}
	\label{eq:Rt_discrete_mean_curvature_local}
	H_f = -\frac{A(c(f), n(f))}{A(c(f))} \text{ and } K_f = \frac{A(n(f))}{A(c(f))}, 
\end{align}
define \emph{discrete mean curvature} and \emph{discrete Gaussian curvature} of the face $f$, respectively.
Recall that the area form of a planar polygon
$A(P)$ is a quadratic from on the space of planar polygons. The corresponding symmetric bilinear form $$A(P, Q) = \frac{1}{2}(A(P+Q)-A(P)-A(Q))$$ is the \emph{mixed area} of two planar polygons $P$ and $Q$ with parallel edges. 

In a slight misnomer we will now consider the mean curvature of a pair $(s, n)$ of an S$_1$-isothermic surface and its Gauss map, where we in fact consider the mean curvature of the pair $(c, n)$ as discussed above.

\begin{theorem}
	\label{Thm:Rt_H=1}
	Let $s$ and $s^*$ be an S$_1$-cmc pair such that the Gauss map $n$, given by \eqref{eq:Rt_vertex_normals}, 
	forms a two-sphere Koebe net.
	Then the pair $(s, n)$ has constant discrete mean curvature $H=1$.
	The Christoffel dual is a constant vertex offset surface in normal direction at distance $1$:
	\begin{align}
		\label{eq:Rt_constant_offset}
		c^* = c + n.
	\end{align}
\end{theorem}

\begin{proof}
	Identity \eqref{eq:Rt_constant_offset} is the definition of $n$.
	The mixed area vanishes for pairs of Christoffel dual quadrilaterals \cite{bobenko2010curvature, bobenko2008discrete}. Thus 
	\begin{align*}
		%\label{eq:Rt_Steiner_1}
		A(c, c^*)=0,
	\end{align*}
	which is equivalent to  
	\begin{align*}
		 H = -\frac{A(c, n)}{A(c)} = 1
	\end{align*}
	if $c^* = c+ n$. 
\end{proof}

We restrict our considerations to discrete $H=1$ cmc surfaces.
An S$_1$-cmc surface pair may be scaled, $s_\mu = \mu s, s^*_\mu = \mu s^*$, by a global factor $\mu$ while maintaining the same Gauss map $n$, to obtain a pair $(s_\mu, n)$ with $H_\mu = \frac{1}{\mu}$. The center net $c^*_\mu$ is a normal offset surface of $c_\mu$ at distance $\frac{1}{H_\mu}$:
\begin{align}
	\label{eq:Rt_constant_offset_scaled}
		c_\mu^* = c_\mu + \frac{1}{H_\mu} n, 
	\end{align}
analogous to the smooth case.

\section{Two-sphere Koebe nets and orthogonal ring patterns in the sphere}
\label{sec:Rt_two_spheres_Koebe_q_nets_and_spherical_orp}
In this section we will consider the correspondence between two-sphere Koebe nets and spherical orthogonal ring patterns detached from S$_1$-cmc surfaces. Classical Koebe polyhedra are known to be in one to one correspondence with spherical orthogonal circle patterns. Each vertex of a Koebe polyhedron can be associated with a sphere centered at the vertex, which intersects the unit sphere $S^2$ orthogonally. The intersection of all such spheres with $S^2$ yields a circle packing. A second circle packing is obtained by intersecting the faces of the Koebe polyhedron with $S^2$. These two families of circles together form an orthogonal circle pattern. Conversely, by lifting the circle centers of a circle pattern off the unit sphere, one can obtain a pair of combinatorially dual Koebe polyhedra \cite{bobenko2004variational}.

Rather than treating primal and dual combinatorics separately, we will now consider the more general combinatorics of S-quad graphs.
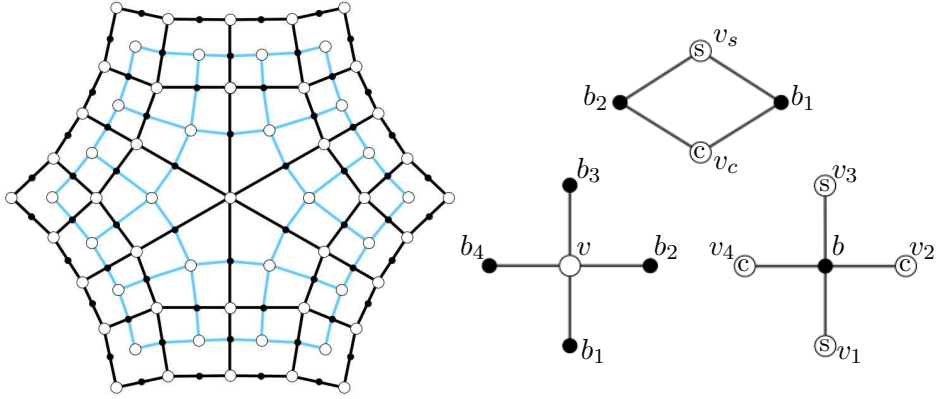
\begin{figure}
	\centering
	\begin{minipage}{.4\linewidth}
		\centering
			\includegraphics[width=\linewidth]{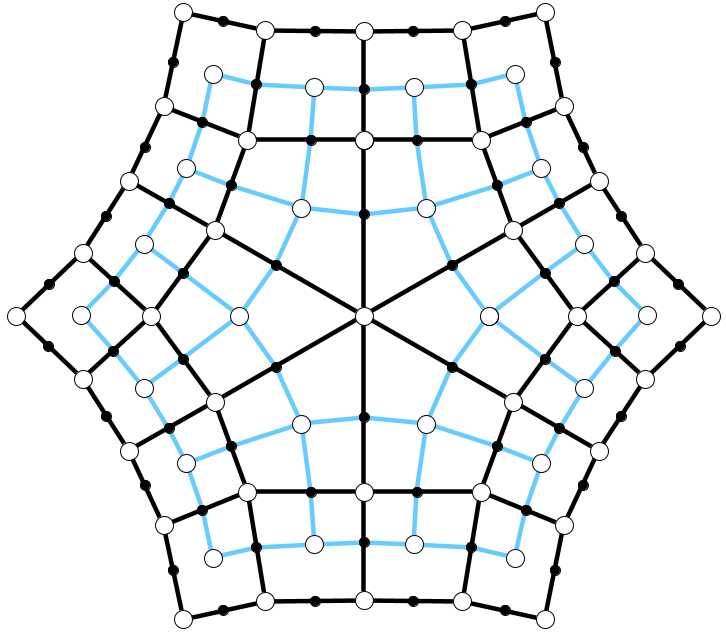}
	\end{minipage}
	\begin{minipage}{.4\linewidth}
		\centering
					\begin{overpic}[ width=\linewidth,]{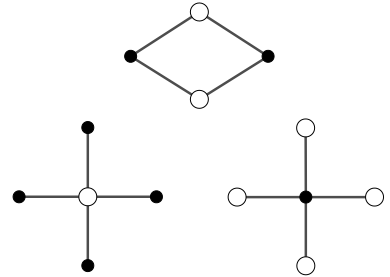}
				\put(25,0){$b_1$}
				\put(41,23){$b_2$}
				\put(24,40){$b_3$}
				\put(-1.5,23){$b_4$}
				\put(24,23){$v$}
				
				\put(81,0){$v_1$}
				\put(77.3,1.6){\small s}
				\put(97,23){$v_2$}
				\put(94.8,19.3){\small c}
				\put(77.3,37){\small s}
				\put(59.3,19.3){\small c}
				\put(80,40){$v_3$}
				\put(53,23){$v_4$}
				\put(80,23){$b$}
				
				\put(49.7,44.3){\small c}
				\put(54,42){$v_c$}
				\put(49.85,67){\small s}
				\put(54,70){$v_s$}
				\put(71,56){$b_1$}
				\put(25.5,56){$b_2$}

			\end{overpic}
	\end{minipage}
\hspace{0cm}
	\caption{Left: The two graphs $\G$ (black, consisting of white \textcircled{s}-vertices) and  $\G^*$ (blue, consisting of white \textcircled{c}-vertices) associated to an S-quad graph $\mathcal{S}$.
	From a graph $\mathcal{G}$ one obtains the S-quad graph via central extension, i.e., by introducing a new white vertex per face and a new black vertex per edge.
	Right: Labeling of vertices.}
	\label{Fig:Rt_S_Quad_graph_diagonal}
\end{figure}

\begin{definition}
	\label{def:S_quad}
	An \emph{S-quad-graph} is a quad graph $\mathcal{S}$ with interior vertices of even degree and the following additional properties:
	\begin{enumerate}[(i)]
		\item The 1-skeleton of $\mathcal{S}$ is bipartite and the vertices are bi-colored black and white
		\begin{align*}
			V(\S) = V_w(\S) \ \dot{\cup} \  V_b(\S).
		\end{align*}
		Then each quadrilateral has two black vertices and two white vertices.
		\item The white vertices may be labeled \emph{\textcircled{c}} and \emph{\textcircled{s}},
		\begin{align*}
			V_w(\S) = V_{\text{\normalfont \textcircled{s}}}(\S) \ \dot{\cup} \  V_{\text{\normalfont \textcircled{c}}}(\S), 
		\end{align*} in such a way that each quadrilateral has one white vertex labeled \emph{\textcircled{c}} and one white vertex labeled \emph{\textcircled{s}}.
		\item Interior black vertices and interior \emph{\textcircled{c}}-vertices have degree 4.
	\end{enumerate}
\end{definition}

Recall that for S-isothermic surfaces, we considered quad graphs with interior vertices of even degree. Any graph $\G$ of this type can be extended to an S-quad graph via \emph{central extension}, see Figure \ref{Fig:Rt_S_Quad_graph_diagonal}. Conversely, from any S-quad graph $\mathcal{S}$, one can obtain two graphs: a quad graph $\G$ with interior vertices of even degree and its dual graph $\G^*$. The edge coloring of the graph $\G$ into horizontal and vertical edges in Figure \ref{Fig:graph_g} can be consistently extended to $\S$ so that all faces of $\S$ have alternating horizontal and vertical edges.

Our definition of S-quad graphs is slightly more restrictive than the definition in \cite{BHS_2006}, where neither of the graphs $\G$ and $\G^*$ is restricted to be a quad graph.

The Definition \ref{def:Rt_two-sphere_Koebe} of two spheres Koebe nets with underlying quad graph $\G$ can be extended to graphs of the combinatorics of $\G^*$ in a straightforward manner.

\begin{definition}
	\label{def:Rt_dual_two-sphere_Koebe}
	Let  $\G$ and  $\G^*$ be the two graphs associated with an S-quad graph $\S$.
	We call the two-sphere Koebe nets $k^s:V(\G) \rightarrow \R^3$ and $k^c:V(\G^ *) \rightarrow \R^3$ \emph{a pair of dual two-sphere Koebe nets} if they satisfy the following conditions:
	\begin{enumerate}
		\item The nets $k^s$ and $k^c$ touch the same two spheres $S_+^2$ and $S_-^2$.
		\item Vertex vectors of $k^s$  form normals of the corresponding faces of $k^c$ and vice versa.
		\item Dual edges of $k^s$ and $k^c$ are orthogonal. Their points of tangency lie on different spheres and are projected to the same point on the unit sphere.
	\end{enumerate}
\end{definition}

Pairs of dual two-sphere Koebe nets are examples of principle binets, which are pairs of combinatorially dual, conjugate nets with orthogonal dual edges. Principle binets were introduced in \cite{affolter2024principalbinets} as a discretization of a curvature line parametrization.

Let $k^s$ and $k^c$ be a pair of dual two-sphere Koebe nets. For our purposes, we assume that the points of tangency adjacent to a vertex of $k^c$ or $k^s$,  corresponding to the points $b_1, b_2, b_3, b_4$ in Figure \ref{Fig:Rt_S_Quad_graph_diagonal} (right),  are cyclically ordered, either in clockwise or in counterclockwise order. We call the corresponding two-sphere Koebe nets \emph{regular}. Note that the two-sphere Koebe nets obtained as the Gauss map from S$_1$-cmc surfaces are regular. We assume that the coloring of the edges of $\G$ is chosen so that the horizontal edges touch $S^2_+$ and the vertical edges touch $S^2_-$.

For a white vertex $v \in V_w(\S)$, let $k_v$ denote the corresponding vertex of either $k^s$ or $k^c$, and 
\begin{align}
	\label{eq:Rt_proj_verts}
	p_v := \frac{k_v}{||k_v||}
\end{align}
its projection onto $S^2$. For a black vertex $\black \in V_\black(\mathcal{S})$ there are two corresponding points of tangency, $t^+_\black$ and $t^-_\black$, one of which belongs to $k^s$, the other to $k^c$. By definition they lie on different spheres $S^2_+$ and $S^2_-$ and project to the same point on $S^2$:
\begin{align}
	\label{eq:Rt_proj_touching_pts}
	q_\black := \frac{t^+_\black}{r^+} =  \frac{t^-_\black}{r^-}, 
\end{align}
where $r_+$ and $r_-$ are the radii of $S^2_+$ and $S^2_-$, respectively. We will show that the projections \eqref{eq:Rt_proj_verts} and \eqref{eq:Rt_proj_touching_pts} give rise to spherical rings (pairs of concentric circles) which intersect orthogonally \cite{bobenko2024rings}. For an illustration see Figure \ref{Fig:Rt_Koebe_2d_vertex}.

A ring is given by a triple $(p, r, R)$ where $p \in S^2$ is the center of the ring, and $r$ and $R$ are the radii of the inner and the outer circles, respectively. We assign an orientation to the ring by allowing $r$ to be negative: positive radius corresponds to counterclockwise orientation and negative radius corresponds to clockwise orientation. The outer radius is always positive. The spherical radii $r$ and $R$ are not allowed to be greater than $\frac{\pi}{2}$.

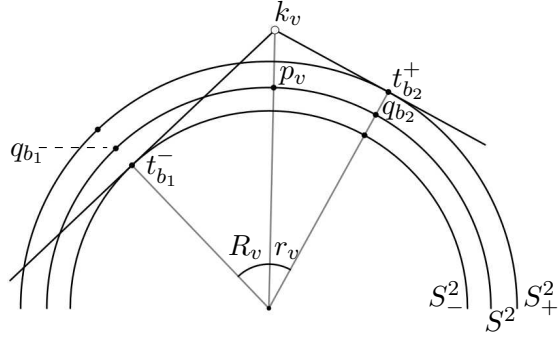
\begin{figure}
	\centering
	\begin{minipage}{.45\linewidth}
		\begin{overpic}[width=\textwidth]
			{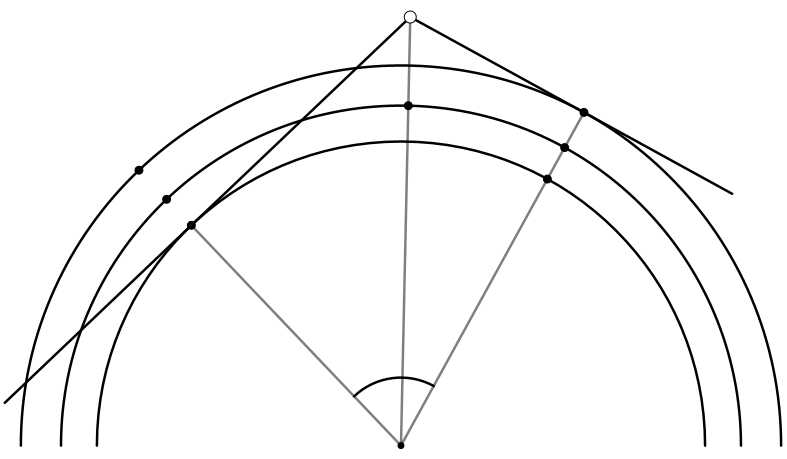}
			\put(27,28){$t^-_{b_1}$}
			\put(0.5, 31.5){$q_{b_1}$}
			\put(6, 33){\tikz \draw[dashed](0,0)--(0.9, 0);}
			
			\put(75,45){$t^+_{b_2}$}
			\put(73,39.7){$q_{b_2}$}
			
			\put(43,12){$R_v$}
			\put(52,12){$r_v$}
			
			\put(52,58){$k_v$}
			
			\put(53,46.5){$p_v$}
			
			\put(100.5,2){$S_+^2$}
			\put(82,2){$S_-^2$}
			\put(93,-3){$S^2$}
		\end{overpic}
	\end{minipage}
	\caption{A vertex $k_v$ of a two-sphere Koebe net, two adjacent edges with tangent points $t^-_{b_1}$ and $t^+_{b_2}$, and the projections $p_v, q_{b_1},$ and $q_{b_2}$ onto $S^2$. The corresponding ring $(p_v, r_v, R_v)$ in the projection has center $p_v$, passes through $q_{b_1}$ and $q_{b_2}$, and has spherical radii $r_v$ and $R_v$ given by the angles shown in the figure.}
	\label{Fig:Rt_Koebe_2d_vertex}	
\end{figure}

\begin{definition}
	\label{Def:Rt_orp}
	Let  $\mathcal{S}$ be an S-quad graph. 
	An \emph{orthogonal ring pattern} consists of rings associated with the white 
	vertices of $\mathcal{S}$ that satisfy the following properties:
	\begin{enumerate}
		\item
		The rings associated with the \emph{\textcircled{c}} and the \emph{\textcircled{s}} vertex of a quad, $(p_{v_1}, r_{v_1}, R_{v_1})$ and  $(p_{v_2}, r_{v_2}, R_{v_2})$, 
		\emph{intersect orthogonally}, i. e. the larger circle of one ring
		intersects the smaller circle of the other ring orthogonally, and vice versa.
		\item
		For four rings $(p_{v_1}, r_{v_1}, R_{v_1}),  (p_{v_2}, r_{v_2}, R_{v_2}), (p_{v_3}, r_{v_3}, R_{v_3})$ and $(p_{v_4}, r_{v_4}, R_{v_4})$ adjacent to a black vertex $b$ and ordered according to the combinatorics of the quad graph, see Figure \ref{Fig:Rt_S_Quad_graph_diagonal} (right), the inner circles $(p_{v_1}, r_{v_1})$ and $(p_{v_3}, r_{v_3})$ and
		the outer circles $(p_{v_2}, R_{v_2})$ and $(p_{v_4}, R_{v_4})$ pass through one point.
		(Then orthogonality implies that the two inner and the two outer circles
		touch at this point.)
		\item 
		The black vertices of the S-quad graph correspond to the touching points. For each ring, these points are ordered according to the orientation of the ring, i.e., the points corresponding to $b_1, b_2, b_3, b_4$ in Figure \ref{Fig:Rt_S_Quad_graph_diagonal} (right) are ordered counterclockwise if the ring corresponding to the white vertex $v$ is positively oriented, $r_v>0$, and clockwise if $r_v<0$.
	\end{enumerate}
\end{definition}

\begin{proposition}
	The projections $p_v$ and $q_b$, given in 
	\eqref{eq:Rt_proj_verts} and \eqref{eq:Rt_proj_touching_pts}, of vertices and tangent points of pairs of dual two-sphere Koebe nets define spherical rings $(p_v, r_v, R_v)$ that are centered at $p_v$ and pass through the projections $q_b$ of adjacent tangent points. The spherical radii $r_v$ and $R_v$ are given 
	(up to sign) by
	\begin{align}
		\label{eq:Rt_cos}
		\cos(r_v) = \frac{r_+}{||k_v||}, \cos(R_v) = \frac{r_-}{||k_v||}.
	\end{align}
	For regular two-sphere Koebe nets, the rings form a spherical orthogonal ring pattern, see Figure \ref{Fig:Rt_Koebe_and_sorp}. 
\end{proposition}

\begin{figure}[b]
	\centering

	\begin{minipage}{.28\linewidth}
		\centering
		\begin{overpic}[width=\linewidth]
			{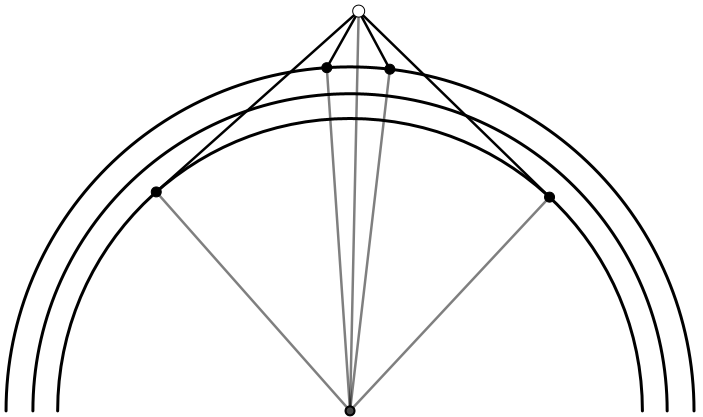}
			\put(50, 60){$k_v$}
		\end{overpic}
		\vspace{.2cm}
	\end{minipage}
	\hfill
	\begin{minipage}{.28\linewidth}
		\centering
		\begin{overpic}[width=\linewidth]
			{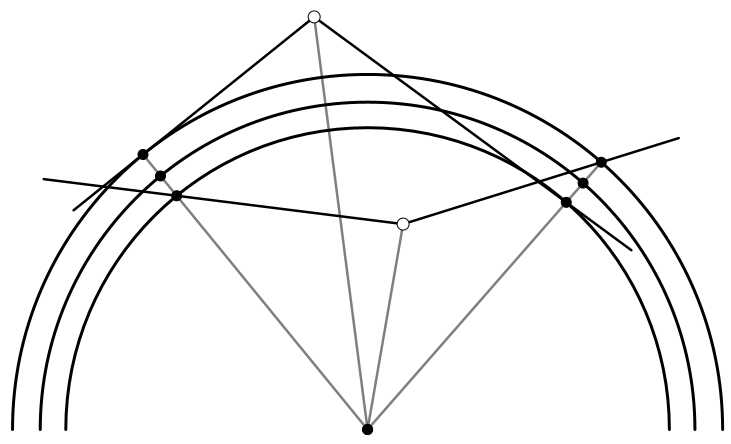}
			\put(40, 61){$k_{v_c}$}
			\put(49, 33.5){$k_{v_s}$}
		\end{overpic}
		\vspace{.2cm}
	\end{minipage}
	\hfill
	\begin{minipage}{.28\linewidth}
		\centering
		\begin{overpic}[width=\linewidth]
			{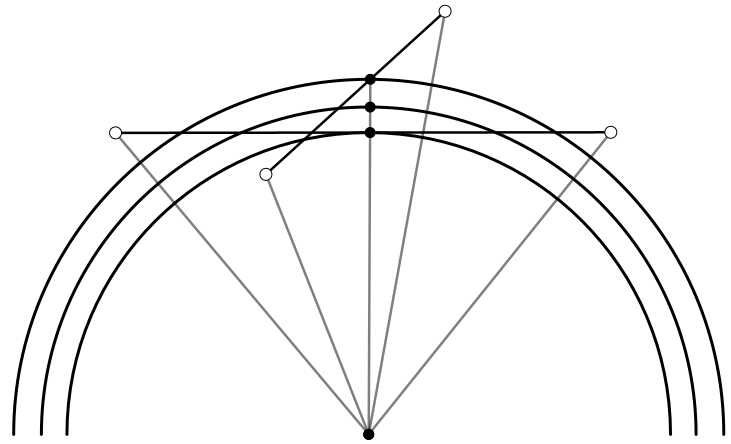}
			\put(60, 61.5){$k_{v_{3}}$}
			\put(39.5, 33){$k_{v_1}$}
			
			\put(10, 47){$k_{v_4}$}
			\put(81, 47){$k_{v_{2}}$}
		\end{overpic}
		\vspace{.2cm}
	\end{minipage}
	
	\begin{minipage}{.28\linewidth}
		\centering
		
		\begin{overpic}[trim={6cm 3cm 6cm 3cm}, clip,   scale=.3]
			{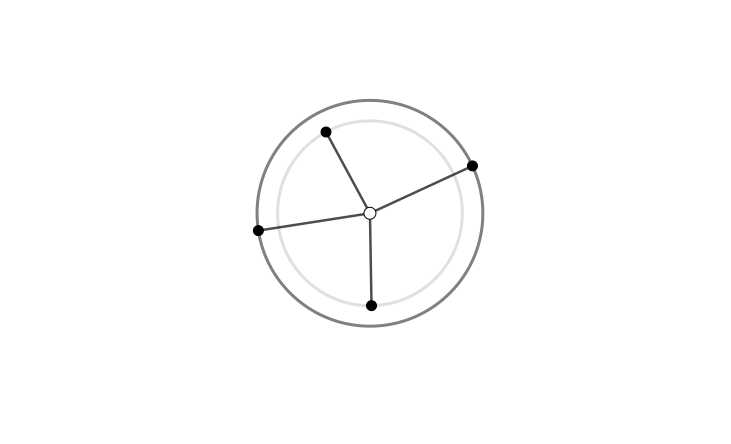}
			\put(55, 33){$p_v$}
		\end{overpic}
		\vspace{.4cm}
		\captionsetup{width=.1\textwidth} 
		\caption*{(1)}
	\end{minipage}
	\hfill
	\begin{minipage}{.28\linewidth}
		\centering
		\begin{overpic}[trim={5cm 0 5cm 0}, clip,   scale=1]
			{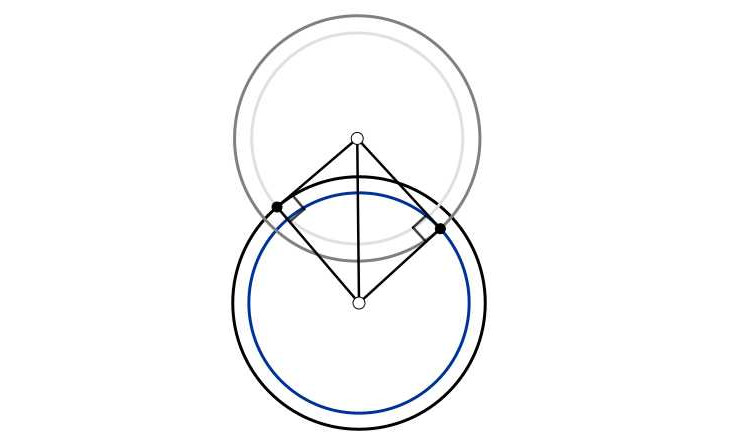}
			\put(-55, 22){$p_{v_s}$}
			\put(-55, 74){$p_{v_c}$}
		\end{overpic}
		\captionsetup{width=.1\textwidth} 
		\caption*{(2)}
	\end{minipage}
	\hfill
	\begin{minipage}{.28\linewidth}
		\begin{overpic}[trim={3.5cm 0 5cm 0}, clip,   scale=.25]
			{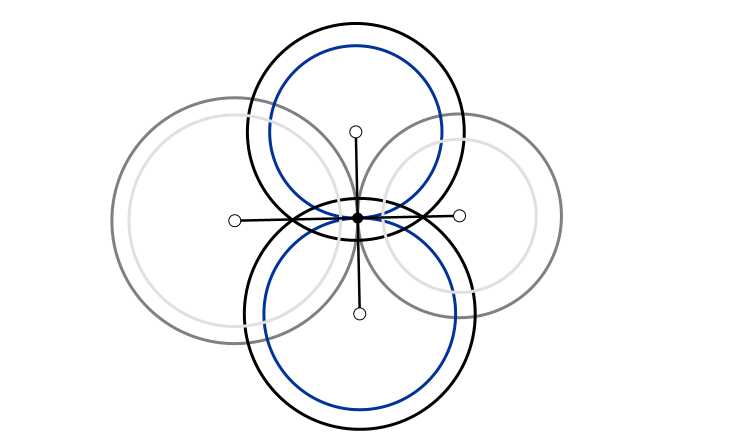}
			\put(53, 19.5){$p_{v_1}$}
			\put(51, 68){$p_{v_{3}}$}
			
			\put(21, 49){$p_{v_4}$}
			\put(75, 50){$p_{v_{2}}$}
		\end{overpic}
		\captionsetup{width=.01\textwidth} 
		\caption*{(3)}
	\end{minipage}
	\caption{Projecting a pair of dual two-sphere Koebe nets to $S^2$. (1) Vertices and adjacent tangent points project to spherical rings, (2) neighboring rings intersect orthogonally, (3) rings of white vertices adjacent to a common black vertex pass through a common point.}
	\label{Fig:Rt_Koebe_to_sorp}
\end{figure} 
\begin{proof}
	All tangent points on $S^2_+$ adjacent to a vertex vertex $k_v$ have the same distance to $k_v$ and are therefore are projected onto a circle in $S^2$ centered at $p_v$. The same holds for the tangent points on $S^2_-$, see Figure \ref{Fig:Rt_Koebe_to_sorp} (1). The two circles form a spherical ring $(p_v, r_v, R_v)$, where for the radii $r_v$ and $R_v$ we have \eqref{eq:Rt_cos}, since they coincide with the angles shown in Figure \ref{Fig:Rt_Koebe_2d_vertex}.
	
	Now let $v_c, b_1, v_s, b_2$ be the vertices of a face of $\S$, see Figure \ref{Fig:Rt_S_Quad_graph_diagonal} (right). The two corresponding rings in the projection,  $(p_{v_s}, r_{v_s}, R_{v_s})$ and 
	$(p_{v_c}, r_{v_c}, R_{v_c})$, intersect orthogonally. To see this, we first note that tangent points on $S^2_+$ always project onto inner circles and tangent points on $S^2_-$ always project onto outer circles. The orthogonality of the rings then follows from the orthogonality of pairs of dual edges of the two-sphere Koebe nets, see Figure \ref{Fig:Rt_Koebe_to_sorp} (2).
	
	With these observations we also find, that four rings corresponding to four vertices $v_1, v_2, v_3, v_4$ adjacent to a common black vertex $b$, see Figure \ref{Fig:Rt_S_Quad_graph_diagonal} (right), pass through a common point $q_b$, see Figure \ref{Fig:Rt_Koebe_to_sorp} (3). 
	Finally, the regularity assumption allows to choose signs of the inner radii of the rings such that the rings are oriented according to the order of the adjacent touching points, and indeed form a spherical orthogonal ring pattern.
\end{proof}
\begin{figure}
\centering
\begin{minipage}{.49\linewidth}
	\centering
	\includegraphics[width=.85\linewidth]{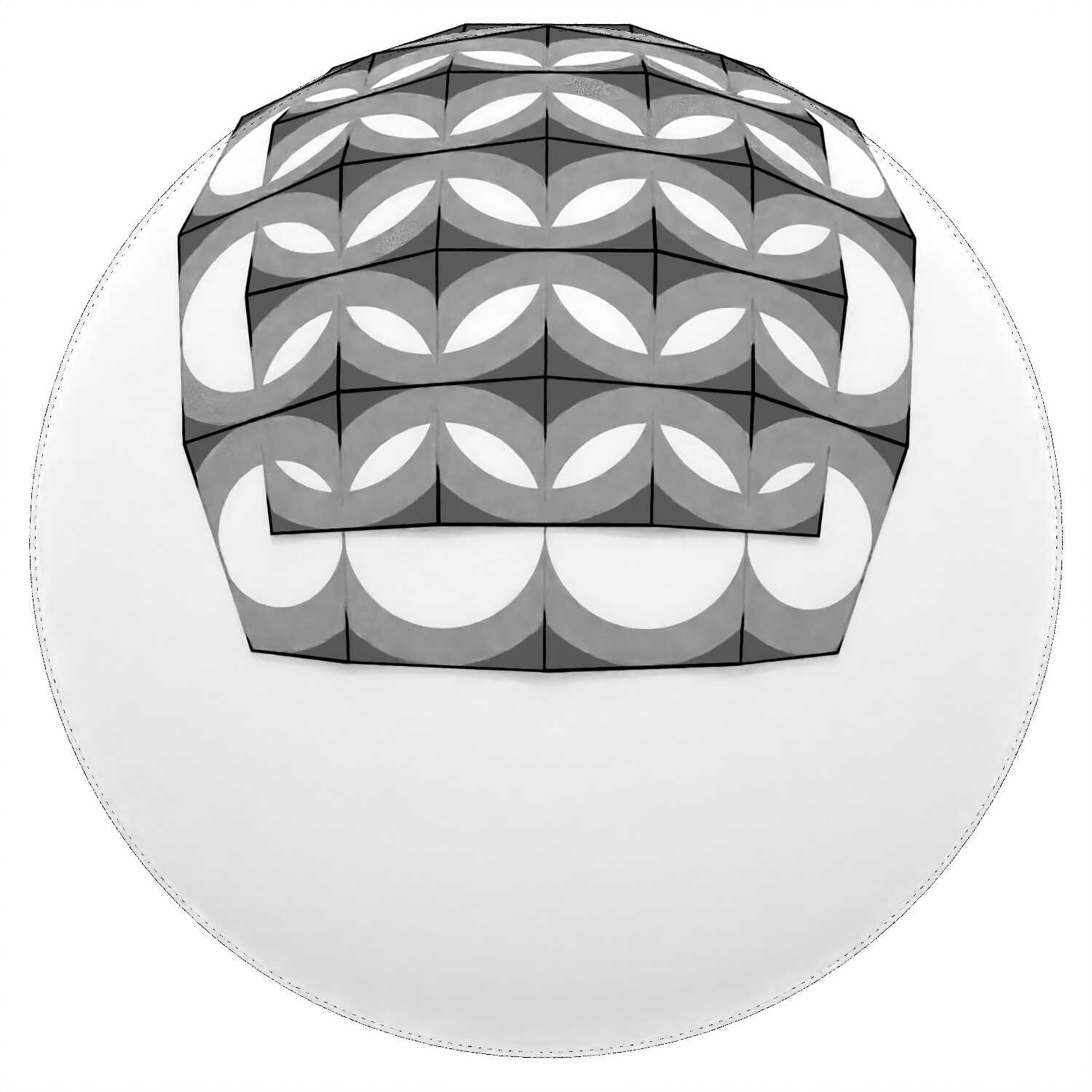}
\end{minipage}
\begin{minipage}{.49\linewidth}\centering
	\includegraphics[width=.85\linewidth]{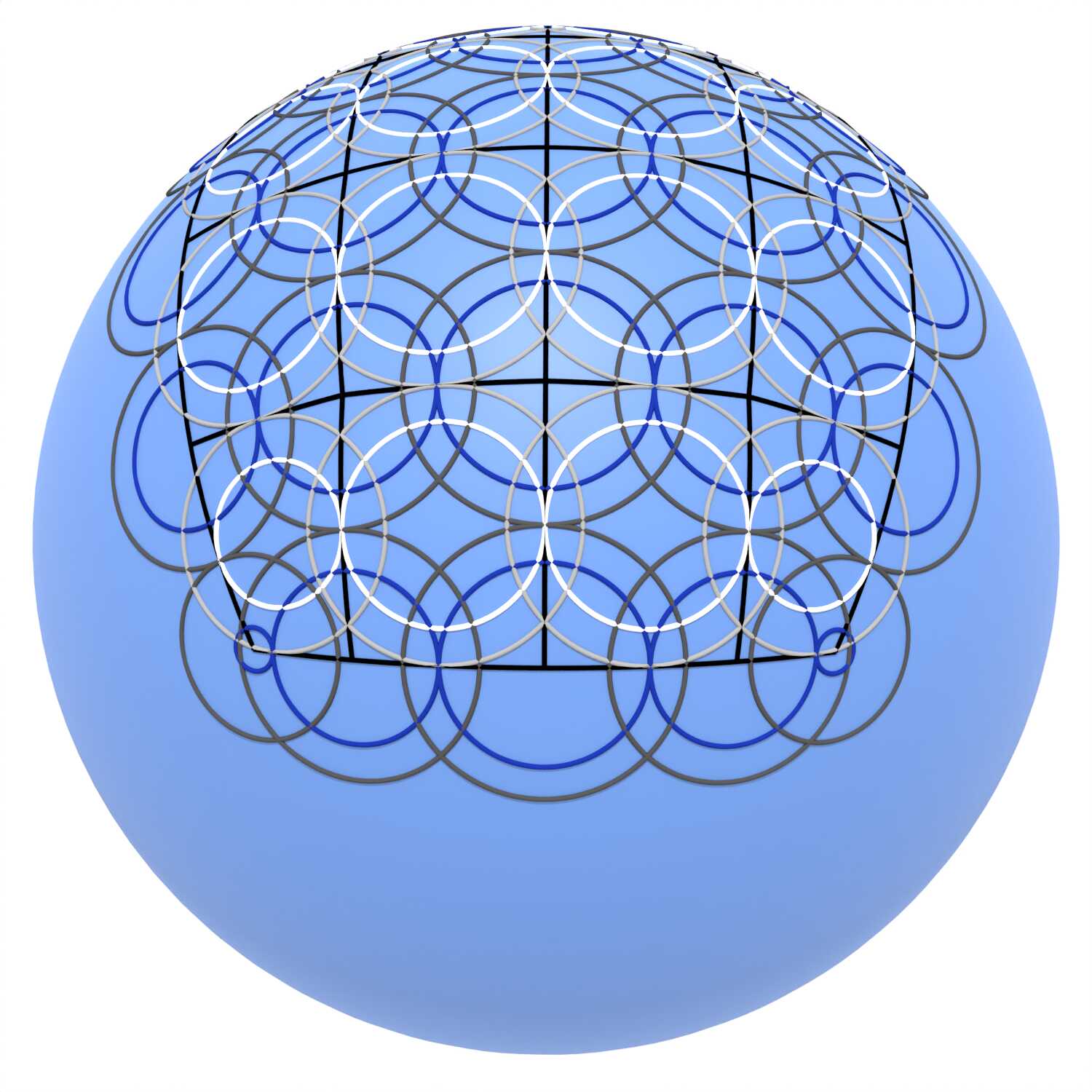}
\end{minipage}
\caption{A pair of dual two-sphere Koebe nets (left) and the corresponding spherical orthogonal ring pattern (shown with the combinatorics of $\G$, right). Vertices of the two-sphere Koebe nets correspond to centers of spherical rings and points of tangency correspond to the touching points of rings.}
\label{Fig:Rt_Koebe_and_sorp}
\end{figure}

If two spherical rings, $(p_{v}, r_{v}, R_{v})$ and $(p_{v'}, r_{v'}, R_{v'})$, intersect orthogonally, then applying the spherical Pythagorean theorem, we obtain
\begin{align*}
	\frac{\cos(R_{v})}{\cos(r_{v})} = \frac{\cos(R_{v'})}{\cos(r_{v'})}.
\end{align*}
For spherical orthogonal ring patterns, this implies that there 
exists a global constant $q < 1$ such that the radii of the inner and outer circles of all rings are related by 
\begin{align}
	\label{eq:Rt_q}
	q  \cos r_v  =  \cos R_v.
\end{align}  

There are two distinct directions in an orthogonal ring pattern, corresponding to horizontal and vertical edges of $\S$. Along horizontal edges, the inner circles of next neighboring rings touch, along vertical edges the outer circles of next neighboring rings touch. The rings form two families, corresponding to the coloring of white vertices $V_w(\S)$ into white \textcircled{s}-vertices and white \textcircled{c}-vertices. Touching rings always belong to the same family. Recall that the radius of inner circles can be negative, which can cause the ring pattern to flip and reverse direction along horizontal edges, see Figure \ref{Fig:Introduction} for an example.

A pair of dual two-sphere Koebe nets can be uniquely recovered from the corresponding spherical orthogonal ring pattern.
\begin{proposition}
	Consider a spherical orthogonal ring pattern with rings 
	$(p_v, r_v, R_v)$ and global parameter \mbox{$q< 1$}. 
	Define the vertices
	\begin{align}
		\label{eq:Rt_sorp_lift_p}
		k_v := \frac{\sqrt{q}}{\cos(R_v)}\ p_v = \frac{1}{\sqrt{q}\cos(r_v)}\ p_v
	\end{align}
	and the points
	\begin{align}
		\label{eq:Rt_sorp_lift_t}
		t_b^+ := \frac{1}{\sqrt{q}}\ q_b,\quad t_b^- := \sqrt{q} \  q_b,
	\end{align}
	which scale the centers $p_v$ and the touching points $q_b$  off the unit sphere.
	By restricting the map $$k: V_w(\mathcal{S}) \rightarrow \Rt, \ v \mapsto k_v$$  to the graphs $\G$ and  $\G^*$ of $\mathcal{S}$, one obtains
	a pair of regular, dual two-sphere Koebe nets, which are alternately tangent to the two spheres  $S^2_+$ and $S^2_-$, with radii $\frac{1}{\sqrt{q}}$ and $\sqrt{q}$, at the points \eqref{eq:Rt_sorp_lift_t}.
\end{proposition}

\begin{proof}
	 Consider four vertices $v_1, v_2, v_3, v_4$ adjacent to a common black vertex $b$ in $\S$, see Figure \ref{Fig:Rt_S_Quad_graph_diagonal} (right). 
	The corresponding spherical rings $(p_{v_1}, r_{v_1}, R_{v_1})$, $(p_{v_{2}}, r_{v_{2}}, R_{v_{2}})$, $(p_{v_3}, r_{v_3}, R_{v_3})$ and $(p_{v_{4}}, r_{v_{4}}, R_{v_{4}})$ pass though a common point $q_b$. Suppose that the inner circles of $(p_{v_1}, r_{v_1}, R_{v_1})$  and $(p_{v_3}, r_{v_3}, R_{v_3})$ and the outer circles of $(p_{v_{2}}, r_{v_{2}}, R_{v_{2}})$ and $(p_{v_{4}}, r_{v_{4}}, R_{v_{4}})$ touch in $q_b$, see Figure \ref{Fig:Rt_Koebe_to_sorp} (3). Scaling the centers and the touching points by \eqref{eq:Rt_sorp_lift_p} and \eqref{eq:Rt_sorp_lift_t} we obtain the edges $[k_{v_1}, k_{v_{3}}]$ and $[k_{v_2}, k_{v_{4}}]$ such that 
	\begin{align}
		\label{eq:Rt_orth_0}
		t^+_b \in [k_{v_1}, k_{v_{3}}] \
        ,\ q_b \perp [k_{v_1}, k_{v_{3}}] \ , \
		t^-_b \in [k_{v_2}, k_{v_{4}}] \ , \
		 q_b \perp [k_{v_2}, k_{v_{4}}].
	\end{align}
	
	The orthogonality of neighboring rings and the tangency of next neighboring rings implies the orthogonality 
	\begin{align}
		\label{eq:Rt_orth_1}
		[k_{v_1}, k_{v_{3}}]  \perp [k_{v_2}, k_{v_{4}}] .
	\end{align}
	
	It remains to show that for the two resulting polyhedral surfaces the vertex vectors of one surface form the face normal vectors of the other. To this end consider a face of $\mathcal{G}$ with white \textcircled{s}-vertices $v_{s_1}, v_{s_2}, v_{s_3}, v_{s_4}$ dual to a vertex $v_c$.
	The vector $k_{v_c}$ is orthogonal to all four edges of the quadrilateral
    $[k_{v_{s_1}}, k_{v_{s_2}}, k_{v_{s_3}},k_{v_{s_4}}]$, which in particular implies in particular the planarity of this quadrilateral. To show, for example the orthogonality \begin{align}
    	\label{eq:Rt_orth_2}
    	[k_{v_{s_1}}, k_{v_{s_2}}] \perp k_{v_{c}}
    \end{align}
 we observe that
     \begin{align}
     	\label{eq:Rt_orth_3}
     	[k_{v_{s_1}}, k_{v_{s_2}}] \perp [t^-_b, k_{v_{c}}].
     \end{align}
     The latter orthogonality follows from \eqref{eq:Rt_orth_1} and the fact that $t^-_b \in [k_{v_c}, k_{v_{c'}}]$. A combination of \eqref{eq:Rt_orth_0} and \eqref{eq:Rt_orth_3} implies \eqref{eq:Rt_orth_2}.
     The same also holds for faces of $\mathcal{G}^*$ with valency greater than four.
\end{proof}
We have established the following correspondence, shown in  Figure \ref{Fig:Rt_Koebe_and_sorp}.

\begin{theorem}
	\label{Thm:Rt_Koebe_and_orp}
	Pairs of regular, dual two-sphere Koebe nets touching the spheres $S^2_+$ and $S^2_-$, with radii $r_+$ and $r_-$, are in one to one correspondence with orthogonal ring patterns in $S^2$ with global parameter $q = \frac{r_-}{r_+}$. Vertices of the Koebe nets correspond to centers of spherical rings, while points of tangency of the Koebe nets correspond to touching points of the rings.
\end{theorem}

\begin{remark}
	Classical Koebe nets are specific examples of S$_1$-isothermic surfaces. In this case, the vertex spheres are given by the spheres that intersect $S^2$ orthogonally, as discussed earlier in this section. However, for two-sphere Koebe nets, each vertex has a pair of concentric vertex spheres: the smaller sphere intersects $S^2_+$ orthogonally and the larger sphere intersects $S^2_-$ orthogonally. The smaller spheres touch along horizontal edges, while the larger spheres touch along vertical edges. For both families of vertex spheres, orthogonal face circles exist, formed by the intersection of the two-sphere Koebe net with the spheres $S^2_+$ and $S^2_-$, respectively. Unlike classical Koebe nets, two-sphere Koebe nets do not belong to the class of S$_1$-isothermic surfaces.
\end{remark}

\section{Analytic description of orthogonal ring patterns in the sphere}
\label{sec:analytic_orp}
This section presents the basics of the analytic description, including the variational description, of spherical orthogonal ring patterns. For proofs and further details we refer to \cite{bobenko2024rings}.

We consider an S-quad graph $\Sz$ that is defined by a simply connected subset of squares of the $\mathbb{Z}^2$ lattice in $\R^2$.
Our main example is a combinatorial rectangle
	\begin{align}
		\label{eq:Z2_rectangle}
		\Sz = \{(i, j) \in \Z^2 | 1 \leq i \leq I, 1 \leq j \leq J\}.
\end{align}
Let us now consider an orthogonal ring pattern with global parameter $q\leq 1$ and underlying S-quad graph $\Sz$. There is a ring associated with each white vertex, with inner radius $r_v$ and outer radius $R_v \leq \frac{\pi}{2}$. Due to \eqref{eq:Rt_q}, these radii can be uniformized in terms of Jacobi elliptic functions, 
\begin{equation}
	\label{eq:Rt_jef_uniformization}
	\cos r_v=\sn (\beta_v,q) ,\ \sin r_v=\cn (\beta_v,q), \ \sin R_v=\dn (\beta_v, q),
\end{equation}
associating a variable $\beta_v \in [0, 2K]$ to each ring.
Here $K$ denotes the real quarter period of Jacobi elliptic functions \cite{nist}. 

To characterize rings forming orthogonal ring patterns, let us introduce the function 
\begin{equation}
	\label{Rt_g}
	g(x):=\frac{\pi}{2}-\arg \sn\left(\frac{x+iK'}{2}\right).
\end{equation}
The uniformizing variables $\beta:V_w(\Sz)\rightarrow [0, 2K]$ determine an orthogonal ring pattern with $R\leq \frac{\pi}{2}$ if and only if for all internal white vertices $v$ they satisfy
\begin{equation*}
	\label{Rt_interiod_vertex_ring_angles_g}
	\sum_{v' \sim v} g(\beta_v+\beta_{v'})-g(\beta_v-\beta_{v'})=2\pi, 
\end{equation*}
where the sum is taken over all neighboring rings of $(p_v, r_v, R_v)$ that intersect it orthogonally. Combinatorially, neighboring vertices  ${v' \sim v}$ are the two white vertices of an elementary quad of the S-quad graph $\Sz$, see Figure \ref{Fig:Rt_S_Quad_graph_diagonal} (right).
With prescribed nominal angles $\Theta$, a similar condition must hold for white boundary vertices:
\begin{eqnarray*}
	\label{eq:angle_boundary_beta}
	\pi \deg(v)+\sum_{v' \sim v} g(\beta_v-\beta_{v'})-g(\beta_v+\beta_{v'})=\Theta_v, \quad \text{if}\ \ r_v>0, \\
	\sum_{v' \sim v} g(\beta_v-\beta_{v'})-g(\beta_v+\beta_{v'})=\Theta_v, \quad \text{if}\ \ r_v<0. \nonumber
\end{eqnarray*}

Using the anti-derivative of (\ref{Rt_g}), 
 \begin{eqnarray}
	\label{eq:F(x)}
	& & F(x)=
	\int_0^x \frac{\pi}{2} -\arg\sn\frac{u+iK'}{2} du=\int_0^x \arctan \frac{(1+q)\sn\frac{u}{2}}{\cn\frac{u}{2}\dn\frac{u}{2}}du,
\end{eqnarray}
one defines the functional
\begin{equation}
	\label{Rt_eq:functional_spherical}
	S_{sph}(\beta):=\sum_{(v, v')} \left( F(\beta_{v}-\beta_{v'})-F(\beta_{v}+\beta_{v'})\right) +\sum_{v}\Phi_{v}\beta_{v},
\end{equation}
where the first sum is taken over all pairs of white vertices corresponding to neighboring rings and the second sum over all white vertices $v \in  V_w(\Sz)$. $\Phi_{v}$ are some prescribed parameters at the vertices.
\begin{figure}	
	\begin{minipage}{.49\linewidth}
		\centering
		\includegraphics[width=.7\linewidth]{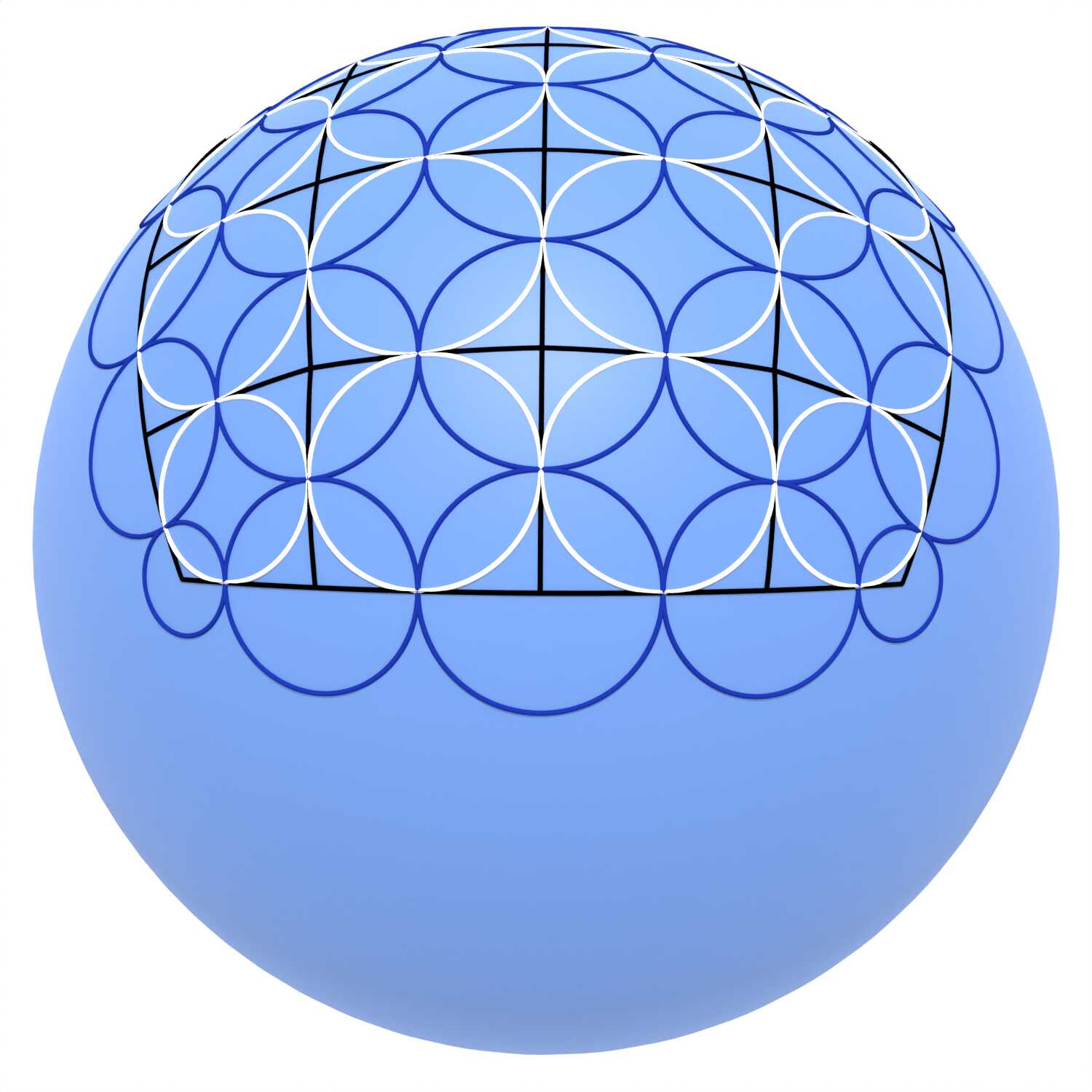}
	\end{minipage}
	\begin{minipage}{.49\linewidth}
		\centering
		\includegraphics[width=.7\linewidth]{figures/4x4_deformation_to_ring}
	\end{minipage}
	\caption{A local patch of a spherical orthogonal circle pattern with q=.9999999 (left) and its deformation to a spherical orthogonal ring pattern with q=0.99 (right). All corner angles are of size $\frac{2\pi}{3}$.}
	\label{Fig:Rt_orp_deformation}
\end{figure}
\begin{theorem}
	\label{thm:variational_spherical}
	The critical points $\beta$ of the functional $S_{sph}$ with all \mbox{$\beta_{v}\in [0,2K]$} and 
	\begin{eqnarray*}
		&\Phi_{v}=2\pi & \text{for  inner rings}, \nonumber \\
		&\Phi_{v}=\pi \deg({v})-\Theta_v &  \text{for positively oriented boundary rings},    \\
		&\Phi_{v}=-\Theta_v &  \text{for negatively oriented boundary rings}, 
	\end{eqnarray*}
	correspond to spherical orthogonal ring patterns with all $R_{v}\le\frac{\pi}{2}$.
	Here $\Theta_v$ is the nominal angle covered by the neighboring rings respectively. 
\end{theorem}
A proof of the theorem is provided in \cite{bobenko2024rings}. The functional \eqref{Rt_eq:functional_spherical} is not convex, its second derivative 
 \begin{eqnarray*}
	D^2 S_{sph}= &\frac{1}{2}\sum_{(v, v')} (\dn(\beta_v-\beta_{v'})+q\cn(\beta_v-\beta_{v'}))(d\beta_v-d\beta_{v'})^2-\\
	&(\dn(\beta_v+\beta_{v'})+q\cn(\beta_v+\beta_{v'}))(d\beta_v+d\beta_{v'})^2\nonumber
\end{eqnarray*}
is negative for the tangent vector $u=\sum_v \partial/\partial\beta_v$, the
index is therefore at least $1$. We define a reduced functional
$\widetilde{S}_{sph}(\beta)$ by maximizing in the direction $u$:
\begin{equation}
	\label{eq:Rt_Ssph_reduced}
	\widetilde{S}_{sph}(\beta)=\max_t S_{sph}(\beta+t u).
\end{equation}
Obviously, $\widetilde{S}_{sph}(\beta)$ is invariant under translations in the
direction $u$. Now the idea is to minimize $\widetilde{S}_{sph}(\beta)$ restricted
to $\sum_v\beta_v=0$. This method has proven to be amazingly
powerful for our computations. In particular, it was used to produce branched orthogonal ring patterns
in the sphere and the examples in Section \ref{sec:Rt_examples}.

For $q=1$ the two radii, $R_v$ and $r_v$, of a ring coincide and we obtain spherical orthogonal circle patterns. The spherical radii of circles can also be expressed using $\beta$-variables, which we denote by $\beta^0$,  related in the following way 
\begin{equation*}
	\tanh \beta^0_v = \cos R_v, \ \frac{1}{\cosh \beta^0_v}=\sin R_v.
\end{equation*} 
The corresponding variation principle for orthogonal circle patterns was suggested in 
\cite{springborn2003variational} and used in \cite{BHS_2006}. 
The global existence and uniqueness of spherical orthogonal circle pattern for a given polytopal cell decomposition of the sphere is ensured by Koebe's Theorem, see e.g. \cite{BHS_2006}.

Factorizing a reflectionally symmetric cell decomposition by its symmetry groups one obtains existence and uniqueness results for circle patterns on the corresponding factor. Further, solutions $\beta^0$ of a Dirichlet or Neumann boundary value problem for circle patterns (corresponding to $q=1$) allow small deformations to orthogonal ring patterns with the same combinatorial and boundary data (corresponding to $q=1-\epsilon$ for small $\epsilon$). In \cite{bobenko2024rings} this fact was proven for rigid circle patterns, i.e., the ones with a non-degenerate Hessian 
\begin{equation}
	\label{eq:Rt_hessian}
	\det \left( \frac{\partial^2 S_{sph}}{\partial\beta^0_{v}\partial\beta^0_{v'}}\right) \neq 0. 
\end{equation}
An example is shown in Figure \ref{Fig:Rt_orp_deformation}.

We will use this deformation of circle patterns to ring patterns in Section \ref{sec:Rt_discrete_cmc_surfaces_as_deformations_of_minimal_surfaces} to construct discrete cmc surfaces with small mean curvature as deformations of discrete minimal surfaces.

\section{Constructing cmc surfaces in $\Rt$}
\label{sec:Rt_constructing_discrete_s_isothermic_cmc_surfaces}

How does one construct a discrete S$_1$-cmc surface analogous to a particular continuous cmc surface under investigation? In this section we outline the general method for
doing so. An analogous construction scheme for discrete S$_1$-minimal surfaces is described in \cite{BHS_2006}. Figure \ref{Fig:Rt_constructing} illustrates the construction of a discrete Schwarz P cmc surface. The difficult part here is to construct the ring
pattern corresponding to the combinatorics of curvature lines and boundary conditions, obtained from the smooth surface (paragraphs 1 and 2 below). The remaining steps, constructing the two-sphere Koebe net and the corresponding pair of S$_1$-cmc surfaces (paragraphs 3 and 4) are rather direct. Note that we only consider simply connected domains.\\
\begin{figure}[t]
	\setstretch{.7} 
	\centering
	
	\begin{minipage}{.3\linewidth}
		\includegraphics[width=\linewidth]{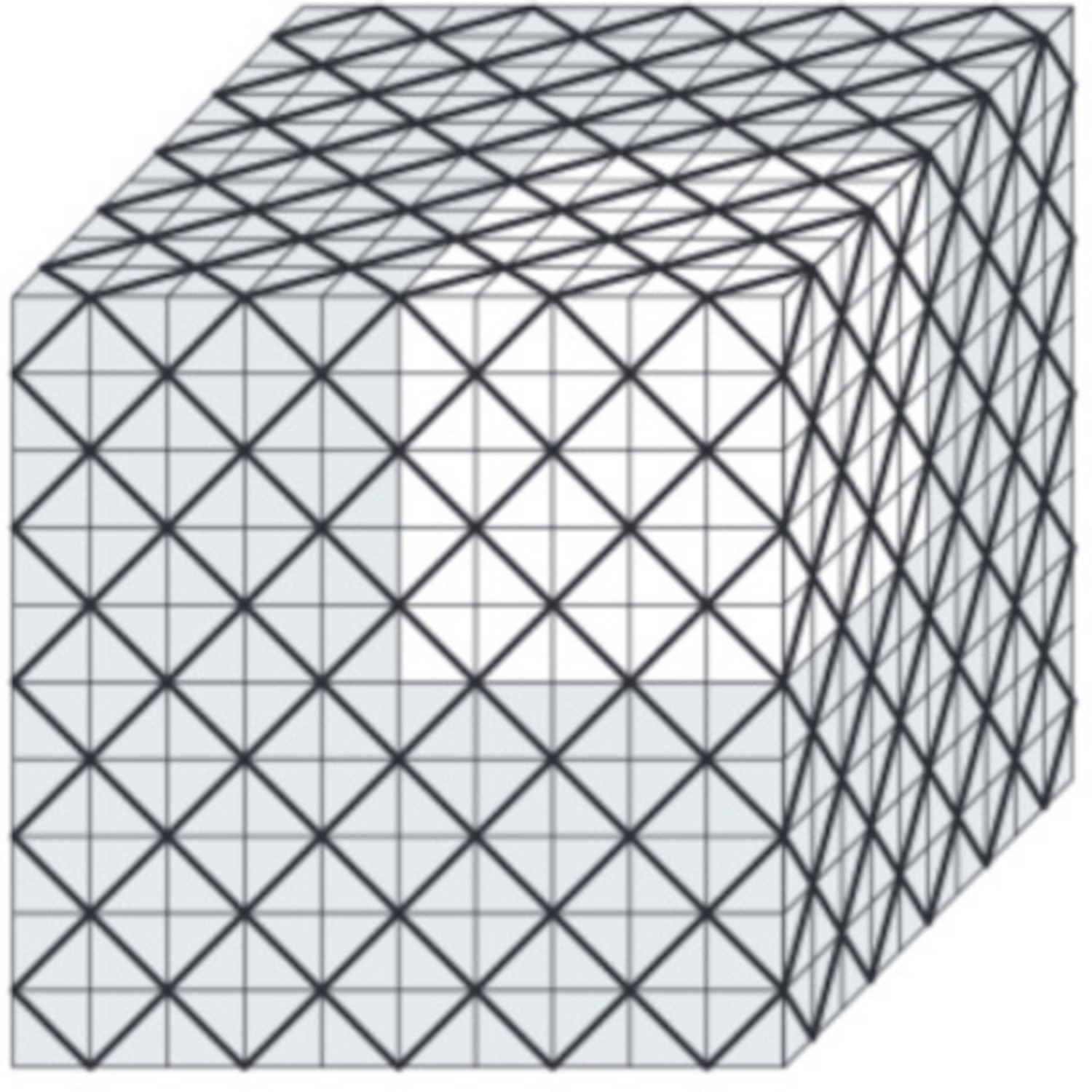} 
		
		\centering{\tiny(1) Combinatorics of the  Gauss image of curvature lines. The fundamental piece is highlighted.}
	\end{minipage}{\Large $ \ \rightarrow$}
	\begin{minipage}{.3\linewidth}
	%	\vspace{.3cm}
		\includegraphics[width=1\linewidth]{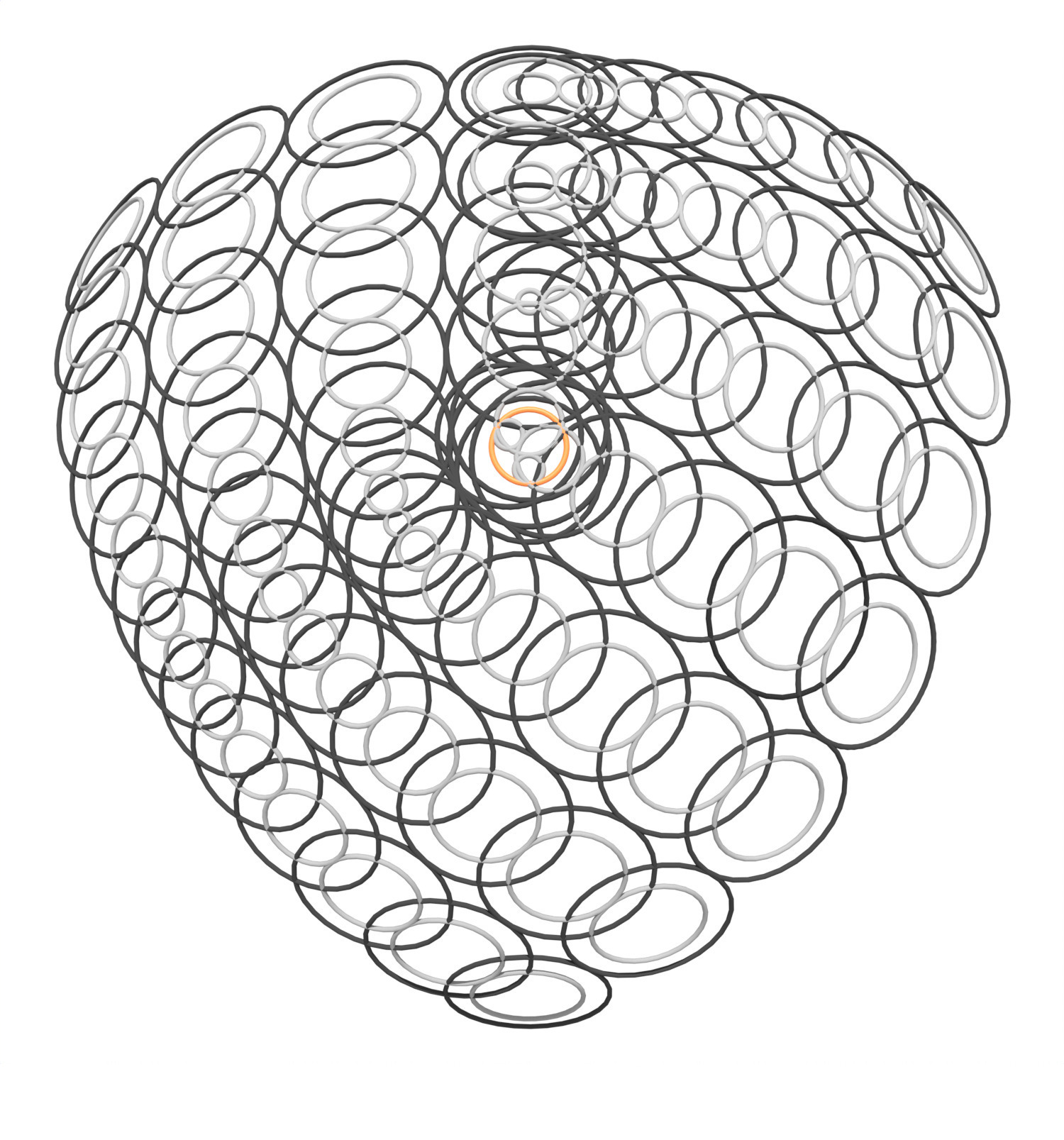}
	%	\vspace{.3cm}
		\centering
		 { \tiny  (2) Ring pattern (one family) of a fundamental piece.}
	\end{minipage}\\
\hspace{-.6cm}
{\Large $ \ \rightarrow$}
\begin{minipage}{.3\linewidth}
	\includegraphics[width=\linewidth]{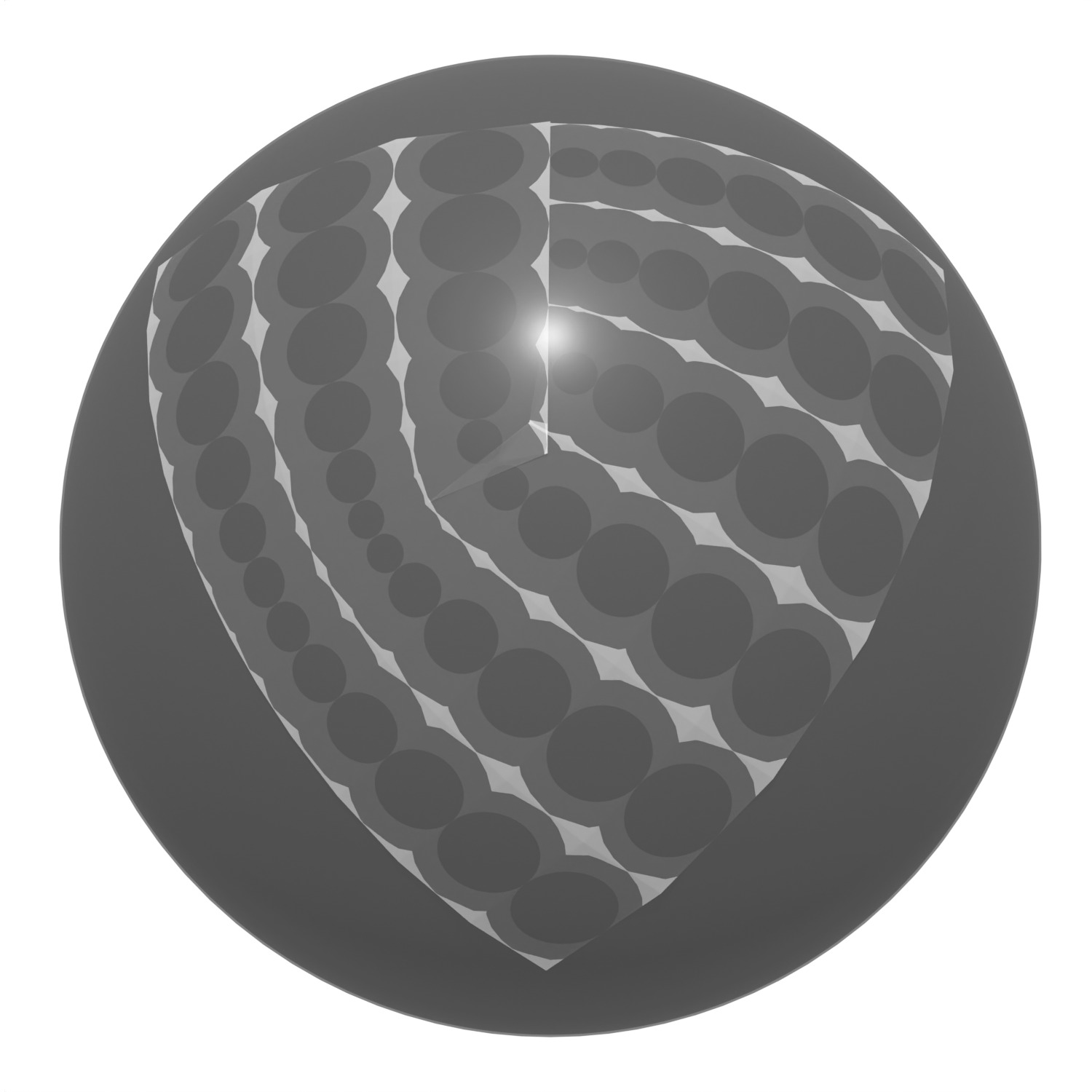}
	\centering
	{\tiny \centering (3) Two-sphere Koebe Q-net. }
\end{minipage}{\Large $ \ \rightarrow$}
\begin{minipage}{.3\linewidth}
	\includegraphics[width=\linewidth]{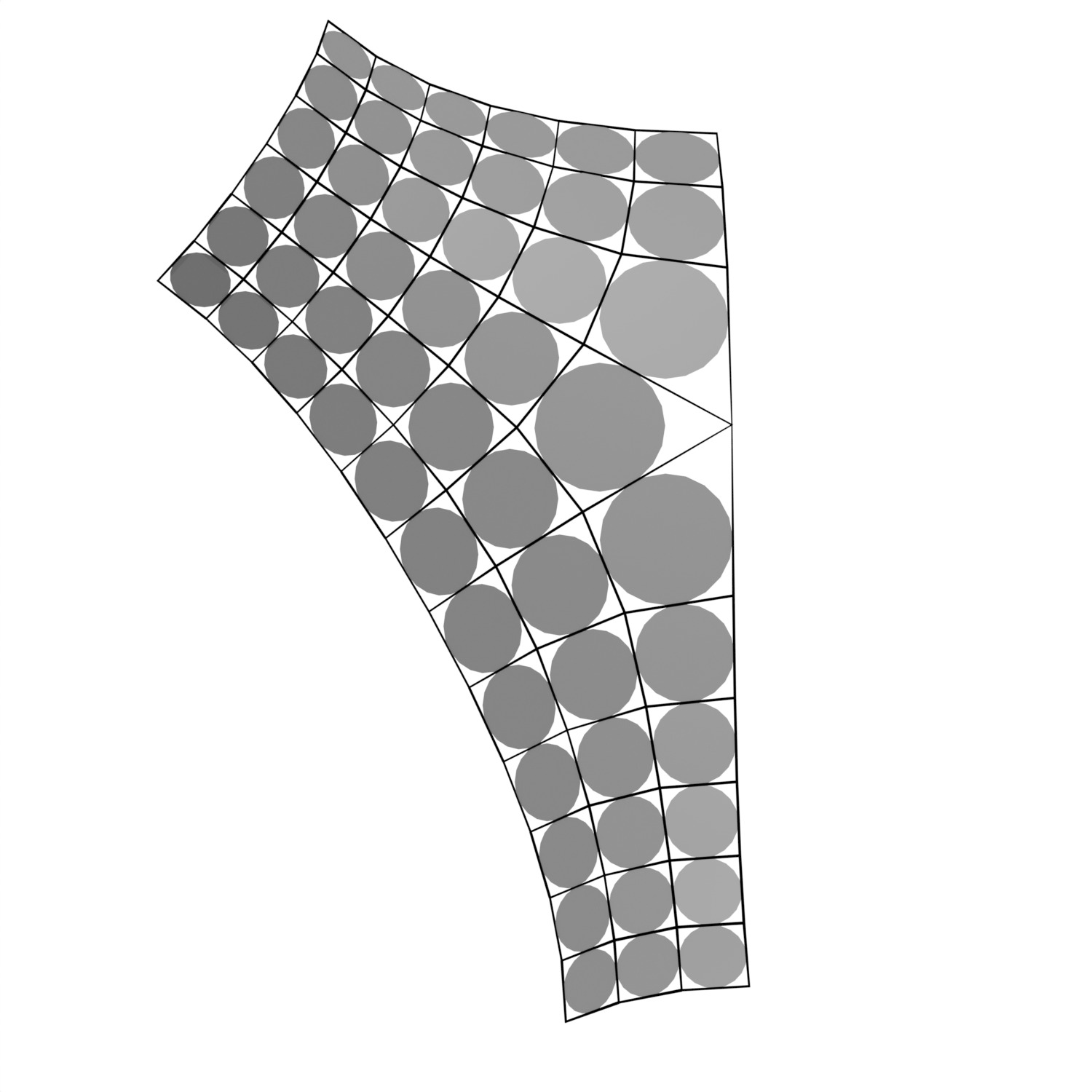}
	\centering
	{\tiny \centering (4) S$_1$-cmc surface. }
\end{minipage}
	\caption{Construction of a discrete S$_1$-cmc Schwarz P surface.}
	\label{Fig:Rt_constructing}
\end{figure}
\noindent \textbf{1. Investigate the Gauss image of the curvature lines.\\} The Gauss map
of the continuous cmc surface maps its curvature lines onto the sphere. We start with
a qualitative picture of this image, which is a quad graph immersed in the sphere. Here the choice of number of curvature lines corresponds to a choice of different levels of refinement of the discrete surface. Umbilics must be white vertices. Generically, the interior vertices are of degree four.
Exceptional vertices correspond to ends and umbilic points of the continuous
cmc surface. For example, in Figure \ref{Fig:Rt_constructing} (1), the corners of the cube are exceptional.\\
\textbf{
2. Construct the ring pattern.}\\ Next, from the quad graph, a parameter $q$, and boundary conditions, we construct a ring pattern by solving the corresponding boundary value problem. For suitable initial data, minimizing the reduced functional \eqref{eq:Rt_Ssph_reduced} determines the corresponding ring pattern. White vertices correspond to centers of rings, black vertices to touching points. Figure \ref{Fig:Rt_constructing} (2) shows the ring pattern around an umbilic point. Note that in this example, the Gauss map forms a double covering of the sphere, of which only one layer is shown in the figure.\\
\textbf{
3. Construct the dual two-sphere Koebe nets.}\\ Now we construct
the pair of dual two-sphere Koebe nets from the ring pattern. The white vertices labeled by \textcircled{s} and \textcircled{c} correspond to the vertices of a pair of dual two-sphere Koebe nets  $k^s$ and $k^c$. These vertices correspond to the spheres and the orthogonal circles of the S$_1$-cmc surface, respectively.\\
\textbf{
4. Construct the discrete S$_1$-cmc surface from its Gauss map.}\\
In Theorem \ref{Rt:Thm_s_cmc_koebe} we described how an S$_1$-cmc pair generates a two-sphere Koebe net as its Gauss map. Here we reverse this construction and show how to recover an S$_1$-cmc pair from its Gauss map. The essential part of this construction is to determine the radii of vertex spheres and face circles for the surfaces in terms of data from the ring pattern, or the Koebe nets. To determine these radii we will consider the central extension of a two-sphere Koebe net $k^s$, i.e., the map $k_\boxplus: V(\mathcal{S}) \rightarrow \Rt$,  such that 
	\begin{equation}
	\label{eq:central_extension_koebe}
	\begin{aligned}	
	 V_{\text{\textcircled{s}}}  \ni v_s &\mapsto \text{vertices } k_{v_s} \text{ of } k^s;\\
	V_{\text{\textcircled{c}}} \ni v_c  &\mapsto \text{face centers } \tilde{k}_{v_c} \text{ of }k^s;\\
	V_{b} \ni b &\mapsto \text{tangent points } t_b^\pm \text{ of } k^s; 
	\end{aligned}
	\end{equation}
see Figure \ref{Fig:Rt_fundamental_kite} (left).
The center $\tilde{k}_{v_c}$ of a face  $[k_{v_{s_1}}, k_{v_{s_2}}, k_{v_{s_3}}, k_{v_{s_4}}]$ is the \emph{properly scaled} dual vector, i.e.
\begin{align}
	\label{eq:face_center_koebe}
	\tilde{k}_{v_c} \perp [k_{v_{s_1}}, k_{v_{s_2}}, k_{v_{s_3}}, k_{v_{s_4}}], \quad \tilde{k}_{v_c} \in [k_{v_{s_1}}, k_{v_{s_2}}, k_{v_{s_3}}, k_{v_{s_4}}].
\end{align}

Consider a fundamental piece $[k_{v_s}, t_b^-, \tilde{k}_{v_c}, t_b^+]$ of $k_\boxplus$, consisting of one vertex, one face center and two touching points. In this section we consider the case shown in Figure \ref{Fig:Rt_fundamental_kite} (left), when this quadrilateral is embedded. Its edge lengths can be expressed in terms of the radii of the corresponding orthogonal ring pattern, see Figure \ref{Fig:Rt_Koebe_2d_vertex},
\begin{equation}
	\label{eq:Rt_koebe_length}
	\begin{aligned}	
		||t_b^- - k_{v_s}|| &=  \sqrt{q} \tan(R_{v_s}), \ \
		||\tilde{k}_{v_c} - t_b^-|| = \sqrt{q} \sin(r_{v_c}), \\
		||t_b^+ - \tilde{k}_{v_c}|| &=  \frac{1}{\sqrt{q}} \sin(R_{v_c}), \ \
		||k_{v_s} - t_b^+|| = \frac{1}{\sqrt{q}} \tan(r_{v_s}).
	\end{aligned}
\end{equation}
Note that the embededness of the quadrilateral corresponds to  $r_{v_c} >0$ and  $r_{v_s} > 0$.
If the two-sphere Koebe net is the Gauss map of an S$_1$-cmc pair then the lengths \eqref{eq:Rt_koebe_length} are related to the vertex sphere radii $d_{v_s}, d_{v_s}^*$ and the face circle radii $d_{v_c}, d_{v_c}^*$ of the  S$_1$-cmc pair such that
\begin{equation}
	\label{eq:Rt_koebe_length_2}
	\begin{aligned}	
		||t_b^- - k_{v_s}|| &= d_{v_s} + d_{v_s}^*, \ \
		||\tilde{k}_{v_c} - t_b^- || = d_{v_c} - d_{v_c}^*,\\
		||t_b^+ - \tilde{k}_{v_c}|| &=  d_{v_c} + d_{v_c}^*, \ \
		||k_{v_s} - t_b^+|| = d_{v_s} - d_{v_s}^*,\\
	\end{aligned}
\end{equation}
see Figure  \ref{Fig:Rt_fundamental_kite} (right). Here we use the parallelity of the primal and dual edges. A combination of \eqref{eq:Rt_koebe_length} and \eqref{eq:Rt_koebe_length_2} uniquely determines the radii of vertex spheres and face circles:
	\begin{equation}
	\label{eq:Rt_sphere_radii_from_koebe}
	\begin{aligned} 
		d_{v_s} &= \frac{1}{2}
		(\sqrt{q} \tan{R_{v_s}} 
		+ \frac{1}{\sqrt{q}}\tan{r_{v_s}}), \ \ d_{v_c} = \frac{1}{2}
		(\frac{1}{\sqrt{q}} \sin{R_{v_c}} 
		+ \sqrt{q} \sin{r_{v_c}}) \\
		d_{v_s}^* &= \frac{1}{2}
		(\sqrt{q} \tan{R_{v_s}} 
		- \frac{1}{\sqrt{q}}\tan{r_{v_s}}), \ \ d_{v_c}^* =  \frac{1}{2}
		(\frac{1}{\sqrt{q}} \sin{R_{v_c}} 
		- \sqrt{q} \sin{r_{v_c}}).
	\end{aligned}
\end{equation} 
Equivalently, in terms of the $\beta$-variables \eqref{eq:Rt_jef_uniformization} the identities \eqref{eq:Rt_koebe_length} can be expressed as 
\begin{equation}
	\label{eq:Rt_koebe_length_jef}
	\begin{aligned}	
		||t_b^- - k_{v_s}|| &=  \frac{1}{\sqrt{q}} \frac{\dn \beta_{v_s}}{\sn \beta_{v_s}}, \ \
		||\tilde{k}_{v_c} - t_b^-|| = \sqrt{q} \cn \beta_{v_c}, \\
		||t_b^+ - \tilde{k}_{v_c}|| &=  \frac{1}{\sqrt{q}} \dn \beta_{v_c}, \ \
		||k_{v_s} - t_b^+|| = \frac{1}{\sqrt{q}} \frac{\cn \beta_{v_s}}{\sn \beta_{v_s}}.
	\end{aligned}
\end{equation}
The radii \eqref{eq:Rt_sphere_radii_from_koebe} become
\begin{equation}
	\begin{aligned}
	\label{eq:Rt_radii}
	d_{v_s}&= 
		\frac{1}{2}\frac{1}{\sqrt{q}}
		\left(
		\frac{\dn \beta_{v_s} + \cn \beta_{v_s}}
		{\sn \beta_{v_s}}
		\right),  \ \
		d_{v_c} = 
		\frac{1}{2}\frac{1}{\sqrt{q}}
		\left(\dn \beta_{v_c} + q \cn \beta_{v_c}
		\right),
		\\
			d^*_{v_s}&= 
		\frac{1}{2}\frac{1}{\sqrt{q}}
		\left(
		\frac{\dn \beta_{v_s} - \cn \beta_{v_s}}
		{\sn \beta_{v_s}}
		\right),  \ \ 
		d^*_{v_c} = \frac{1}{2}\frac{1}{\sqrt{q}}
		\left(\dn \beta_{v_c} - q \cn \beta_{v_c}
		\right).
\end{aligned}
\end{equation}

We denote these radii by $d_v$ and $d_v^*$, where the choice $v \in V_{\text{\textcircled{s}}}$ or $v \in V_{\text{\textcircled{c}}}$ determines whether
we are considering the vertex sphere radii or the face circle radii.

The identities $\eqref{eq:Rt_radii}$ determine the length of all edges of the (central extension of the) S$_1$-cmc pair. The directions are determined by the edge parallelism to the two-sphere Koebe net. The following theorem ensures that the resulting surfaces with these edge lengths and directions exist.

\begin{figure}
	\centering	
	
	\begin{minipage}{.49\linewidth}
		\centering
		\begin{overpic}[width=.85\linewidth, ]{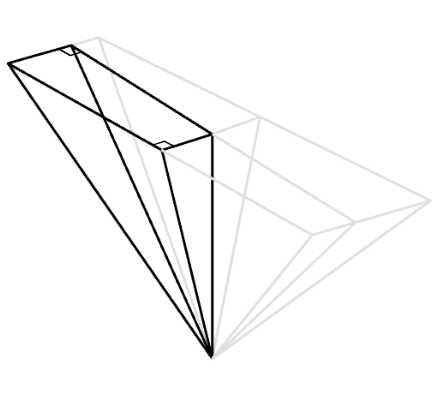}
			\put(-7, 78){$k_{v_s}$}
			\put(14, 87){$t_b^+$}
			\put(30, 55){$t_b^-$}
			\put(49, 63){$\tilde{k}_{v_c}$}
		\end{overpic}
	\end{minipage}
	\hfill
	\begin{minipage}{.49\linewidth}
		\centering
		\begin{overpic}[width=.95\linewidth]{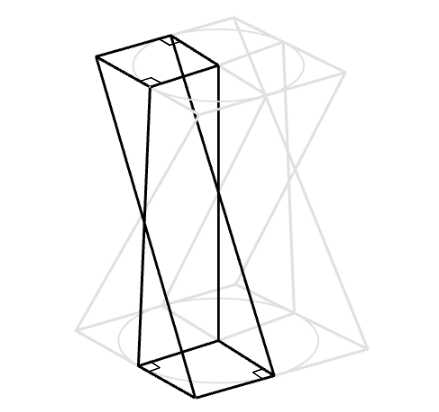}
			\put(32, 3){$d_{v_s}$}
			\put(53, 1){$d_{v_s}$}
			
			\put(53.5, 14){$d_{v_c}$}
			\put(41, 17){$d_{v_c}$}
			
			\put(25, 72){$d_{v_s}^*$}
			\put(23, 85.2){$d_{v_s}^*$}
			
			\put(36, 70){$d_{v_c}^*$}
			\put(44, 83){$d_{v_c}^*$}
			
		\end{overpic}
	\end{minipage}
	\caption{Left: A fundamental piece of a two-sphere Koebe net with a pair of opposite right angles. Right: Corresponding primal and dual fundamental kites and their normal vectors in a fundamental S$_1$-cmc hexahedron. The labels denote the radii of the vertex spheres, $d_{v_s}$ and $d^*_{v_s}$, and the radii of the face circles, $d_{v_c}$ and $d^*_{v_c}$.}
	\label{Fig:Rt_fundamental_kite}
\end{figure}

\begin{theorem}
	\label{Thm:Rt_Koebe_to_cmc}
	Let $k_\boxplus: V(\mathcal{S}) \rightarrow \Rt$
	be the central extension \eqref{eq:central_extension_koebe} of a two-sphere Koebe net. Then the $\Rt$-valued discrete one-form $\partial c_\boxplus$  defined by
	\begin{align}
		\label{Rt_one_form_1}
		\partial_{(v, b)}  c_\boxplus = d_v  \ \frac{\partial_{(v, b)}   k_\boxplus }{||\partial_{(v, b)}   k_\boxplus||},
	\end{align}
	and the the $\Rt$-valued discrete one-form $\partial c^*_\boxplus$ defined by 
	\begin{align}
		\label{Rt_one_form_2}
		\partial_{(v, b)}  c^*_\boxplus =  \pm \ d^*_v \ \frac{\partial_{(v, b)}   k_\boxplus }{||\partial_{(v, b)}   k_\boxplus||},
	\end{align} 
	with $d_v, d_v^*$ given by \eqref{eq:Rt_radii}, are exact. The signs $(+)$ and $(-)$ in \eqref{Rt_one_form_2} are chosen for horizontal and vertical edges, respectively.
	The integration of \eqref{Rt_one_form_1} and \eqref{Rt_one_form_2} defines the central extension of two surfaces $c$ and $c^*$, which, when appropriately placed, form the center nets of an S$_1$-cmc pair $s, s^*$ with the Gauss map $k_\boxplus$. The vertex sphere radii and face circle radii of $s$ and $s^*$ are given  by \eqref{eq:Rt_radii}. The S$_1$-cmc pair with the Gauss map $k_\boxplus$ is unique up to translation. 
\end{theorem}
\begin{proof} We give a proof for the case of embedded quadrilaterals, see Figure \ref{Fig:Rt_fundamental_kite} (left). The cases of non-embedded quadrilaterals, when $r_{v_c} < 0$ or $r_{v_s} < 0$, can be considered similarly, formulas \eqref{eq:Rt_radii}, \eqref{Rt_one_form_1}, \eqref{Rt_one_form_2} hold in these cases as well. 
	
With  \eqref{eq:Rt_koebe_length_jef} and by defining unit normal vectors $u_1, u_2, u_3, u_4$ in counterclockwise order along the edges of a fundamental piece of the Koebe net, the closeness of $\partial_{(v, b)}   k_\boxplus$ can be expressed as
	\begin{align}
		\label{eq:Rt_quad}
		\frac{1}{\sqrt{q}} \frac{\dn \beta_{v_s}}{\sn \beta_{v_s}} 
		u_1 + 
		\sqrt{q} \cn \beta_{v_c}
		u_2 + 
		\frac{1}{\sqrt{q}} \dn \beta_{v_c}
		u_3 + 
		\frac{1}{\sqrt{q}} \frac{\cn \beta_{v_s}}{\sn \beta_{v_s}}
		u_4 =0
	\end{align}
	Reflecting the fundamental piece in the angle bisector between $u_1$ and $u_4$ at $k_{v_s}$ defines a quadrilateral with edges parallel to the fundamental piece, satisfying
	\begin{align}
		\label{eq:Rt_reflected_quad}
		\frac{1}{\sqrt{q}} \frac{\cn \beta_{v_s}}{\sn \beta_{v_s}}
		u_1 + 
		\frac{1}{\sqrt{q}} \dn \beta_{v_c}
		u_2 + 
		\sqrt{q} \cn \beta_{v_c}
		u_3 + 
		\frac{1}{\sqrt{q}} \frac{\dn \beta_{v_s}}{\sn \beta_{v_s}} 
		u_4 =0.
	\end{align}
	The half sum and difference of the identities \eqref{eq:Rt_quad}  and \eqref{eq:Rt_reflected_quad} yield two additional closed quadrilaterals with parallel edges. In particular, with \eqref{eq:Rt_radii} it is
	\begin{equation*}
		\begin{aligned}	
			\frac{1}{2} \left( \eqref{eq:Rt_quad}  + \eqref{eq:Rt_reflected_quad} \right) &= 
			d_{v_s}
			u_1 + 
			d_{v_c}
			u_2 + 
			d_{v_c}
			u_3 + 
			d_{v_s}
			u_4 =0, \\ \ \ 
			\frac{1}{2} \left( \eqref{eq:Rt_quad}  - \eqref{eq:Rt_reflected_quad} \right) &= 
			d^*_{v_s}
			u_1 -
			d^*_{v_c}r
			u_2 + 
			d^*_{v_c}
			u_3 -
			d^*_{v_s}
			u_4 =0.
		\end{aligned}
	\end{equation*}
	This proves that the one-forms \eqref{Rt_one_form_1} and 
	\eqref{Rt_one_form_2} are closed, and their exactness follows. The minus signs in the second equation represent the change of orientation for vertical edges in the Christoffel dual. 
\end{proof}

\begin{figure}[h]
	\centering	
	\begin{minipage}{.3\linewidth}
		\centering
		\includegraphics[width=\linewidth]{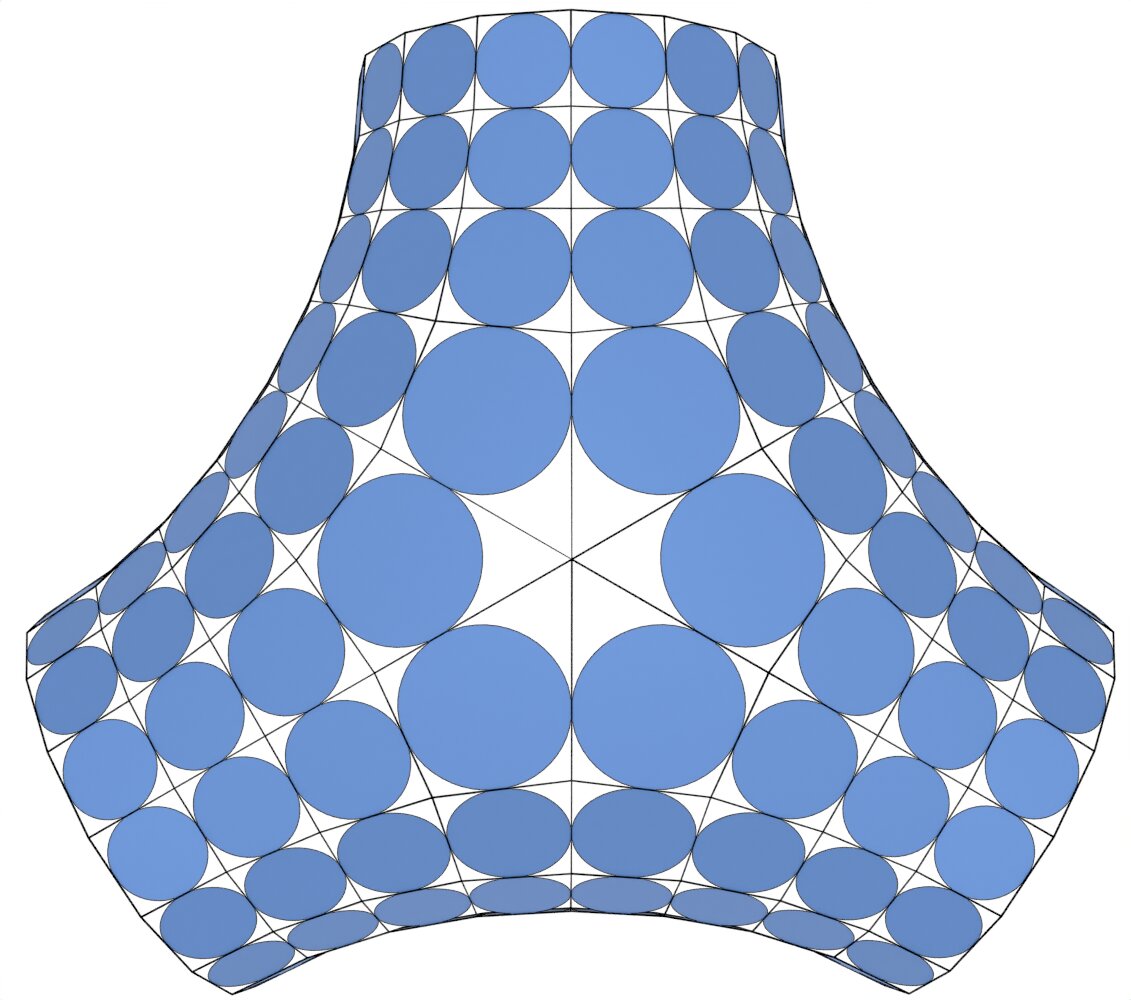}
	\end{minipage}
	\hspace{1cm}
	\begin{minipage}{.3\linewidth}
		\centering
		\vspace{.35cm}
		\includegraphics[width=\linewidth]{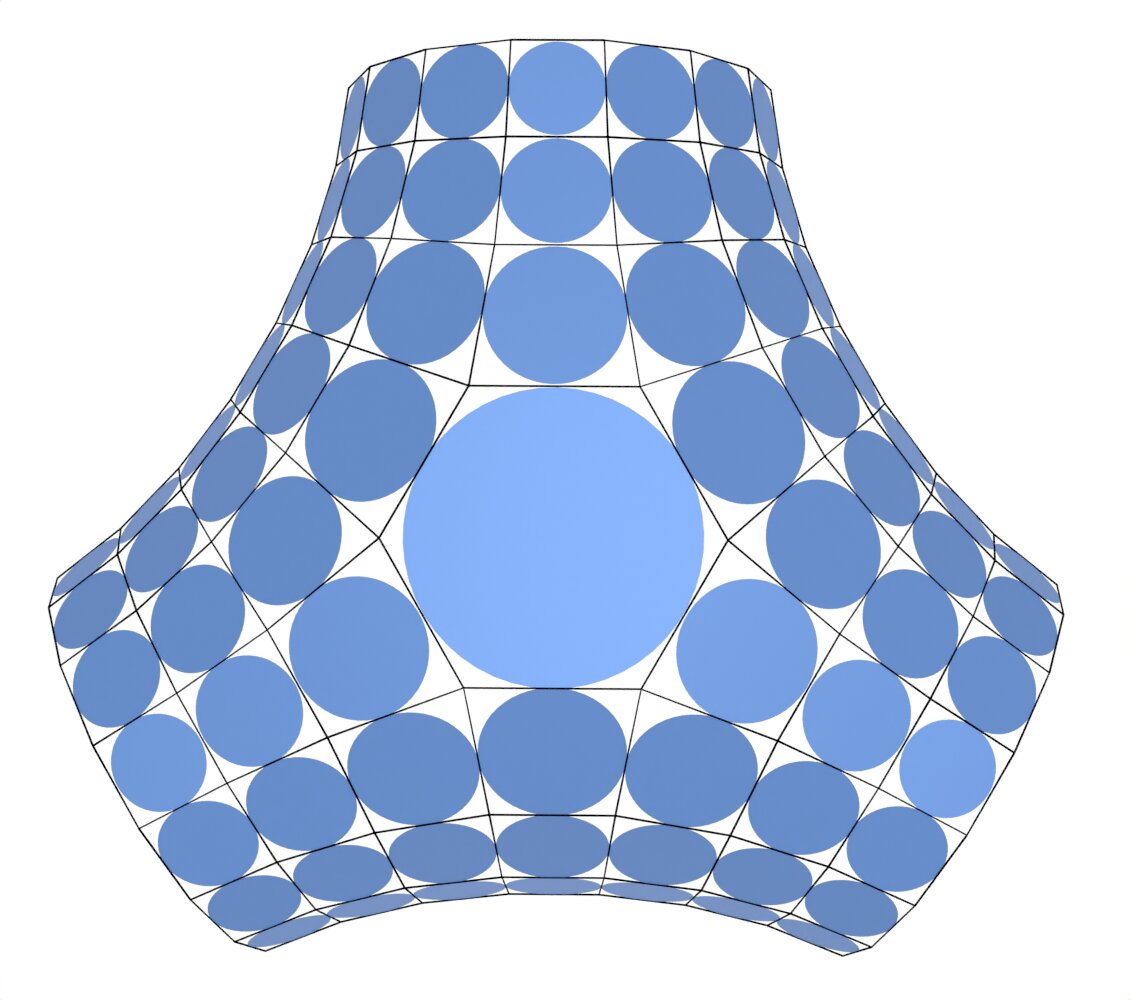}
	\end{minipage}
	\caption{A pair of combinatorially dual S$_1$-isothermic surfaces constructed using the method presented in Section \ref{sec:Rt_constructing_discrete_s_isothermic_cmc_surfaces}. Corresponding dual edges are orthogonal. The umbilic point is represented by a vertex of even valence greater than four (left), or by a face of even valence greater than four (right).}
	\label{Fig:Rt_umbilcs}
\end{figure}

By interchanging the roles of $V_\text{\textcircled{c}}$ and $V_\text{\textcircled{s}}$ one obtains another, closely related, S$_1$-cmc surface, see Figure \ref{Fig:Rt_umbilcs}. Circles of such S$_1$-isothermic surfaces correspond to spheres of the other. The resulting surface pair is a principle binet in the sense of \cite{affolter2024principalbinets}.

\section{Minimal surfaces in the limit}
\label{sec:Rt_discrete_cmc_surfaces_as_deformations_of_minimal_surfaces}
Let $s$ and $s^*$ be an S$_1$-cmc pair with Gauss map $n = c^* - c$ and constant mean curvature $H=1$. 
Recall that a scaling of the surface pair, $s_\mu := \mu s, s^*_\mu := \mu s^*$, by a global factor $\mu$ changes the mean curvature to $H_\mu := \frac{1}{\mu}$ and the scaled center net $c^*_\mu$ becomes a normal offset surface of $c_\mu$ at distance $\frac{1}{H_\mu}$, 
\begin{align*}
	c_\mu^* = c_\mu + \frac{1}{H_\mu} n.
\end{align*}
For $\mu \rightarrow \infty$ the mean curvature $H_\mu$ goes to zero, $H_\mu \rightarrow 0$, and the distance $||c_\mu^*- c_\mu||=\mu$ goes to infinity. Let us investigate this limit in more detail.

In the limit $q\rightarrow 1$ a ring pattern deforms to a circle pattern. Any rigid circle pattern can be deformed to a one-parameter $\epsilon$-family of ring patterns, corresponding to $q= 1-\epsilon$ with small $\epsilon$ with the same combinatorics and boundary data, see Section \ref{sec:analytic_orp} and \cite{bobenko2024rings}. Let us consider the surfaces $s_\mu$ and $s_\mu^*$ with $\mu =\frac{1}{\epsilon}$ in the limit $\epsilon \rightarrow 0$. We will show that one of the surfaces converges to an  S$_1$-minimal surface while the dual surface converges to a Koebe net tangent to the unit sphere.

\begin{theorem}
	\label{Thm:Rt_limit}
	Consider a rigid spherical orthogonal circle pattern, i.e., with non-degenerate Hessian \eqref{eq:Rt_hessian}, and fixed boundary conditions, as well as its deformation to a spherical orthogonal ring pattern with $q=1-\epsilon$ and small $\epsilon$. Then there exists an $\epsilon$-family of associated S$_1$-cmc pairs, $\frac{1}{\epsilon}s$ and $\frac{1}{\epsilon}s^*$ with mean curvature $H = \epsilon$. In the limit $\epsilon \rightarrow 0$, $s$ converges to a classical Koebe net with the edges touching the unit sphere, and $\frac{1}{\epsilon}s^*$ converges to its dual S$_1$-minimal surface.
\end{theorem}
\begin{proof}
	To analyze the behavior of the surfaces, we examine how the vertex sphere radii \eqref{eq:Rt_radii} evolve in the limit $\epsilon \rightarrow 0$.
	We want to emphasize that in the following  the $\beta$-variables depend on the value of $q$ but for simplicity we will write $\beta_v$ instead of $\beta_v(q)$.
For $q \rightarrow 1$ the Jacobi elliptic functions degenerate into hyperbolic trigonometric \mbox{functions \cite{nist}}:
	\begin{align}
		\label{eq:Rt_jacobi_limit}
		\sn(z, q) \rightarrow \tanh(z), \quad \cn(z, q) \rightarrow \frac{1}{\cosh(z)}, \quad \dn(z, q) \rightarrow \frac{1}{\cosh(z)}.
	\end{align}

	\noindent {\scshape Primal radii}: For the scaled primal radii it is
	\begin{align*} 
		%\label{eq:Rt_primal_radii_limit}
		&\frac{1}{\epsilon}
		\lim_{q\rightarrow 1} 
		\frac{1}{2\sqrt{q}}
		\left(
		\frac{\dn(\beta_v, q)+ \cn(\beta_v, q)}
		{\sn(\beta_v, q)}
		\right)
		= 
		\frac{1}{\epsilon}\frac{1}{\sinh(\beta_v)}.
	\end{align*}
	We see that the primal radii go to infinity. However they remain finite for the normalized surface \mbox{$s = \epsilon \frac{1}{\epsilon}s$}.
	
	\noindent {\scshape Dual radii}: 
	To investigate the limit 
	\begin{equation}
		\label{eq:Rt_limit}
		\begin{aligned} 
			&\lim_{q\rightarrow 1}  
			\frac{1}{1-q}
			\frac{1}{2\sqrt{q}}
			\left(
			\frac{\dn(\beta_v, q)- \cn(\beta_v, q)}
			{\sn(\beta_v, q)}
			\right)
			= 
			\frac{1}{2\tanh(\beta_v)} \lim_{q\rightarrow 1}  
			\frac{\dn(\beta_v, q) - \cn(\beta_v, q)}{1-q}
		\end{aligned}
	\end{equation}
	we approximate $\cn$ and $\dn$  as follows
	\begin{equation}
		\label{eq:Rt_cn_dn_approx}
		\begin{aligned} 
		&\dn(\beta_v, 1-\epsilon) = \frac{1}{\cosh(\beta_v)} + \epsilon \ a_{\dn} + o(\epsilon),\\
		&\cn(\beta_v, 1-\epsilon) = \frac{1}{\cosh(\beta_v)} + \epsilon \ a_{\cn} +  o(\epsilon), \epsilon \rightarrow 0, 
			\end{aligned}
\end{equation}
with $a_{\cn}:= -\frac{\partial}{ \partial q}\cn(\beta_v, 1)$ and $a_{\dn}:= -\frac{\partial}{ \partial q}\dn(\beta_v, 1)$.
	Using the elementary identity $\cn^2 + \sn^2 = q^2\sn^2 + \dn^2$
	and thus
	\begin{align*}
		(\dn - \cn)(\dn + \cn)= \epsilon(q+1)\sn^2
	\end{align*}
	we find by comparing the sides in order $\epsilon$: 
	\begin{align*}
		a_{\dn}- a_{\cn} =\frac{\sinh^2(\beta_v)}{\cosh(\beta_v)}.
	\end{align*}
	Substitution into \eqref{eq:Rt_limit} yields
	\begin{align*}
		\frac{1}{2\tanh(\beta_v)}
		\lim_{q\rightarrow 1}  
		\left(
		\frac{\dn(\beta_v, q)- \cn \beta_v, q)}
		{1-q}
		\right)
		= 
		\frac{1}{2}
		\sinh(\beta_v).
	\end{align*}

	During the limiting process the pair of surfaces may only be scaled simultaneously in order to keep the cmc properties of the pair. 
	Afterwards the two surfaces can be treated separately. A global scaling of the primal surface in the limit allows to find a finite surface with vertex spheres of radii 
	$\frac{1}{\sinh(\beta_v)}$. This scaling preserves the Christoffel duality.

	In conclusion in the limit we obtain two S$_1$-isothermic surfaces, $s$ and $\frac{1}{\epsilon} s^*$, that are Christoffel dual with vertex sphere radii 
	\begin{align*}
		d_v &= \frac{1}{\sinh(\beta_v)} \quad
		\text{ and } \quad d_v^* =\frac{1}{2} \sinh(\beta_v).
	\end{align*}

	In the limit $q\rightarrow 1$ the spheres $S^2_\pm$ of the two-sphere Koebe net $n$ converge to the unit sphere, and $n$ becomes a classical Koebe net with all edges touching the unit sphere. Further, 
	\begin{align*}
		s = n + s^* \xrightarrow[\epsilon \  \rightarrow \ 0]{} n
	\end{align*}
	since $\frac{1}{\epsilon}s^*$ stays finite. Thus the surfaces $s$ and $\frac{1}{\epsilon}s^*$ indeed form a Koebe net and an S$_1$-minimal surface, and in the limit $\epsilon \rightarrow 0 $ we recovered the description of S$_1$-minimal surfaces from \cite{BHS_2006}.
\end{proof}

\section{Examples of cmc surfaces in $\Rt$}
\label{sec:Rt_examples}

We apply the method described in Section \ref{sec:Rt_constructing_discrete_s_isothermic_cmc_surfaces} to construct concrete examples of S-cmc surfaces in $\Rt$. 
\begin{figure}[h]
	\centering
	\includegraphics[width=.35\linewidth]{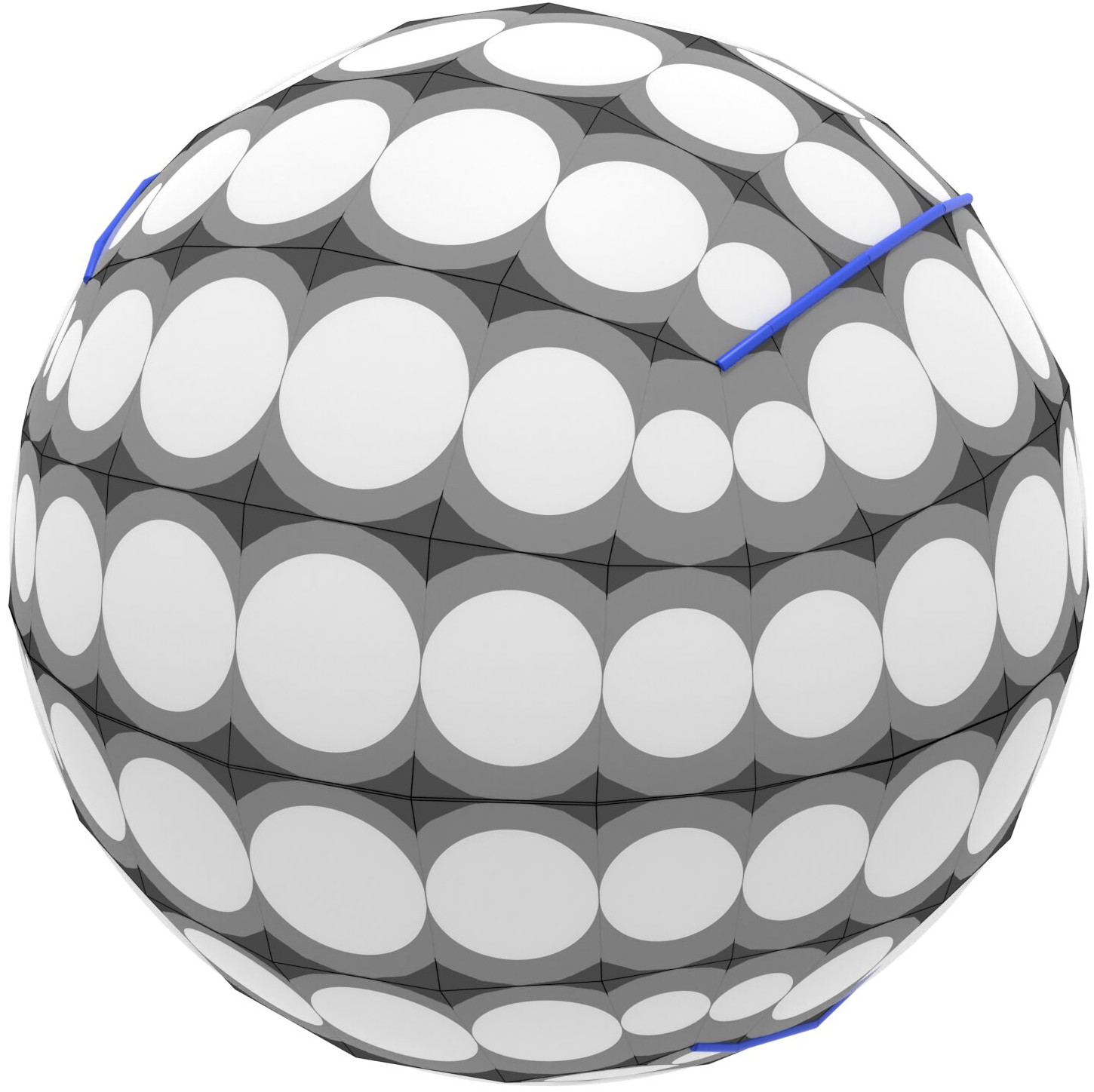}
	\caption{A two-sphere Koebe net forming a double cover of the sphere, one layer of the covering is shown. In the figure exceptional vertices are adjacent to three copies of a fundamental piece. The cuts of the covering, shown in blue, connect two exceptional vertices.}
	\label{Fig:Rt_Koebe_polyhedron}	
\end{figure}
We restrict the underlying S-quad graphs to combinatorial rectangles. Under suitable boundary conditions these combinatorial rectangles  will lead to fundamental pieces of reflectionally symmetric cmc surfaces. Umbilic points, or vertices with valence greater than four, can appear at the corners of these fundamental pieces after reflection.

\begin{figure}[t]
	\begin{minipage}{.325\linewidth}
		\includegraphics[width=\linewidth]{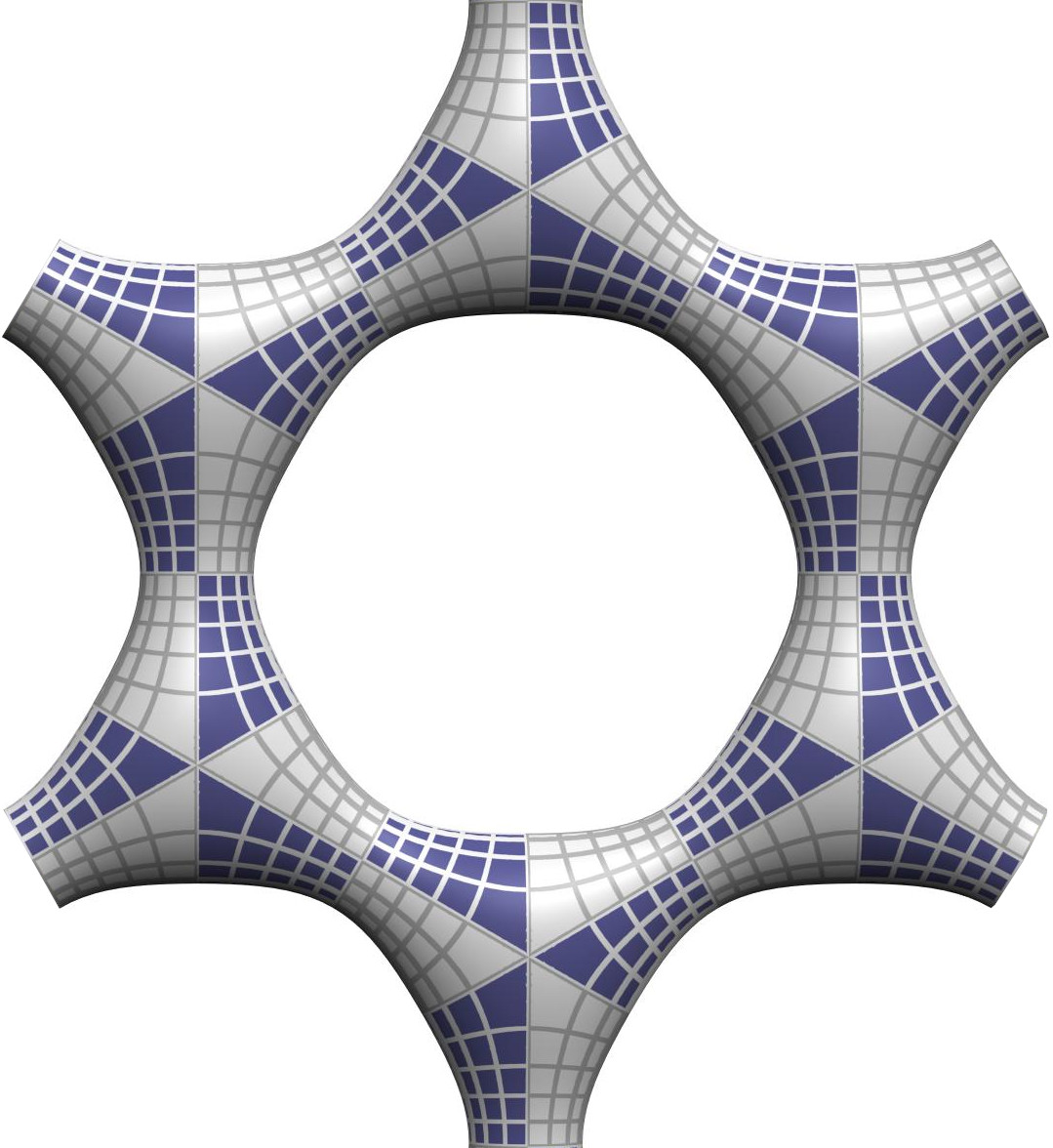}
	\end{minipage}
	\begin{minipage}{.325\linewidth}
		\includegraphics[width=\linewidth]{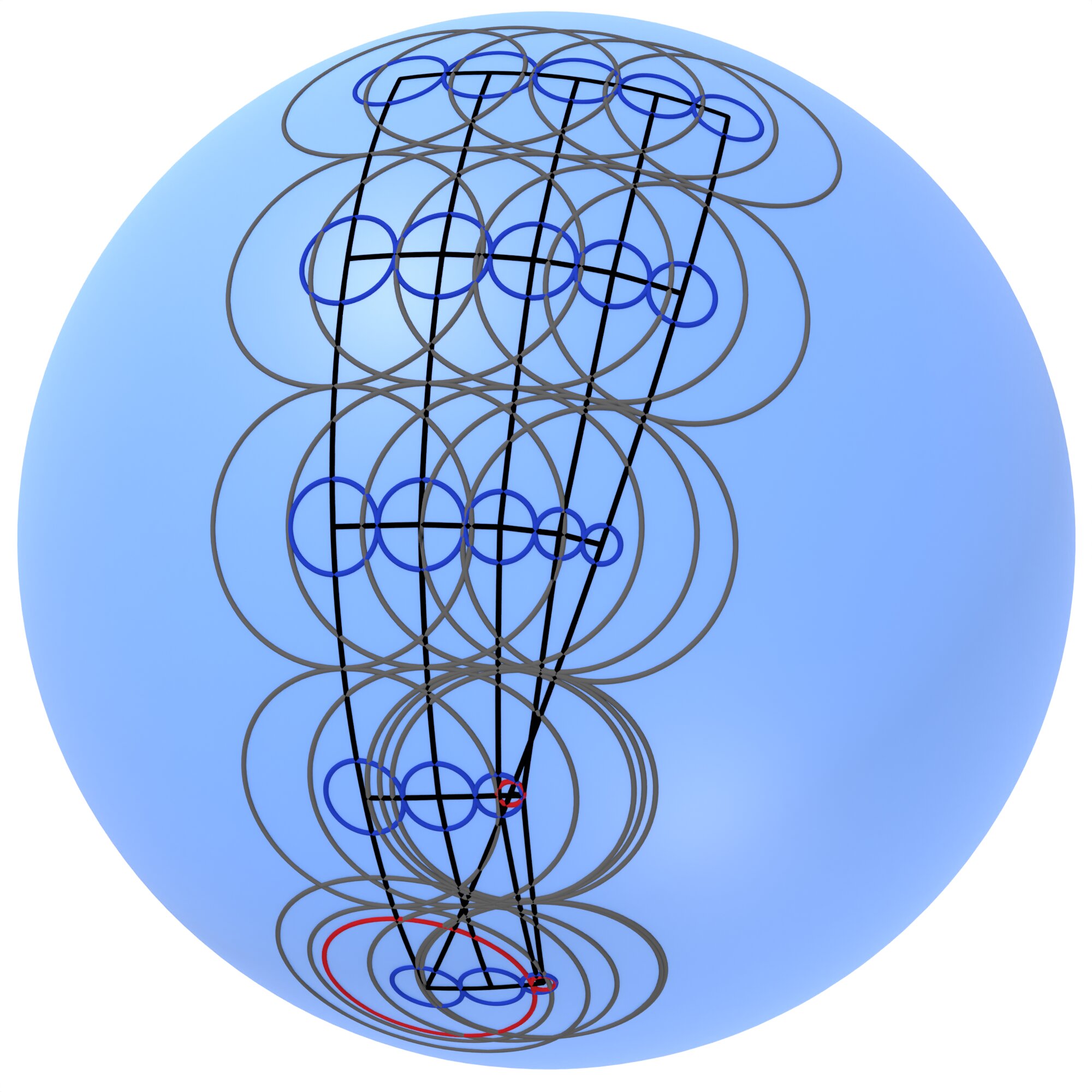}
	\end{minipage}
	\begin{minipage}{.325\linewidth}
		\includegraphics[width=\linewidth]{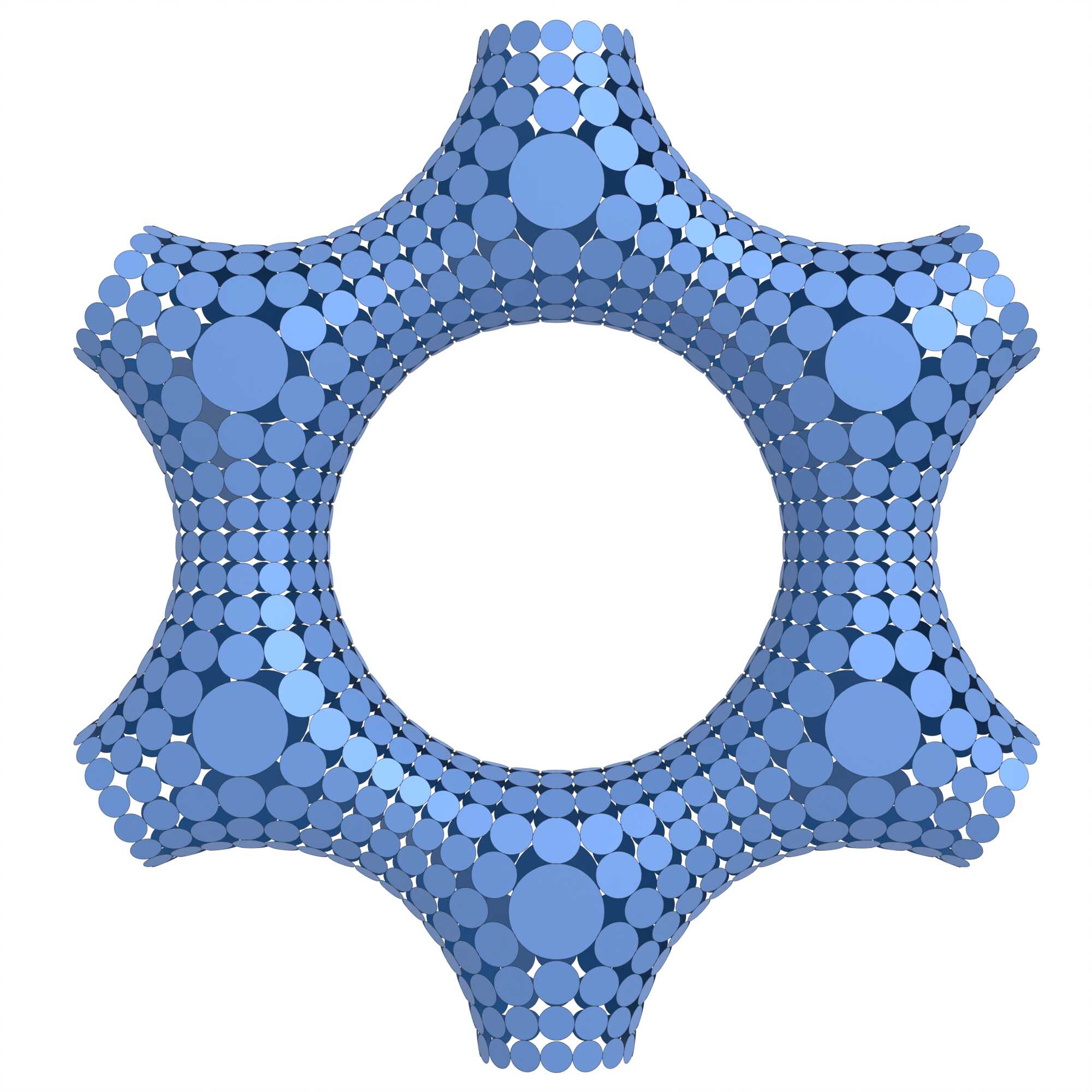}
	\end{minipage}
	
	\vspace{.5cm}
	
	\begin{minipage}{.325\linewidth}
		\includegraphics[width=\linewidth]{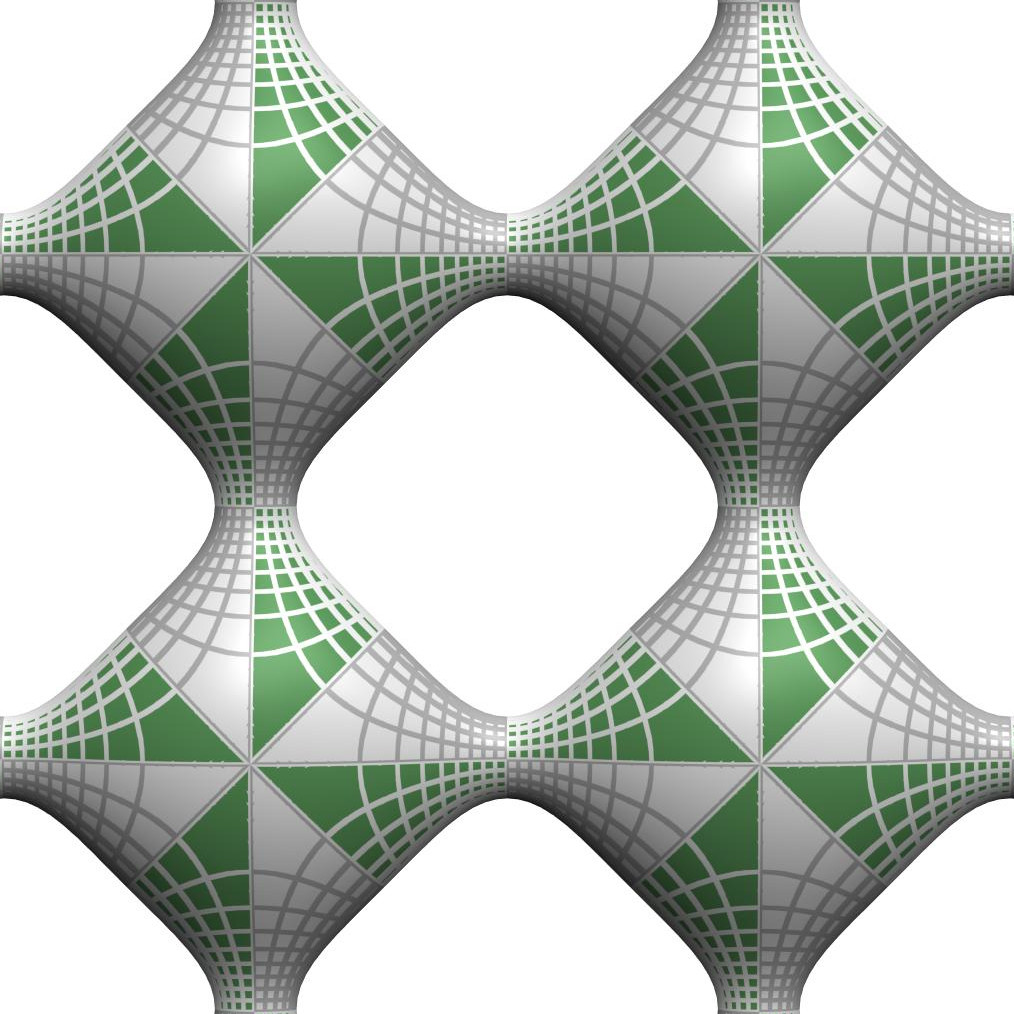}
	\end{minipage}
	\begin{minipage}{.325\linewidth}
		\includegraphics[width=\linewidth]{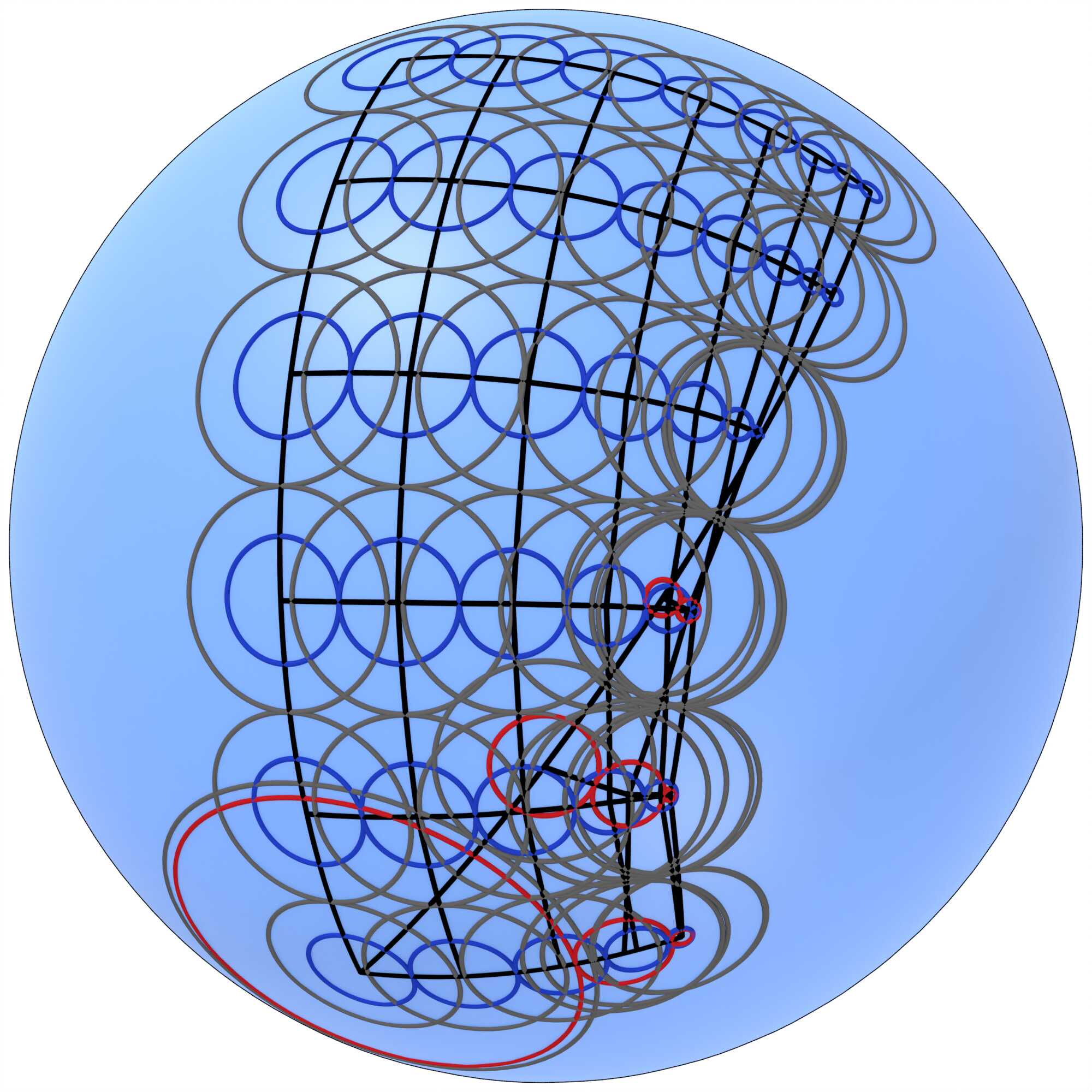}
	\end{minipage}
	\begin{minipage}{.325\linewidth}
		\includegraphics[width=\linewidth]{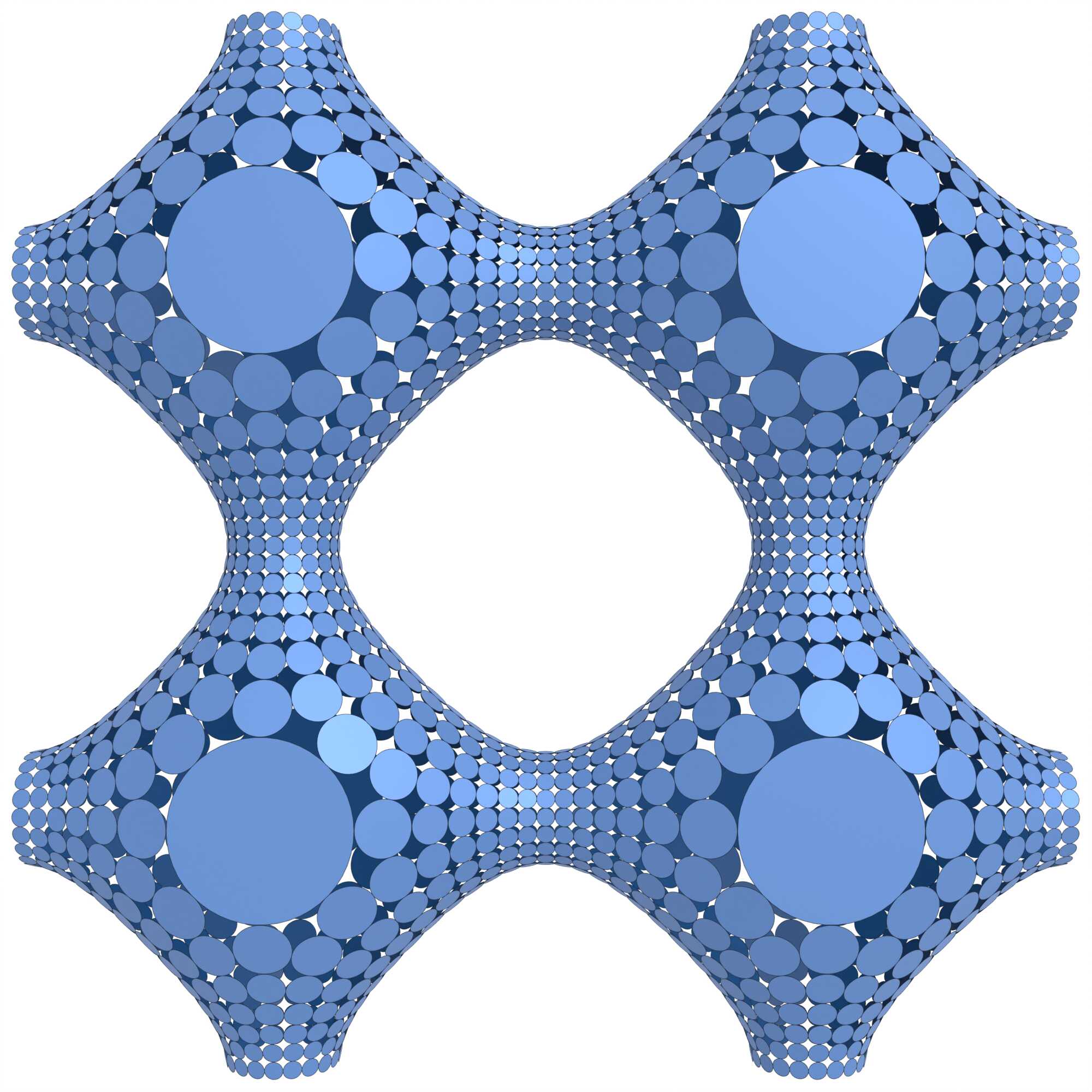}
	\end{minipage}
	\caption{The doubly periodic cmc surfaces $\psi(U_{2, 2})$ (top) and $\psi(U_{3, 3})$ (bottom) with hexagonal and quadrilateral lattice symmetry, respectively. The rows show the smooth cmc surfaces, the orthogonal ring patterns (only half of the rings are shown), and the discrete cmc surfaces. The shape of the orthogonal ring patterns is determined by the nominal angles at the corners, $\frac{\pi}{2}, \frac{2\pi}{3}, \frac{\pi}{2}, \frac{\pi}{2}$  and $\frac{\pi}{2}, \frac{\pi}{4}, \frac{\pi}{2}, \frac{\pi}{2}$, and the global parameter,  $q \approx 0.982889$ and $q \approx 0.991636$, respectively.}
	\label{Fig:Rt_Examples_doubly}
\end{figure}

Let us consider the case when a fundamental piece of a reflectionally symmetric smooth cmc surface is bounded by four curvature lines, each lying in a (reflection) plane, see \cite{bobenko2021constant}. The Gauss map along these curvature lines forms a spherical quadrilateral whose edges are segments of great circles. At corners that are not exceptional, the curvature lines and their Gauss image intersect at an angle $\frac{\pi}{2}$. At exceptional vertices, the curvature lines of the surface intersect at an angle $\frac{\pi}{n}$, their Gauss images at an angle $\pi - \frac{\pi}{n}$. These boundary angles and a sampling of curvature lines provide initial data for constructing a fundamental piece of an orthogonal ring pattern using the variational principle.

We choose the parameter $q$ so that the orthogonal ring pattern is bounded by the same spherical quadrilateral as the smooth Gauss map, ensuring that the ring pattern retains the same reflection symmetries. In this case we can construct closed examples of (reflectionally symmetric) two-sphere Koebe nets, see Figure \ref{Fig:Rt_Koebe_polyhedron}. Self intersections and double covers occur.

After constructing the fundamental piece, the corresponding reflections determine the complete discrete periodic cmc surface.

\subsection{Doubly periodic S$_1$-cmc surfaces}
\label{sec:Rt_Examples_doubly}
Cmc surfaces in  Euclidean three-dimensional space  are isometric to minimal surfaces in the three-dimensional sphere $S^3$.	This fact is known as the classical Lawson correspondence \cite{lawson1970complete}.  Lawson first used this approach to show the existence of two reflectionally symmetric, doubly periodic cmc surfaces in $\Rt$, denoted by $\psi(U_{2, 2})$ and $\psi(U_{3, 3})$, where $\psi$ denotes the immersion of the conjugate surfaces into $\Rt$. The isometric minimal surfaces in $S^3$ are the classical Lawson surfaces  $\xi_{2, 2}$ and $\xi_{3, 3}$, while $U_{2, 2}$ and $U_{3, 3}$ denote their universal Riemannian coverings. The integers refer to the symmetry group of the surface, and the fundamental pieces of $U_{2, 2}$ and $U_{3, 3}$ are obtained by solving the corresponding Plateau problem in $S^3$.

 The smooth cmc surfaces $\psi(U_{2, 2})$ and $\psi(U_{3, 3})$ have recently been constructed in \cite{bobenko2021constant} using the DPW method \cite{dorfmeister1998weierstrass}. They are shown, together with their discrete analogues and corresponding orthogonal ring patterns, in Figure \ref{Fig:Rt_Examples_doubly}. In these doubly periodic examples, the spherical quadrilateral bounding the fundamental piece of the Gauss map degenerates into a triangle.
 
We want to emphasize that the smooth surfaces on the left in Figure \ref{Fig:Rt_Examples_doubly} and their discrete counterparts on the right are constructed using entirely different methods.

\begin{figure}[t]
	\includegraphics[width=\linewidth]{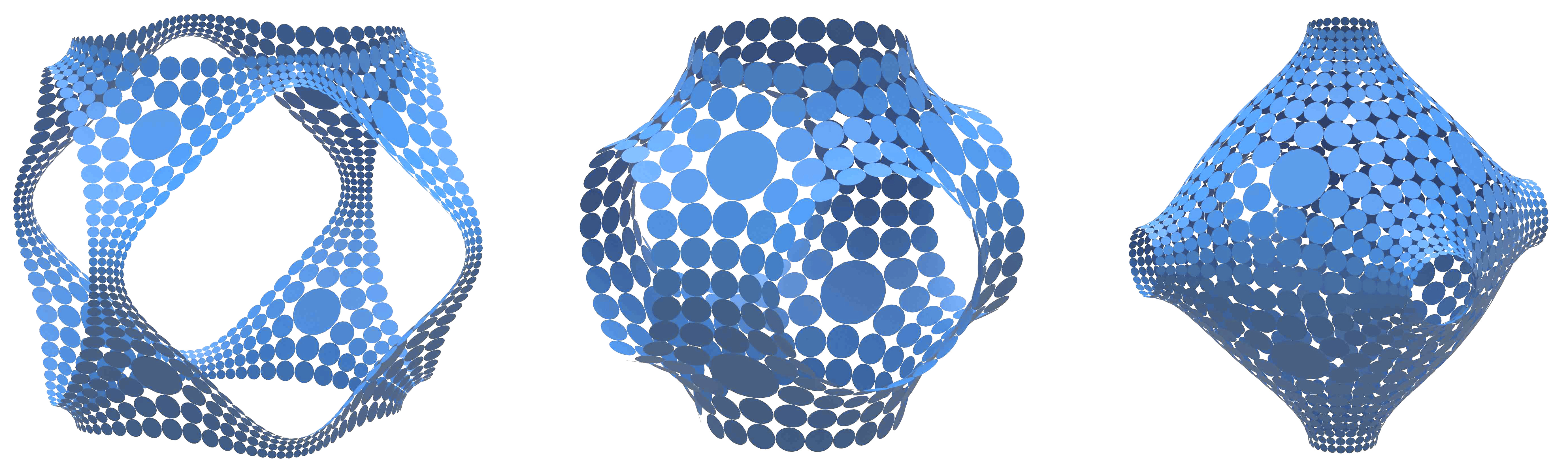}
	\caption{
		S$_1$-isothermic discrete Schwarz P minimal surface (center) and Schwarz P cmc surfaces, with negative mean curvature (left) and positive mean curvature (right). 
		The fundamental pieces of the underlying S-quad graph correspond to combinatorial rectangles with dimensions $(n, m), (m, m)$ and $(m, n)$, respectively, where $n>m$. The Gauss map for all three surfaces is bounded by the same spherical quadrilateral, indicating that the surfaces share the same reflectional symmetries.}
	\label{Fig:Rt_Schwarz-P_discrete}
\end{figure}

\subsection{Triply periodic S$_1$-cmc surfaces}
\label{sec:Rt_Examples_triply}

In the 1880s, Schwarz discovered several triply periodic minimal surfaces \cite{schwarz1972gesammelte}, which were later extended by Schoen \cite{schoen1970infinite}. Two examples are the Schwarz P surface and Schoen's I-WP surface. The surfaces admit reflective symmetries and allow deformations into triply periodic surfaces of constant mean curvature \cite{karcher1989triply}, which retain the same reflective symmetries as their minimal counterparts. The existence of the cmc deformations is proven by solving Plateau problems in $S^3$. The Lawson correspondence is used to generate the corresponding cmc surfaces in $\Rt$ \cite{karcher1989triply}. Discrete Schwarz P surfaces with mean curvature $H<0$, $H=0$ and $H>0$ are shown in Figure \ref{Fig:Rt_Schwarz-P_discrete}.

In \cite{bobenko2021constant} a smooth Schwarz P cmc surface and Schoen's I-WP cmc surface have been constructed using the DPW method \cite{dorfmeister1998weierstrass}. They are shown, together with their discretizations and orthogonal ring patterns, in Figure \ref{Fig:Introduction}. The shape of the orthogonal ring pattern for the Schwarz P surface is determined by the nominal angles $\frac{\pi}{2}, \frac{2\pi}{3}, \frac{\pi}{2}, \frac{\pi}{2}$ at the corners and the parameter $q \approx 0.995798$. Similarly, the orthogonal ring pattern for Schoen's I-WP surface has cone angles $\frac{\pi}{2}, \frac{2\pi}{3}, \frac{3\pi}{4}, \frac{\pi}{2}$ at the corners and parameter $q \approx 0.994351$. Note that the fundamental piece of the Schwarz P surface has the same boundary cone angles as the doubly periodic surface  $\psi(U_{2, 2})$ of Section \ref{sec:Rt_Examples_doubly}, only the global parameter $q$ is different.

Our numerical experiments with discrete cmc surfaces show an astonishingly good convergence. It would be desirable to give a mathematical proof of this fact. For discrete minimal surfaces the convergence is proven to be $C^\infty$ \cite{lan2009c}.

\section{Discrete S-isothermic surfaces in $\Rto$}
\label{sec:Rto_discrete_s_isothermic_surfaces}

In the remaining sections of this paper, we will translate the geometric constructions from the first part from $\Rt$ to the Lorentz space $\Rto$, with the goal of constructing discrete spacelike cmc surfaces in $\Rto$. 
We will show that, with a method analogous to the construction scheme of discrete cmc surfaces in $\Rt$, one can construct discrete spacelike cmc surfaces in $\Rto$, see e.g. \cite{milnor1983harmonic, inoguchi1997surfaces} for a differential geometric treatment of (spacelike cmc) surfaces in $\Rto$. Furthermore, we will show that the discrete Gauss map of the discrete spacelike cmc surfaces corresponds to spacelike two-sphere Koebe nets and hyperbolic orthogonal ring patterns \cite{bobenko2024rings}. 

To define spacelike S-isothermic nets in $\Rto$, we first introduce some basic geometric facts and Möbius geometry of $\Rto$. Note that in literature the conformal compactification of $\Rto$ is also denoted as the Einstein universe \cite{barbot2007primer21einsteinuniverse} and the geometry as pseudo-conformal geometry. We stick to the naming conventions \emph{Möbius quadric} and \emph{Möbius geometry} as our treatment is close to the Euclidean case.

The Lorentz space $\Rto$ is equipped with the non-degenerate bilinear form 
\begin{align}
	\label{Rto_eq:lor_bilinear_form}
	\lorsca{x, y}  = x_1y_1 + x_2y_2 - x_3y_3
\end{align} 
of signature $\texttt{(++-)}$. 
Lines and planes in $\Rto$ are called \emph{spacelike}, \emph{isotropic} or \emph{timelike} depending on the signature of the induced metric on the subspaces. For example, planes can have signatures $\texttt{(++)}$, $\texttt{(+0)}$ or $\texttt{(+-)}$ and corresponding normal vectors of signature $\texttt{(-)}$, $\texttt{(0)}$ or $\texttt{(+)}$, respectively. The three types of planes and the \emph{light cone}, consisting of all isotropic lines, are shown in Figure \ref{Fig:Rto_Lorentz} (left).

A spacelike surface surface in $\Rto$ is a surface such that the induced metric of the surface is positive definite (i.e. a Riemannian metric). Its unit normal field consists of timelike vectors, the normal vectors  are orthogonal to the spacelike tangent planes.

\begin{figure}[tbp]
	\centering
	\begin{minipage}{.325\textwidth}
		\includegraphics[width=\textwidth]{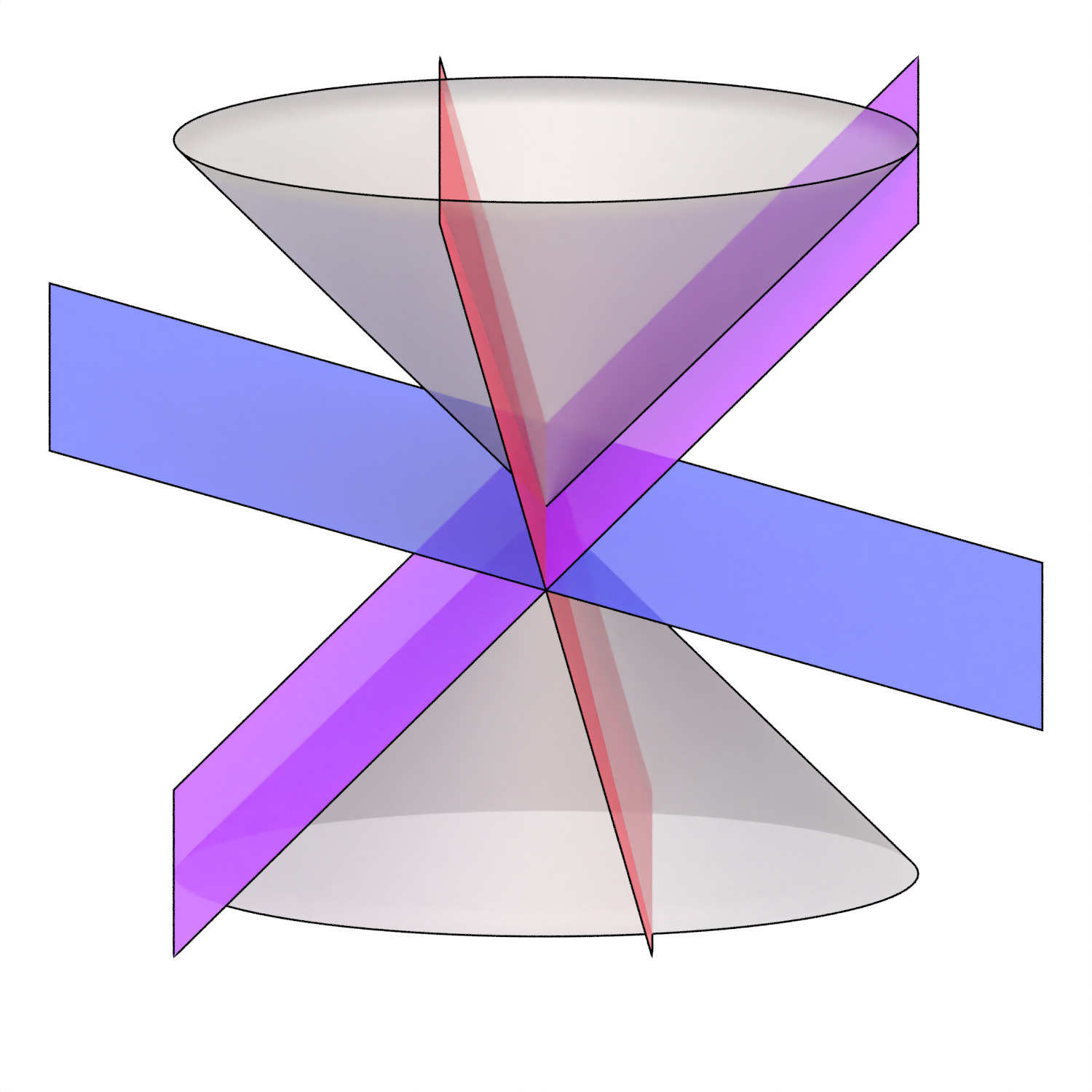}
	\end{minipage}
	\begin{minipage}{.325\textwidth}
		\includegraphics[width=\textwidth]{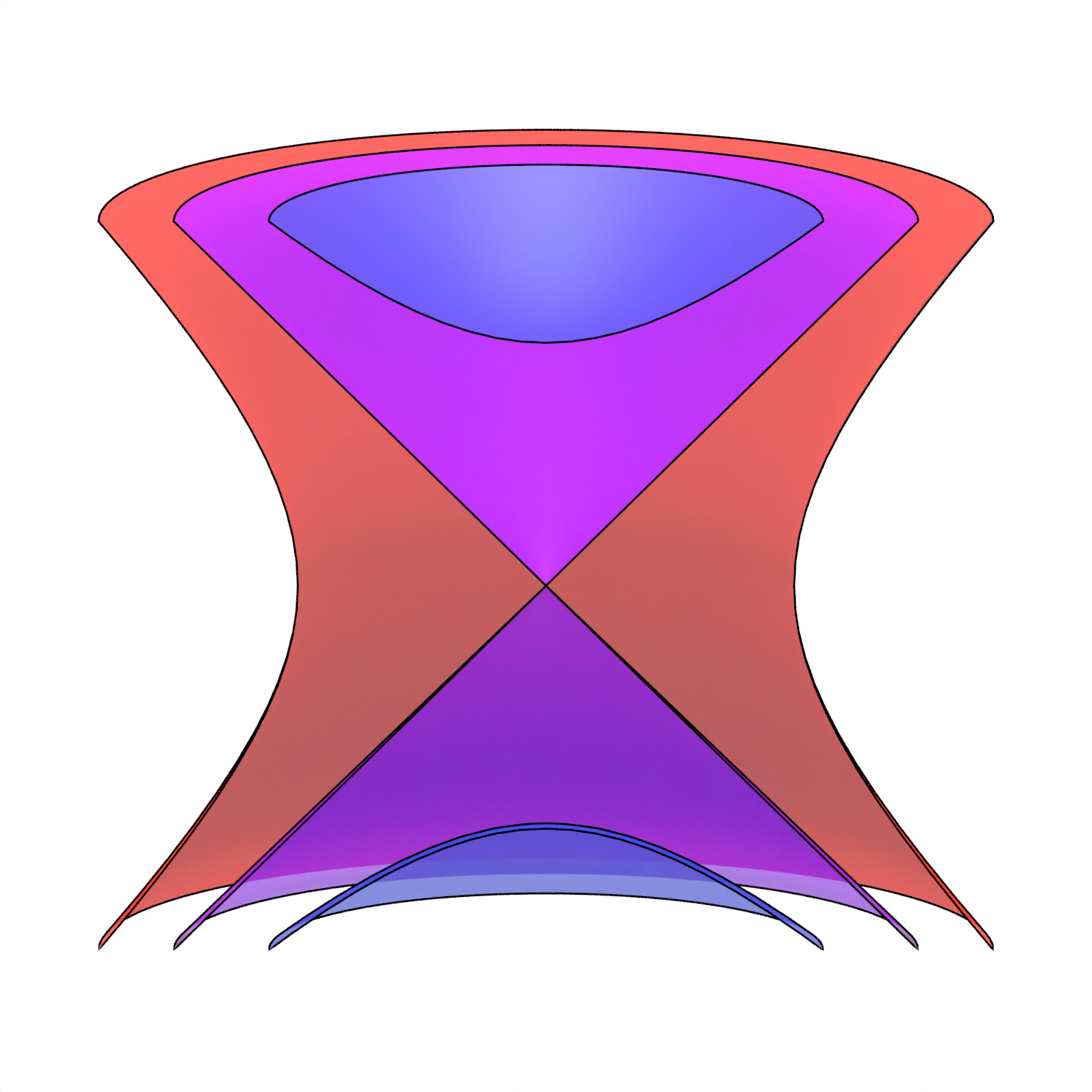}
	\end{minipage}
	\begin{minipage}{.325\textwidth}
		\includegraphics[width=1\textwidth]{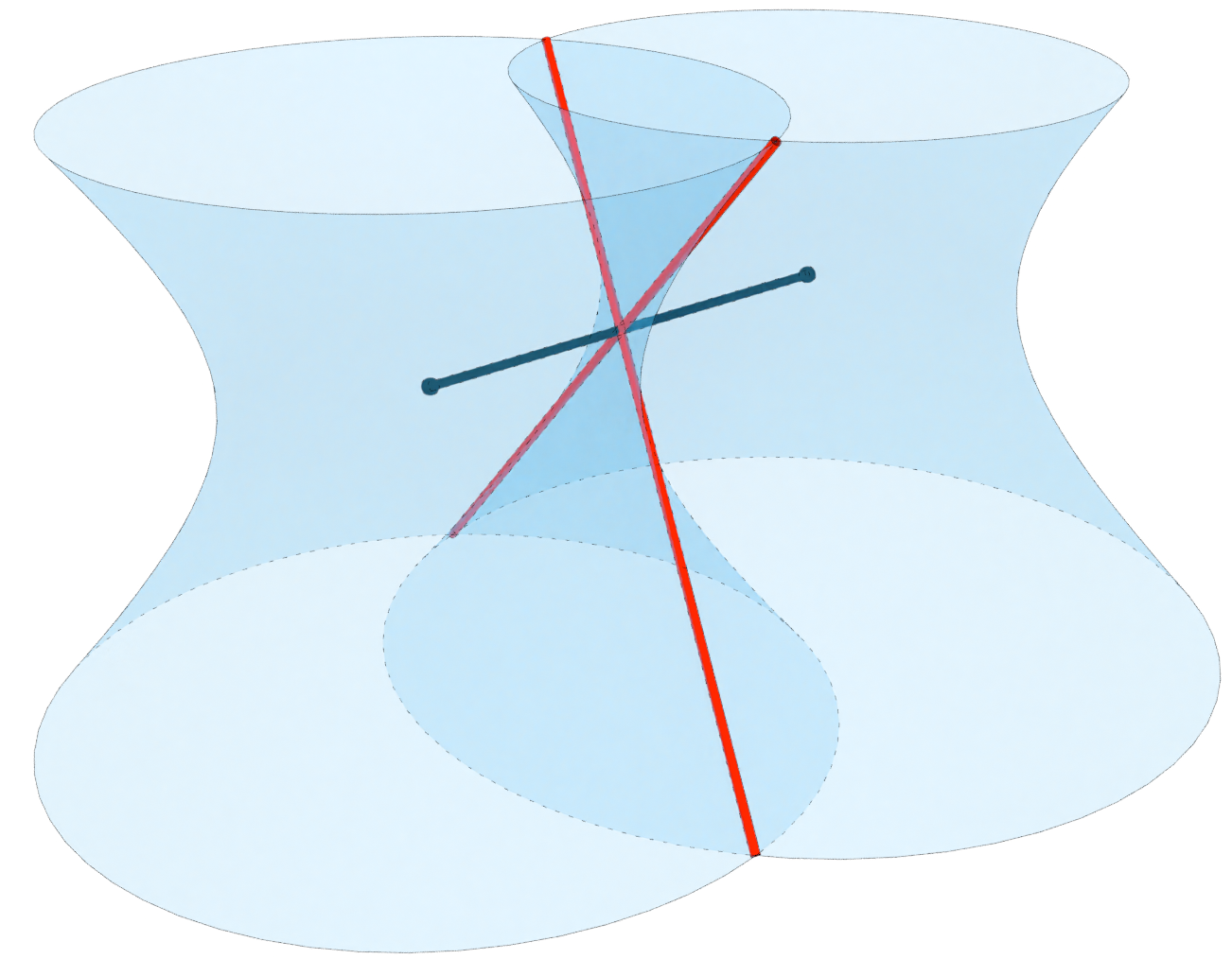}
	\end{minipage}
	\caption{Left: A spacelike plane (blue), an isotropic plane (violet), a timelike plane (red), and the light cone (gray) in $\Rto$. Center: A spacelike sphere (blue), a null sphere (violet), and a timelike sphere (red).
		Right: Two touching timelike spheres.}
	\label{Fig:Rto_Lorentz}
\end{figure}

A sphere in Lorentzian space with center $c$ and squared radius $d^2$ is given by the point set
\begin{equation*}
	s_{c, d^2} := \{x \in \Rto \ | \lorsca{x-c, x-c} = d^2\}.
\end{equation*} 
The squared radius $d^2$ can be positive, zero, or negative.
The three types of spheres are called \emph{spacelike}, \emph{null} and \emph{timelike} spheres and correspond to the cases $d^2< 0 , d^2 = 0$ and $d^2>0$, respectively. From an Euclidean point of view they correspond to two-sheeted hyperboloids, two-dimensional cones and one-sheeted hyperboloids, see Figure \ref{Fig:Rto_Lorentz} (center). 

Two timelike spheres \emph{touch} if they share a common point $t$ and have the same tangent plane at that point. This touching condition implies that the two spheres share two generators that intersect at $t$, see Figure \ref{Fig:Rto_Lorentz} (right). The line  connecting the centers of the spheres is spacelike, passes through $t$, and is orthogonal to the common timelike tangent plane.

Non-empty planar sections of spheres are \emph{spacelike, null} or \emph{timelike Lorentz circles}. Our primary examples we will be spacelike circles contained in spacelike planes, which correspond to ellipses from a Euclidean point of view. In addition, we will encounter spacelike circles contained in timelike planes which correspond to hyperbolas with a timelike axis.  Note that timelike planes can contain not only spacelike circles, but also null circles (such as the one-dimensional cones in the common tangent plane of two touching timelike spheres) and timelike circles (which correspond to hyperbolas with a spacelike axis).

To study M\"obius geometry of $\Rto$ we are using $\Rttwo$, which is  equipped with the non-degenerate bilinear from  
\begin{align}
	\label{eq:Rto_lmsca}
	\lmsca{x, y} = x_1y_1 + x_2y_2 + x_3y_3 - x_4y_4 - x_5y_5.
\end{align}
For the standard basis $e_1, ..., e_5$ of $\Rttwo$ we introduce new basis vectors $e_0:=\frac{1}{2}(e_5-e_3)$ and $e_\infty:=\frac{1}{2}(e_3+e_5)$.
Let
$$\hat{\mathbb{S}}_0 := \{\hat{x} \in \Rttwo \ | \ \lmsca{\hat{x}, \hat{x}} = 0\}$$ 
be the light cone in $\Rttwo$.
Points of $\Rto\cup \{\infty\}$ can be identified with points of $\hat{\mathbb{S}}_0$ via the identification
\begin{align}
	\label{eq:Rto_point_lift}
	\Rto \ni x  \leftrightarrow & \ \hat{x} = x+ e_0 +||x||^2e_\infty \in \hat{\mathbb{S}}_0.
\end{align}

In this context, $x$ on the right hand side is understood as $x_1e_1 + x_2 e_2 + x_3 e_4 \in \Rttwo$, and points of the form $\hat{x} = x + e_0 + ||x||^2e_\infty \in \hat{\mathbb{S}}_0$ are normalized so that $\langle \hat{x}, e_\infty \rangle_{4, 1} = - \frac{1}{2}$.
The point $\infty \in \Rt \cup \{ \infty\}$  is identified with $e_\infty$.
If we interpret $\Rttwo$ as the space of homogeneous coordinates of $\RP ^4$, points on
$$\mathcal{M} := \{ [\hat{x}] \in \R P ^4| \langle \hat{x}, \hat{x} \rangle_{3,2} = 0 \}, $$
which is called the \emph{M\"obius quadric} of $\Rto$, can be identified with null spheres in $\Rto$. A point of the form \eqref{eq:Rto_point_lift} in $\hat{\mathbb{S}}_0$ represents a special choice of homogeneous coordinates for the corresponding projective point in $\mathcal{M}$, and it represents the center of the null sphere.
Points inside the Möbius quadric,
$$[\hat{x}] \in \mathcal{M}_- := \{ [\hat{x}] \in \R P ^4 | \langle \hat{x}, \hat{x} \rangle_{3, 2} < 0 \}, $$
can be identified with non-oriented spacelike spheres in $\Rto$, 
points outside the Möbius quadric,
$$[\hat{x}] \in \mathcal{M}_+ := \{ [\hat{x}] \in \R P ^4 | \langle \hat{x}, \hat{x} \rangle_{3, 2} > 0 \}, $$ 
can be identified with non-oriented timelike spheres in $\Rt$. Hyperplanes are considered as (spacelike, isotropic, timelike) spheres with infinite radius. 

We will further consider M\"obius geometry of timelike spheres in more detail.
A timelike sphere with center $c$ and radius $d^2$ is represented by the projective point
	\begin{align}
		\label{eq:Rto_sphere_lift_projective}
		[\hat{s}] = [c + e_0 + (||c||^2-d^2)e_\infty] \in \mathcal{M}_+.
	\end{align}
Its homogeneous coordinates can be normalized to
	\begin{align}
		\label{eq:Rto_sphere_lift}
		%\Rt \supseteq  s_{c, d}  \leftrightarrow & \ 
		\hat{s} = \frac{1}{d} \left( c + e_0 + (||c||^2-d^2)e_\infty \right),
	\end{align} which is a point on the Lorentz unit sphere
	\begin{align*}
		%\label{eq:Rttwo_unit_sphere}
		\hat{\mathbb{S}}_1 := \{\hat{x} \in \Rttwo \ | \ \langle \hat{x}, \hat{x} \rangle_{3, 2} = 1\} \subset \Rttwo.
	\end{align*}
Recall that for the squared radius of timelike spheres  $d^2>0$, so its square root $d$ is always real. One can associate a sign to $d$ that encodes the orientation of the sphere. The two points \eqref{eq:Rto_sphere_lift} with $\pm d$ correspond to the same point \eqref{eq:Rto_sphere_lift_projective} in $\mathcal{M}_+$.

To study spacelike S-isothermic surfaces, let us recall the combinatorics of $\G$, a quad graph with interior vertices of even valence and edges divided into 'horizontal' and 'vertical' edges, as shown in Figure \ref{Fig:graph_g}.

\begin{definition}
	\label{def:Rto_S-isothermic}
	A map
	$$s: V(\G) \rightarrow \{\text{oriented timelike spheres in } \Rto \}$$
	is called a \emph{discrete spacelike S-isothermic surface}
	if the following conditions hold:
	\begin{enumerate}[(i)]
		\item The centers of four oriented spheres associated with a face of $\G$ lie on a spacelike plane.
		\item The corresponding map  ${\hat{s}: V(\G) \rightarrow \hat{\mathbb{S}}_1 \subset \Rttwo}$, that maps oriented timelike spheres to 
		\begin{align*}
			%\label{eq:Rto_S_iso_lift}
			\hat{s}=\frac{1}{d}\left( c + e_0 +(||c||^2 - d^2)e_\infty\right),
		\end{align*} satisfies the discrete Moutard equation 
		\begin{align}
			\label{eq:Rto_Moutard_s_iso}
			\hat{s}_{v_{ij}} - \hat{s}_v  = a_{ij}(\hat{s}_{v_j} - \hat{s}_{v_i})
		\end{align}
	
	for some $a_{ij}: F(\G) \rightarrow \R$. 
	\end{enumerate}
\end{definition}

The vertex spheres of a spacelike discrete S-isothermic surface are timelike. Their positive squared radius corresponds to the positive metric of the surface.
The labeling property (cf. \eqref{eq:Rt_labelling_property}) also holds for spacelike S-isothermic surfaces
	\begin{equation*}
	\begin{aligned} 
		\langle \hat{s}_v,  \hat{s}_{v_i} \rangle_{3, 2}  &= \langle \hat{s}_{v_j}, \hat{s}_{v_{ij}} \rangle_{3, 2}  =:\alpha_i\\
		\langle \hat{s}_v, \hat{s}_{v_j} \rangle_{3, 2}  &= \langle \hat{s}_{v_i}, \hat{s}_{v_{ij}} \rangle_{3, 2}  =:\alpha_j.
	\end{aligned}
\end{equation*}
The elementary quadrilaterals of a spacelike S-isothermic surface are called \emph{spacelike S-isothermic quadrilaterals}. 
There are three types of spacelike S-isothermic quadrilaterals, analogous to the three types in $\Rt$ presented in Figure \ref{Fig:Rt_Q_congruences}, but with different geometric properties.
A special case are those with touching vertex spheres, called \emph{spacelike S$_1$-isothermic quadrilaterals}.

\begin{figure}
	\centering
	\begin{minipage}{.4\linewidth}
				\begin{overpic}[width=\linewidth]{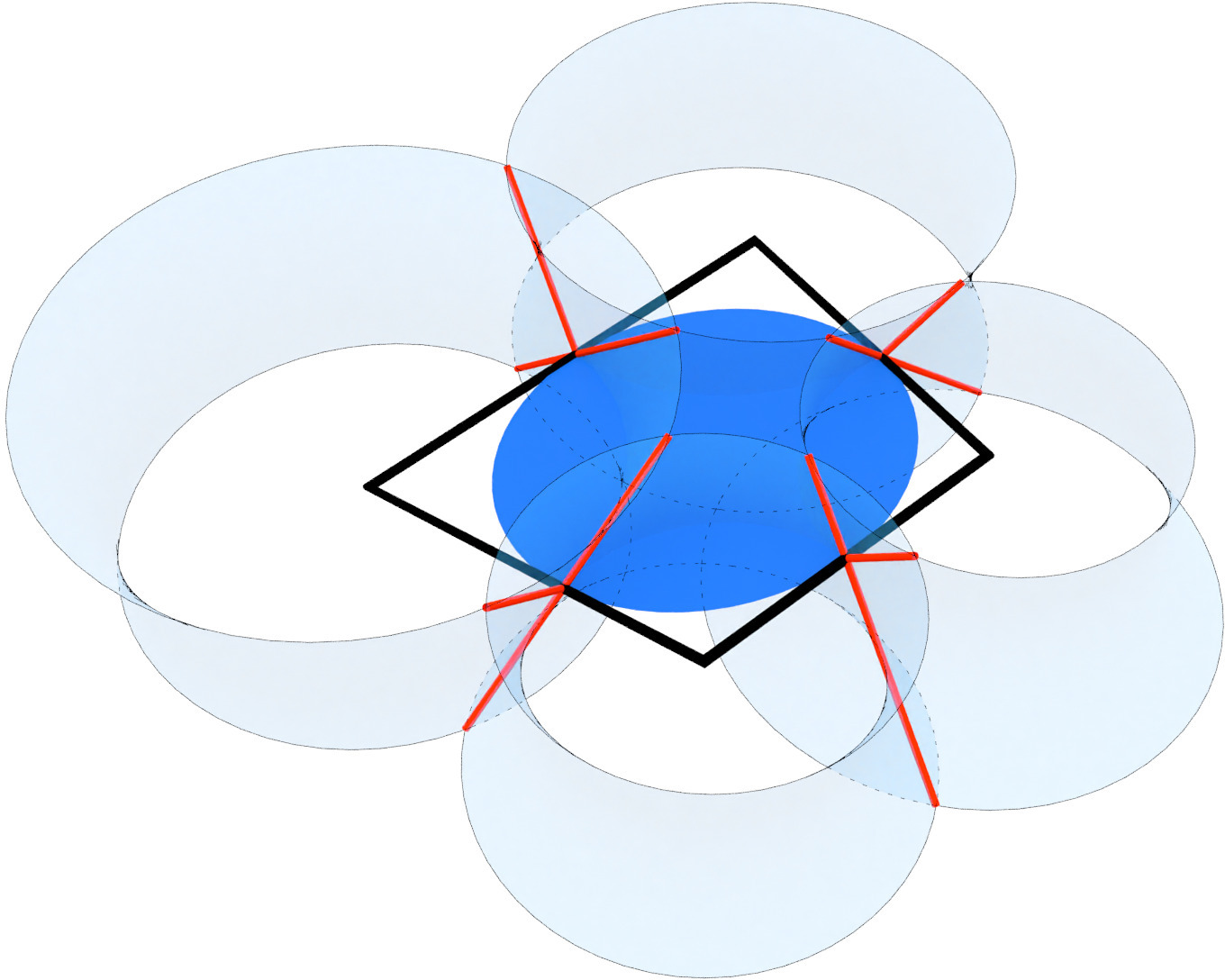}
			\put(22, 39){$c_{v}$}
			\put(56, 20){$c_{v_i}$}
			
			\put(82, 42){$c_{v_{ij}}$}
			\put(57, 63){$c_{v_{j}}$}
		\end{overpic}
	\end{minipage}
\hspace{2cm}
	\begin{minipage}{.3\linewidth}
			\begin{overpic}[width=\linewidth]{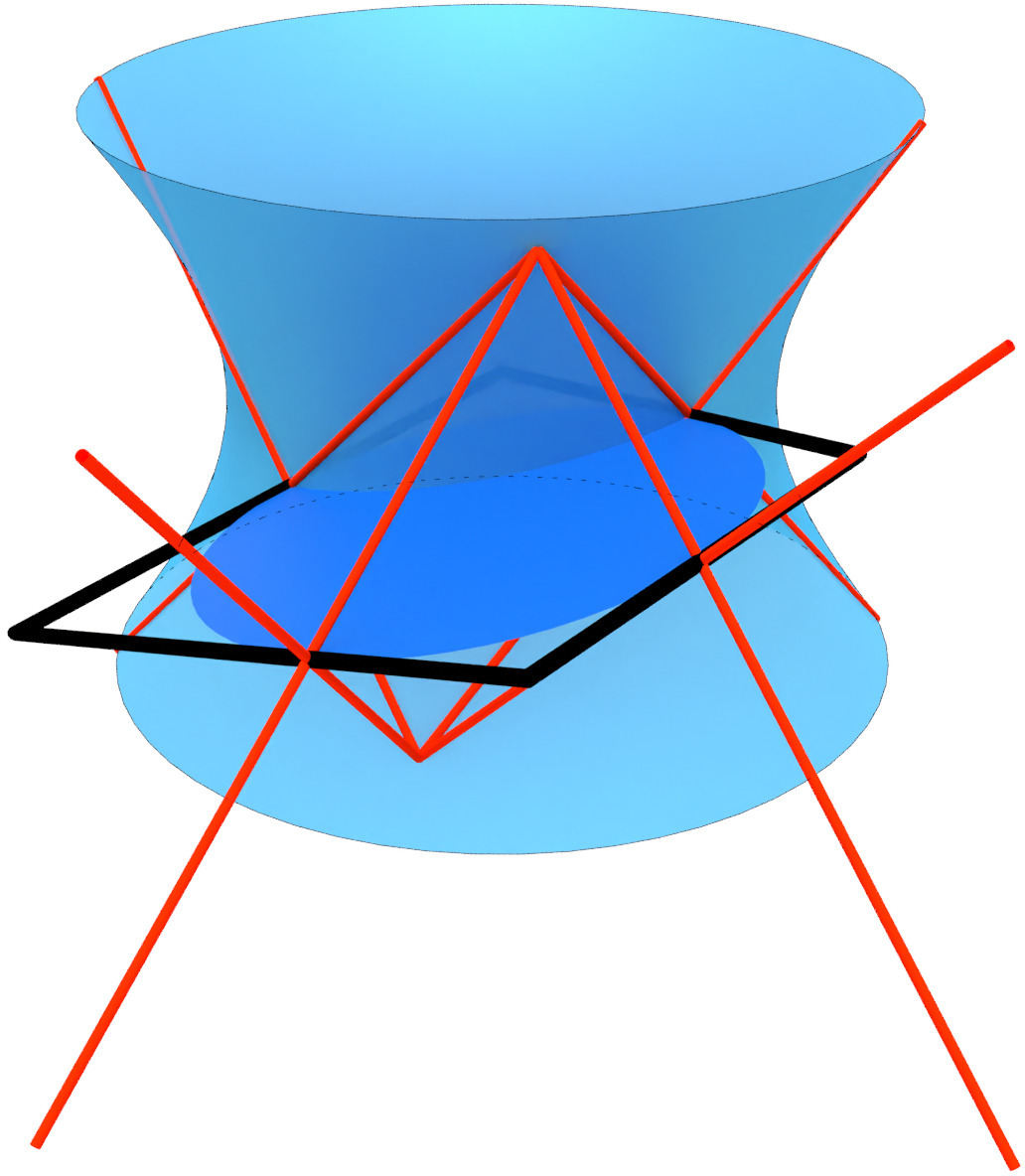}
			\put(79, 95){$s_{v_{j}}$}
		\end{overpic}
	\end{minipage}
\hspace{.5cm}
	\caption{A spacelike S$_1$-isothermic quadrilateral. Neighboring spheres touch and have a common orthogonal circle (left). The four spheres intersect in two points that form the centers of two null spheres that contain the orthogonal circle (right).}
	\label{Fig:Rto_Q_congruences}
\end{figure}

\begin{proposition}
	Spacelike S$_1$-isothermic quadrilaterals have spacelike incircles that intersect the four adjacent spheres orthogonally. The four timelike vertex spheres intersect at two points which are the centers of the two null spheres containing the orthogonal circle, see Figure \ref{Fig:Rto_Q_congruences}. 
\end{proposition}

\begin{proof}
	The faces of the net $\hat{s}$ in $\Rttwo$ are planar. Similar to the proof of Corollary \ref{Cor:Rt_Q_congruences}, we consider the linear subspace containing a face of $\hat{s}$. The orthogonal complement of the linear subspace is of signature $(1, 1)$. Thus it represents (in contrast to the Euclidean case) a circle that intersects the four spheres associated to the face orthogonally and that is, by definition, spacelike.
	Furthermore, the orthogonal complement having signature $(1, 1)$ also implies that the four spheres in $\Rto$ are orthogonal to two null spheres and thus in particular intersect in their centers. For more details we refer to \cite{ADMPS24a}. \end{proof}

The \emph{Christoffel dual}  $s^*$ of a spacelike S$_1$-isothermic surface $s$ can be defined analogously to the $\Rt$ case by the exact one-form
 $\partial c^*$ defined by
	\begin{align*}
		\partial_{(v, v')}c^* = \frac{\partial_{(v, v')} c}{d_v d_{v'}}.
	\end{align*}
Recall that $$c: V(\G) \rightarrow \Rto, \ v \mapsto c_v$$ denotes the center net of $s$ and $d: V(\G) \rightarrow \R, \ v \mapsto d_v$ the corresponding signed radii. For spacelike S$_1$-isothermic surfaces $s$ and $s^*$ let $c_f, c^*_f$ and $d_f, d_f^*$ denote the centers and radii of the spacelike orthogonal face circles.

A spacelike S$_1$-isothermic surface and its Christoffel dual have the following properties:
\begin{enumerate}[(i)]
	\item The Christoffel dual of a spacelike S$_1$-isothermic net is a spacelike S$_1$-isothermic net.
	\item The radii of the vertex spheres are related by $d_v^* = \frac{1}{d_v}$.
	\item Orthogonal circle radii are related by $d_f^* = \frac{1}{d_f}$.
\end{enumerate}
In general, the Christoffel dual is defined up to global scaling, thus
\begin{equation}
	\label{eq:Rto_scaling_christoffel}
	d_vd_v^* = d_f d_f^* = \lambda
\end{equation}
for a $\lambda \in \R$.
\section{Discrete cmc surfaces in $\Rto$}
\label{sec:Rto_discrete_s_isothermic_cmc_surfaces}

\begin{definition}
	\label{Def:Rto_Darboux}
	Two S-isothermic surfaces $s$ and $s^+$ are called a \emph{timelike Darboux pair} if the corresponding Moutard nets $\hat{s}, \hat{s}^+ : V(\G) \rightarrow \Rttwo $ are related by a Moutard transformation,
	\begin{align*}
		\hat{s}^+_{v'}- \hat{s}_v  = a_+(\hat{s}^+_{v} - \hat{s}_{v'})
	\end{align*}
	with some $a_+: E(\G) \rightarrow \R$,
	 and the normal vectors connecting corresponding vertex sphere centers,
	\begin{align}
		\label{eq:Rto_normal_vectors}
		n_v := c_v^* - c_v,
	\end{align}
	are timelike.
	In this case, one surface is called a \emph{timelike Darboux transform} of the other.
\end{definition}
The normal vectors of smooth spacelike surfaces in $\Rto$ are timelike. The vectors \eqref{eq:Rto_normal_vectors} will serve as the discrete Gauss map for the following definition of S-cmc surfaces and are therefore restricted to be timelike. Timelike Darboux transforms preserve the edge labels $\alpha_i$ and $\alpha_j$ given in \eqref{eq:Rto_Moutard_s_iso}. Associated with a timelike Darboux transform is a constant parameter $\alpha$ given by 
\begin{align}
	\label{eq:Rto_alpha}
	-2\alpha  := -2\langle \hat{s}_v,  \hat{s}_{v}^+ \rangle_{3, 2} = ||c_v - c_v^+||^2 - \left( d_v^2 + {d_v^+}^2\right).
\end{align}

\begin{definition}
	A spacelike S-isothermic surface $s$ is a \emph{spacelike S-cmc surface}, if its Christoffel dual simultaneously is a timelike Darboux transform (after appropriate scaling and translation).
\end{definition} 
We now restrict ourselves to the case of touching spheres, i.e., spacelike S$_1$-cmc surfaces.
A fundamental hexahedron of a Lorentz S$_1$-cmc surface pair has the same properties as in the Euclidean case, see Figure \ref{Fig:Rt_fundamental_hex}, also for notation. In particular, the normal vectors have the following properties.
The \emph{vertex normals}
\begin{align}
	\label{eq:Rto_vertex_normals}
	n: V(\G) \rightarrow \Rto,\ v \mapsto n_v := c^*_v - c_v, 
\end{align}
connect centers of primal and dual spheres. They are timelike per definition.

The \emph{edge normals}
\begin{align}
	\label{eq:Rto_edge_normals}
	l: E(\G) \rightarrow \Rto,\ (v, v') \mapsto l_{(v, v')} := t^*_{(v, v')} - t_{(v, v')}
\end{align}
connect touching points of vertex spheres. They are orthogonal to the parallel pair of primal and dual edge and they lie in timelike planes containing the pair of parallel edges, and are therefore timelike. They are of squared lengths
 \begin{align}
	\label{eq:Rto_Delta}
	\Delta^2_i := -2\alpha +2\lambda \text{ \ or  \ } 
	\Delta^2_j := -2\alpha -2\lambda,
\end{align}
depending on the label of the corresponding pair of parallel edges, see Figure 	\ref{Fig:Rto_side_faces}. Here $\alpha$ denotes the parameter of the Darboux transform \eqref{eq:Rto_alpha} and $\lambda$ denotes the global constant \eqref{eq:Rto_scaling_christoffel}. Note that \mbox{$ \Delta_j^2  < \Delta_i^2 <0$}.

The \emph{face normals}
\begin{align}
	\label{eq:Rto_face_normals}
	m: F(\G) \rightarrow \Rto,\ f \mapsto m_f := c_f^*-c_f
\end{align}
connect centers of primal and dual orthogonal circles. The circles are coaxial and the face normals are orthogonal to the parallel spacelike faces.

\begin{figure}
	\centering
	\begin{minipage}{.48\linewidth}
		
		\begin{overpic}[width=.8\linewidth]
			{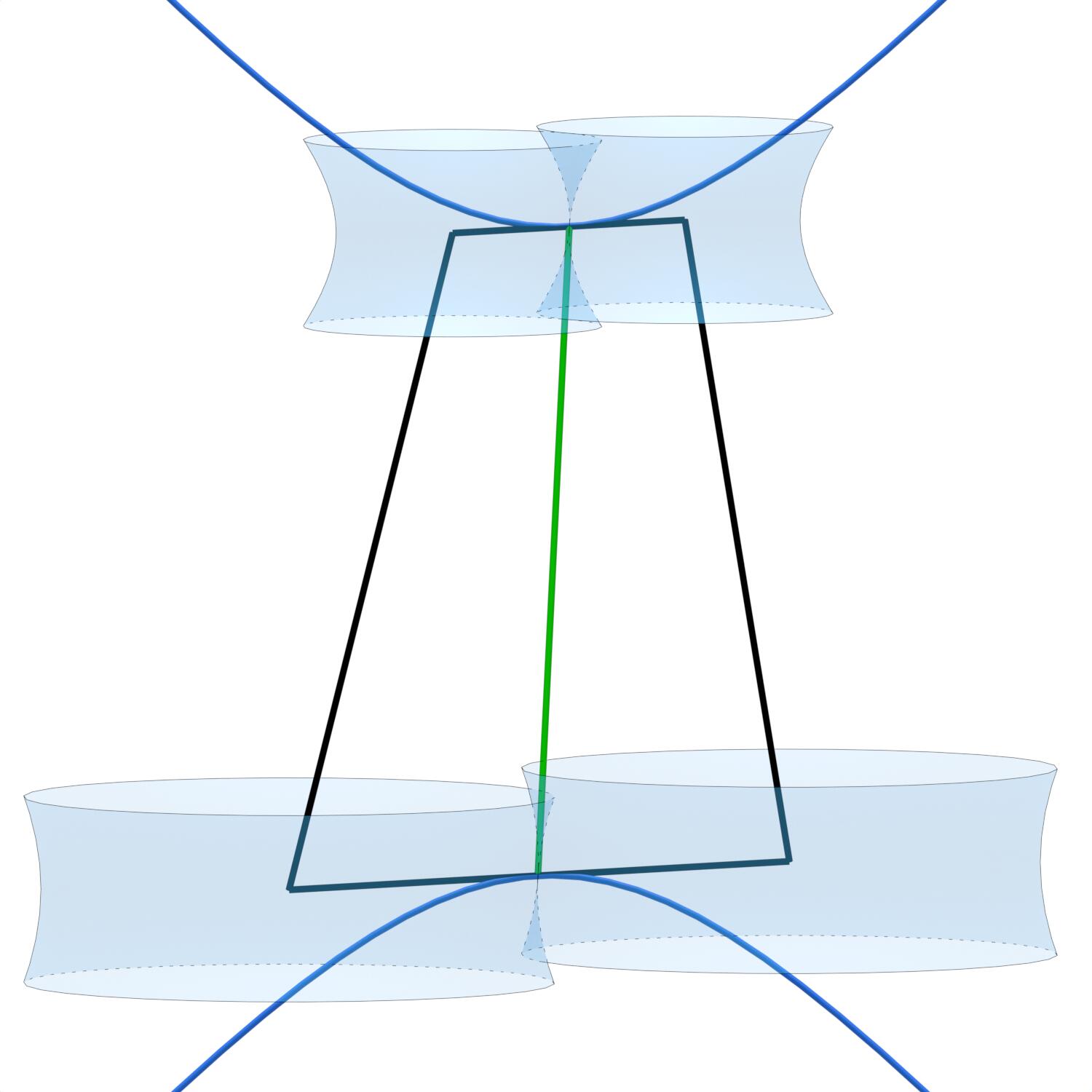}
			
			\put(33,76){$c_v^*$}
			\put(65,77){$c_{v_i}^*$}
			
			\put(26,55){$n_v$}
			\put(68,54){$n_{v_i}$}
			\put(51,38){$l_{(v, v_i)}$}
			
			\put(19,16){$c_v$}
			\put(73.5,19){$c_{v_i}$}
		\end{overpic}
	\end{minipage}
\hspace{-1cm}
	\begin{minipage}{.48\linewidth}
		
		\begin{overpic}[width=.85\linewidth]
			{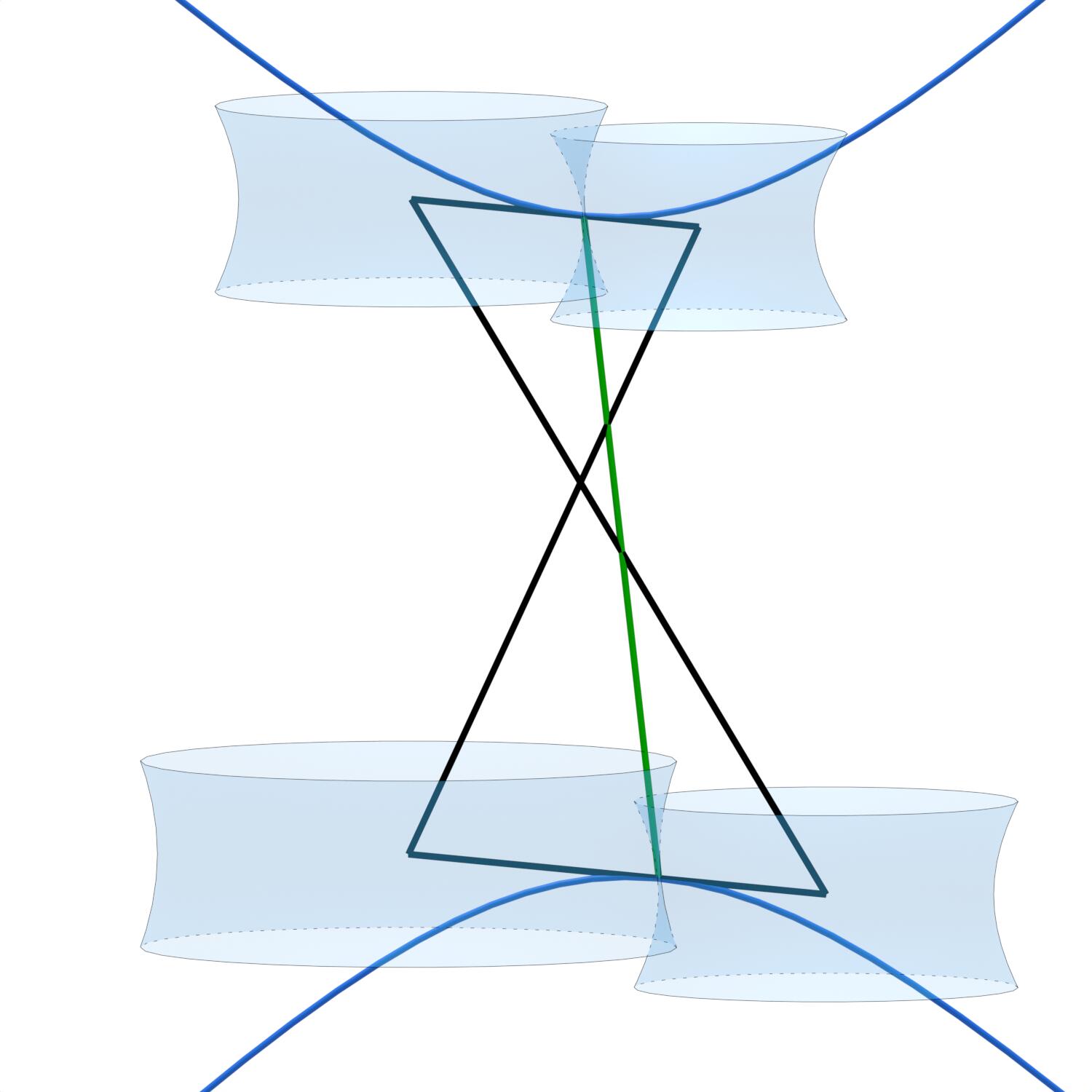}
			
			\put(30,80){$c_{v_j}^*$}
			\put(65,76){$c_v^*$}
			
			\put(36,38){$n_v$}
			\put(64,38){$n_{v_j}$}
			\put(57,54){$l_{(v, v_j)}$}
			
			\put(31,20.5){$c_v$}
			\put(76.5,17){$c_{v_j}$}
		\end{overpic}
	\end{minipage}
	\caption{The two types of transformation faces of a spacelike S$_1$-cmc pair together with vertex normals (black) and edge normals (green). The circle orthogonal to the four adjacent vertex spheres (blue) is spacelike. Embedded faces correspond to horizontal edges, non-embedded faces correspond to vertical edges.}
	
	\label{Fig:Rto_side_faces}
\end{figure}
\begin{definition}
\label{def:Rto_two-sphere_Koebe}
A Q-net $k:V(\G) \rightarrow \Rto$ is called a \emph{spacelike two-sphere Koebe net} if its edges alternately touch two spacelike spheres $S_+^2$ and $S_-^2$ (concentric with the spacelike unit sphere) whose squared radii satisfy the relation $r^2_+r^2_- = 1$.
\end{definition}
Recall that in Section \ref{sec:Rt_discrete_curvatures} we introduced the discrete mean curvature $H$ for pairs of discrete surfaces with parallel edges in $\Rt$. Since we restrict our considerations to discrete spacelike surfaces, that is discrete surfaces with positive metric, the considerations for $\Rt$ can be applied directly and analogously to the case of spacelike discrete surfaces with parallel edges in $\Rto$.

The following theorem can be proved analogously to Theorems \ref{Rt:Thm_s_cmc_koebe} and \ref{Thm:Rt_H=1}.

\begin{theorem}
	\label{Rto:Thm_s_cmc_koebe}
	Let $s$ and $s^*$ be a (suitably scaled) spacelike S$_1$-cmc pair.
	The Gauss map  \begin{align}
		\label{eq:Rto_vertex_normals2}
		n: V(\G) \rightarrow \Rto,\ v \mapsto n_v := c^*_v - c_v, 
	\end{align}
	between the sphere centers of $s$ and $s^*$, forms a spacelike two-sphere Koebe net. Furthermore, the pair $(s, n)$ of an S$_1$-cmc surface and its Gauss map has constant discrete mean curvature $H=1$.
\end{theorem}

\begin{figure}
	
	\begin{overpic}[scale=.2]
		{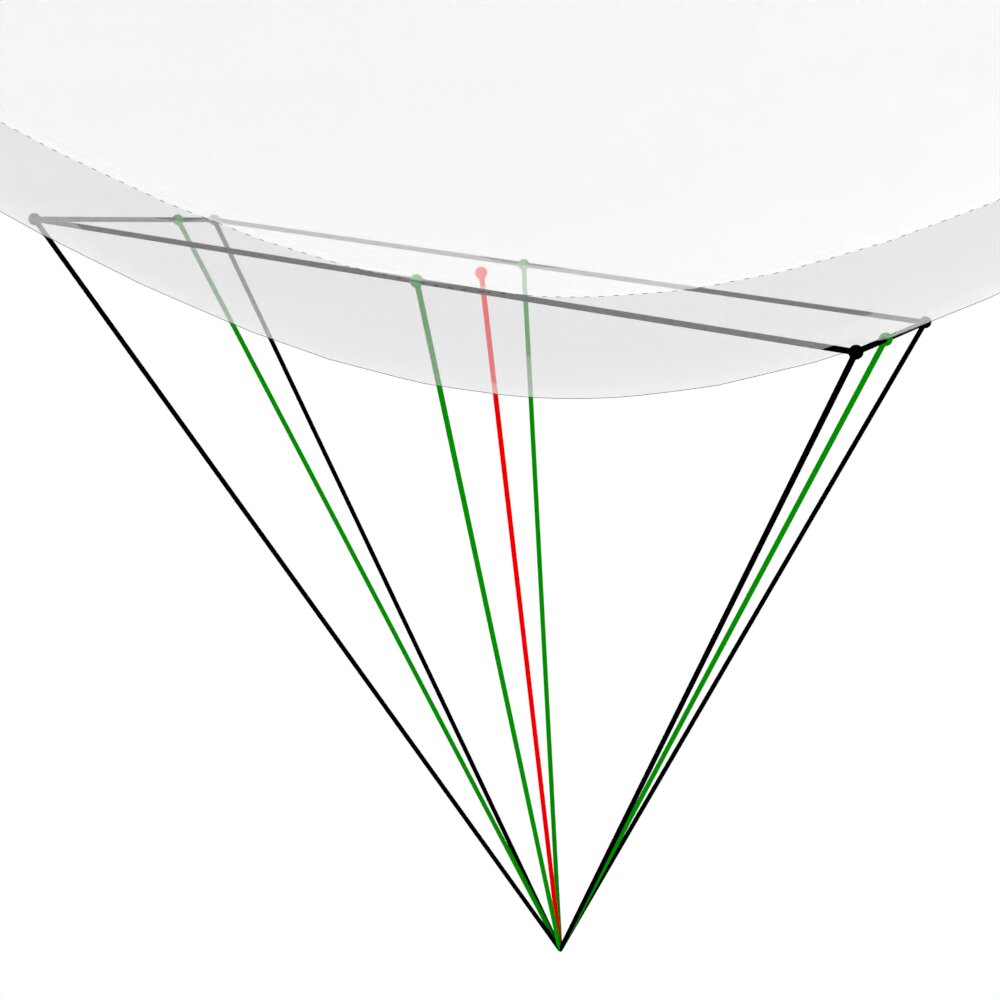}
		
		\put(-2,80){$n_v$}
		\put(21,79){$n_{v_i}$}
		\put(77,63){$n_{v_j}$}
		\put(90.5,69){$n_{v_{ij}}$}

		\put(41,74){$m_f$}
		
		\put(36,68){$l_{(v, v_j)}$}
		
		\put(07,81){$l_{(v, v_i)}$}
		
		\put(101,80){$S_+^2$}
		\put(101,70){$S_-^2$}
	\end{overpic}
	\caption{The Gauss image of a fundamental spacelike S$_1$-cmc hexahedron, which forms a face of a spacelike two-sphere Koebe net, alternately tangent to the two spacelike spheres $S^2_+$ and $S^2_-$.  The points of tangency are given by the edge normal vectors $l_{(v, v')} \in S_\pm^2$ (green). Horizontal edges always touch $S^2_-$ and vertical edges always touch $S^2_+$. The vector $m_f$ (red) is orthogonal to the planar face.} 
	\label{Fig:Rto_fundamental_hex_gauss}
\end{figure}

\section{Spacelike two-sphere Koebe nets and hyperbolic orthogonal ring patterns}
\label{sec:Rto_two_spheres_Koebe_q_nets_and_hyperbolic_orp}
In this section, we will study the correspondence between pairs of spacelike dual two-sphere Koebe nets and hyperbolic orthogonal ring patterns. Properties derived analogously to those in the correspondence between two-sphere Koebe nets in $\Rt$ and spherical orthogonal ring patterns remain without a proof. For a more detailed consideration we refer to Section \ref{sec:Rt_two_spheres_Koebe_q_nets_and_spherical_orp}.

Let $k^s$ and $k^c$ be a pair of regular, dual spacelike two-sphere Koebe nets (cf. Definition \ref{def:Rt_dual_two-sphere_Koebe}) with underlying S-quad graph $\S$ (cf. Definition \ref{def:S_quad}), which alternately touch the upper parts of two spacelike spheres $S^2_+$ and $S^2_-$. The spheres $S^2_+$ and $S^2_-$ with radii $r_+^2$ and $r_-^2$ are assumed to be concentric with the spacelike unit sphere \mbox{$S^2 := \{x \in \Rto \mid  \langle x, x \rangle = -1 \} \subset \Rto$}. The upper parts of $S^2_+$ and $S^2_-$ are given by 
\begin{equation*}
	S^2_\pm:=\{x \in \Rto \ \mid \lorsca{x, x} = r_\pm^2, \ x_3 >0 \}.
\end{equation*} 
The squared radii $r_+^2$ and $r_-^2$ are negative, by $|r_+|$ and $|r_-|$ we denote the absolute values of their imaginary square roots.
Analogously, we introduce the notation
\begin{align}
	\label{eq:Rto_normalizing}
	|x|: = |\sqrt{\lorsca{x, x}}| \  \ \text{ for } x \in \Rto.
\end{align}
Projecting timelike vectors onto $S^2$ corresponds to normalizing the vectors with \eqref{eq:Rto_normalizing}. In particular, for vertices $k_v$ and tangent points $t^\pm_\black$ of either $k^s$ or $k^c$, let 
\begin{align*}
	p_v : = \frac{k_v}{|k_v|}, \
	q_\black := \frac{t^+_\black}{|r^+|} =  \frac{t^-_\black}{|r^-|}
\end{align*}
denote their projection onto $S^2$.
The upper part of $S^2$ can be interpreted as the hyperboloid model of the hyperbolic plane $H^2$ embedded in $\Rto$. One can can associate hyperbolic rings (pairs of concentric hyperbolic circles) with the white vertices $v \in V_w(\S)$, centered at $p_v$ with hyperbolic radii (determined up to sign)
\begin{align*}
	%\label{eq:Rto_cos}
	\cosh(r_v) = \frac{|r_-|}{|k_v|}, \
	\cosh(R_v) = \frac{|r_+|}{|k_v|}.
\end{align*} 
Here $r_v$ always denotes the smaller radius (of the inner circle) and $R_v$ the larger radius (of the outer circle) of the ring.
Each ring passes through the projection of adjacent tangent points $q_b$. Two rings corresponding to white vertices of a face of $\S$, see Figure \ref{Fig:Rt_S_Quad_graph_diagonal} (right), intersect orthogonally, i.e., the inner circle of one ring intersects the outer circle of the other ring orthogonally, and vice versa. The rings form a hyperbolic orthogonal ring pattern \cite{bobenko2024rings}, see Figure   \ref{Fig:hyperbolic_orp}. 

\begin{figure}
	\centering
	\includegraphics[width=.6\linewidth]{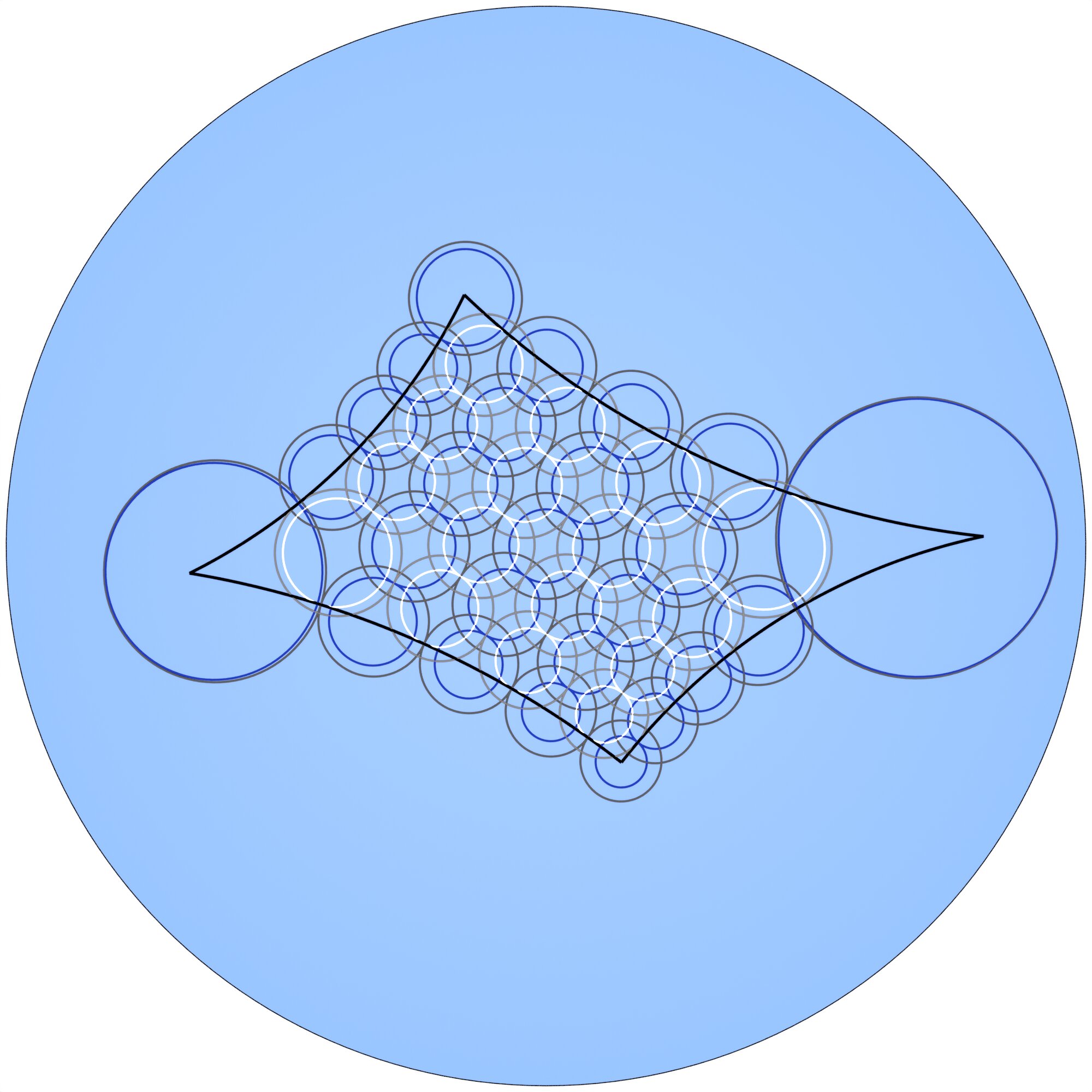}
	\caption{A hyperbolic orthogonal ring pattern, shown in the Poincaré disk model, with $q=0.99$ and angles 
		$\pi$ for the boundary vertices and  $\frac{\pi}{2}, \frac{\pi}{10}, \frac{2\pi}{5}, \frac{\pi}{5}$ for the four corner vertices. }%
	\label{Fig:hyperbolic_orp} 
\end{figure}

For hyperbolic orthogonal ring patterns there exists a global constant $q<1$ such that the radii of inner and outer circles of all rings are related by 
\begin{align}
	\label{eq:Rto_q}
	q  \cosh R_v  =  \cosh r_v.
\end{align}  

This follows from the orthogonality of the rings and the hyperbolic Pythagorean theorem. Conversely, lift the centers $p_v$ of the hyperbolic rings to 
\begin{align}
	\label{eq:Rto_sorp_lift_p}
	k_v := \frac{\sqrt{q}}{\cosh(r_v)}\ p_v = \frac{1}{\sqrt{q}\cosh(R_v)}\ p_v,
\end{align}
and the touching points $q_b$ to the two points 
\begin{align}
	\label{eq:Rto_sorp_lift_t}
	t_b^+ := \frac{1}{\sqrt{q}}\ q_b, \ t_b^- := \sqrt{q} \ q_b.
\end{align}

Restricting $\S$ to the two graphs $\G$ and $\G^*$, see Figure \ref{Fig:Rt_S_Quad_graph_diagonal} (left), we obtain two combinatorially dual, polyhedral surfaces with vertices \eqref{eq:Rto_sorp_lift_p}. They form a pair of dual spacelike two-sphere Koebe nets. They are alternately tangent to the upper part of the spacelike spheres $S^2_+$ and $S^2_-$ where the points of tangency are given by \eqref{eq:Rto_sorp_lift_t}. The spacelike spheres $S^2_+$ and $S^2_-$ are of squared radii $-\frac{1}{q}$ and $-q$, respectively. We obtain a correspondence analogous to the one between dual two-sphere Koebe nets and spherical orthogonal ring pattern (cf. Theorem \ref{Thm:Rt_Koebe_and_orp}).

\begin{theorem}
	\label{Thm:Rto_Koebe_and_orp}
	Pairs of regular, dual spacelike two-sphere Koebe nets touching the upper parts of the spacelike spheres $S^2_+$ and $S^2_-$, with radii of absolute values $|r_+|$ and $|r_-|$, are in one to one correspondence to hyperbolic orthogonal ring patterns in the hyperbolic plane with global parameter $q = \frac{|r_-|}{|r_+|}$. Vertices of the Koebe nets correspond to centers of hyperbolic rings, while points of tangency of the Koebe nets correspond to touching points of the rings.
\end{theorem}

Hyperbolic orthogonal ring patterns allow a variational description that can be derived in a way analogous to the spherical case.  We give a brief overview here and refer to Section \ref{sec:Rt_two_spheres_Koebe_q_nets_and_spherical_orp} and to \cite{bobenko2024rings} for further details proofs. 

Recall that by $\S_z$ we denote an S-quad graph defined by a simply connected subset of squares of the $\Z^2$ lattice in $\R^2$, and by $V_B$ the set of its white boundary vertices. Our main example is the combinatorial rectangle \eqref{eq:Z2_rectangle}.

Due to \eqref{eq:Rto_q} one can uniformize the hyperbolic ring radii $R_v$ and $r_v$ in terms of Jacobi elliptic functions:
\begin{equation}
	\label{eq:hyperbolic_ring_uniformization_gamma}
	\cosh R_v=\frac{1}{q\sn (\gamma_v,q)} ,\ \tanh R_v= \dn (\gamma_v,q), \ \sinh r_v= \frac{\cn (\gamma_v, q)}{\sn (\gamma_v,q)}, 
\end{equation}
where $\gamma_v\in [0, 2K]$ for $(r_v,R_v)\in [-\infty, \infty] \times [R_0,\infty]$, and $\cosh R_0= 1/q$. 
Using the function $F$ from \eqref{eq:F(x)} one defines the functional

\begin{equation}
	\label{eq:Rto_functional_hyperbolic}
	S_{hyp}(\gamma):=\sum_{(v, v')} \left( F(\gamma_{v}-\gamma_{v'})+F(\gamma_{v}+\gamma_{v'})\right) +\sum_{v }\Phi_{v}\gamma_{v},
\end{equation}
where the first sum is taken over all pairs of white vertices corresponding to neighboring rings and the second sum is taken over all white vertices $v \in  V_w(\Sz)$.
Critical points $\gamma$ of this functional with appropriately chosen $\Phi_{v}$ correspond to the radii of the rings of hyperbolic orthogonal ring patterns. Remarkably, the functional \eqref{eq:Rto_functional_hyperbolic} is convex. Its minimization allows to construct hyperbolic orthogonal ring patterns from suitable boundary data, and one can prove their existence and uniqueness \cite{bobenko2024rings}.

\begin{theorem} 
	\label{Thm:Rto_bdy_problem}
	Let $\Sz$  be a combinatorial rectangle \eqref{eq:Z2_rectangle} with white corner vertices and white boundary vertices $V_B$. Further let $q\leq 1$.
	Hyperbolic orthogonal ring patterns can be obtained as solutions of the following boundary value problems:
	\begin{itemize}
		\item (Dirichlet boundary conditions)
		For any choice of prescribed radii $\gamma:V_B\to [0,2K]$ of boundary rings there exists a unique hyperbolic orthogonal 
		ring pattern with these boundary radii.
		\item (Neumann boundary conditions) 
		Let $\Theta:V_B\to (-2\pi, 2\pi)$ be prescribed boundary cone angles such that $|\Theta_v|<\pi$ at the corners.
		Then  there exists a unique hyperbolic orthogonal ring pattern $\mathcal R$ with these boundary cone angles.  
	\end{itemize}
	Moreover, the solution of the Neumann boundary value problem is given by the unique minimizer of the functional \eqref{eq:Rto_functional_hyperbolic} with $\Phi_{v}= - 2\pi$ for all inner vertices, $\Phi_{v} = -\Theta_v$ for all positively oriented boundary rings, $\Phi_{v} = -\Theta_v- \pi$ for negatively oriented corners, and $\Phi_{v} = -\Theta_v - 2 \pi$ for other negatively oriented boundary rings.
%		Moreover, the solution of the Neumann boundary value problem is given by the unique minimizer of the functional \eqref{eq:Rto_functional_hyperbolic} with $\Phi_{v}= - 2\pi$ for all inner vertices, $\Phi_{v} = \Theta_v$ for all negatively oriented boundary rings, $\Phi_{v} = \Theta_v- \pi$ for all positively oriented corners, and $\Phi_{v} = \Theta_v - 2 \pi$ for other positively oriented boundary rings.
\end{theorem}

For $q=1$ one obtains a hyperbolic orthogonal circle pattern. The circle radii in terms of the $\gamma$-variables, which we denote by $\gamma^0$, are given by 
\begin{equation*}
	\cosh R= \coth \gamma^0, \ \sinh R= \frac{1}{\sinh \gamma^0}.
\end{equation*}
\section{Constructing cmc surfaces in $\Rto$}
\label{sec:Rto_constructing_discrete_s_isothermic_cmc_surfaces}
The construction scheme for cmc surfaces in $\Rto$ is analogous to the one presented in Section \ref{sec:Rt_constructing_discrete_s_isothermic_cmc_surfaces} for cmc surfaces in $\Rt$. A sampling of curvature lines, a parameter $q$, and boundary angles, determine a hyperbolic orthogonal ring pattern and a spacelike two-sphere Koebe net. Our goal now is to reverse the construction of Theorem \ref{Rto:Thm_s_cmc_koebe}, and show how to recover a spacelike S$_1$-cmc pair from its Gauss map.

The essential part of this construction is to determine the radii of vertex spheres and face circles for the surfaces in terms of data of the ring pattern, or the Koebe nets. We refer the reader to Figure \ref{Fig:Rt_fundamental_kite} that illustrates the idea of the construction for the $\Rt$ case. Note that in our case horizontal edges always touch $S^2_-$ and vertical edges always touch $S^2_+$, see Figure \ref{Fig:Rto_fundamental_hex_gauss}.

We consider a fundamental piece $[k_{v_s}, t_b^+, \tilde{k}_{v_c}, t_b^-]$ of the central extension $k_\boxplus$, defined in \eqref{eq:central_extension_koebe}, of a spacelike two-sphere Koebe net. The piece  consists of one vertex, one face center and two touching points, where the face center is determined by \eqref{eq:face_center_koebe}.

The edge lengths of the fundamental piece $[k_{v_s}, t_b^+, \tilde{k}_{v_c}, t_b^-]$
determine the radii of the vertex spheres and face circles, $d_{v_s}, d_{v_s}^*$ and $d_{v_c}, d_{v_c}^*$, cf. \eqref{eq:Rt_koebe_length} - \eqref{eq:Rt_sphere_radii_from_koebe}. The trigonometric functions in \eqref{eq:Rt_koebe_length} and \eqref{eq:Rt_sphere_radii_from_koebe} are replaced by the corresponding hyperbolic trigonometric functions. 

For the radii of the resulting S$_1$-cmc surface pair, expressed in terms of the $\gamma$-variables \eqref{eq:hyperbolic_ring_uniformization_gamma}, we obtain 
\begin{equation}
\begin{aligned}
	\label{eq:Rto_radii}
	d_{v_s}&= \frac{1}{2}\frac{1}{\sqrt{q}}
	\left(\dn \gamma_{v_s} + q \cn \gamma_{v_s}
	\right),  \ \
	d_{v_c} = 	\frac{1}{2}\frac{1}{\sqrt{q}}
	\left(
	\frac{\dn \gamma_{v_c}  + \cn \gamma_{v_c}}
	{\sn \gamma_{v_c}}
	\right) ,
	\\
	d^*_{v_s}&= \frac{1}{2}\frac{1}{\sqrt{q}}
	\left(\dn \gamma_{v_s} - q \cn \gamma_{v_s}
	\right),  \ \ 
	d^*_{v_c} = \frac{1}{2}\frac{1}{\sqrt{q}} \left(
	\frac{\dn \gamma_{v_c} - \cn \gamma_{v_c}}
	{\sn \gamma_{v_c}}
	\right).
\end{aligned}
\end{equation}

We denote these radii by $d_v$ and $d_v^*$, where the choice $v \in V_{\text{\textcircled{s}}}$ or $v \in V_{\text{\textcircled{c}}}$ determines whether we are considering vertex spheres or face circles.

The radii $d_v$ and $d_v^*$ determine the length of all edges of the (central extension of the) S$_1$-cmc pair. The directions are determined by the edge parallelism to the two-sphere Koebe net. The following theorem ensures that the resulting spacelike surfaces with given edge lengths and directions exist. The proof is analogous to the proof of Theorem \ref{Thm:Rt_Koebe_to_cmc} \mbox{for $\Rt$.}

\begin{theorem}
	\label{Thm:Rto_Koebe_to_cmc}
	Let $k_\boxplus: V(\mathcal{S}) \rightarrow \Rt$
	be the central extension \eqref{eq:central_extension_koebe} of a spacelike two-sphere Koebe net. Then the $\Rto$-valued discrete one-form $\partial c_\boxplus$  defined by
	\begin{align}
		\label{Rto_one_form_1}
		\partial_{(v, b)}  c_\boxplus = d_v  \ \frac{\partial_{(v, b)}   k_\boxplus }{|\partial_{(v, b)}   k_\boxplus|},
	\end{align}
	and the $\Rto$-valued discrete one-form $\partial c^*_\boxplus$ defined by 
	\begin{align}
		\label{Rto_one_form_2}
		\partial_{(v, b)}  c^*_\boxplus = \pm \ d^*_v \ \frac{\partial_{(v, b)}   k_\boxplus }{|\partial_{(v, b)}   k_\boxplus|},
	\end{align}
	with $d_v$ and $d_v^*$ given by 
	\eqref{eq:Rto_radii},
	are exact. The signs $(+)$ and $(-)$ in \eqref{Rt_one_form_2} are chosen for horizontal and vertical edges, respectively. The integration of \eqref{Rto_one_form_1} and \eqref{Rto_one_form_2} defines the central extension of two surfaces $c$ and $c^*$, which, when appropriately placed, form the center nets of a spacelike S$_1$-cmc pair $s, s^*$ with the Gauss map $k_\boxplus$. The vertex sphere radii and face circle radii of $s$ and $s^*$ are given by \eqref{eq:Rto_radii}. The S$_1$-cmc pair with the Gauss map $k_\boxplus$ is unique up to translation. 
\end{theorem}

Let us consider the case of a combinatorial rectangle. The Dirichlet and Neumann boundary data on $V_B$ prescribe the metric or the cone angles on the boundary of the corresponding $S_1$-cmc surface pair, respectively. Using Theorem \ref{Thm:Rto_bdy_problem} we can formulate the following general existence and uniqueness result.

\begin{theorem}
	\label{Thm:Rto_boundary_to_cmc}
	For any Dirichlet or Neumann boundary data as in Theorem \ref{Thm:Rto_bdy_problem} there exists an unique spacelike S$_1$-cmc pair $s, s^*$ with these boundary metric or cone angles.
	\end{theorem}

\section{Maximal surfaces in the limit}
\label{sec:Rto_discrete_cmc_surfaces_as_deformations_of_minimal_surfaces}
Maximal surfaces in $\Rto$ are the natural analogues of minimal surfaces in $\Rt$.
Following the geometric characterization of S$_1$-minimal surfaces as S$_1$-isothermic surfaces that are Christoffel dual to Koebe nets \cite{BHS_2006}, we will now introduce S$_1$-maximal surfaces in $\Rto$, see also \cite{ADMPS24b}. 

\begin{definition}
	An S$_1$-maximal surface is an S$_1$-isothermic surface whose Christoffel dual is a spacelike Koebe net.
\end{definition}

Analogous to the Euclidean case, we introduce a small parameter $\epsilon$ and consider an S$_1$-cmc pair $\frac{1}{\epsilon}s, \frac{1}{\epsilon} s^*$ with $q=1-\epsilon$ and mean curvature $H=\epsilon$. 
Recall that the existence of $s, s^*$, as stated in the following Theorem,  is a consequence of Theorem 
	\ref{Thm:Rto_boundary_to_cmc}.
\begin{theorem}
		Consider a one parameter family of hyperbolic orthogonal ring patterns with $q=1-\epsilon$, fixed boundary conditions, and a limiting hyperbolic orthogonal circle pattern as $q\rightarrow 1$. Then there exists an $\epsilon$-family of associated S$_1$-cmc pairs, $\frac{1}{\epsilon}s$ and $\frac{1}{\epsilon}s^*$, with mean curvature $H = \epsilon$. In the limit $\epsilon \rightarrow 0$, $s$ converges to a spacelike Koebe net with the edges touching the unit sphere, and $\frac{1}{\epsilon}s^*$ converges to its dual S$_1$-maximal surface.
\end{theorem}

\begin{proof}
	 To analyze the behavior of the surfaces, we examine how the vertex sphere radii \eqref{eq:Rto_radii} evolve in the limit $\epsilon \rightarrow 0$. The procedure is analogous to the one in the proof of Theorem \ref{Thm:Rt_limit}. Due to
	\begin{align}
		\label{eq:Rto_identity}
		(\dn- q \cn)(\dn+ q \cn)= \epsilon(1+q), 
	\end{align} and the approximation \eqref{eq:Rt_cn_dn_approx}, 
	 we find
	\begin{align}
		\label{eq:Rto_identity2}
		a_{\dn} - a_{\dn} = \cosh(\gamma_v),
	\end{align}
where $a_{\cn}:= -\frac{\partial}{ \partial q}\cn(\beta_v, 1)$ and $a_{\dn}:= -\frac{\partial}{ \partial q}\dn(\beta_v, 1)$.
	With this, we observe that the dual radii approach
	\begin{equation*}
		%\label{eq:Rto_limit}
		\begin{aligned} \frac{1}{2} \cosh(\gamma_v),
		\end{aligned}
	\end{equation*}
	while the primal radii go to infinity,
	\begin{align*} 
		&
		\frac{1}{\epsilon}
		\frac{1}{\cosh(\gamma_v)}.
	\end{align*}
	The primal radii remain finite for the normalized surface $s=\epsilon\frac{1}{\epsilon}s$.
	 Following the same reasoning as in Theorem \ref{Thm:Rt_limit}, the two surfaces with vertex sphere radii $\cosh(\gamma_v)$ and $\frac{1}{\cosh(\gamma_v)}$ indeed form a spacelike Koebe net and an S$_1$-maximal surface, respectively.
\end{proof}

\begin{remark}
In the limit we obtain spacelike Koebe nets. Through a projective transformation, the spacelike unit sphere can be mapped to the unit sphere in $\Rt$. Since Koebe nets are invariant under projective transformations, the image of a spacelike Koebe net under this projective transformation is a Koebe net whose edges are tangent to one half of the unit sphere in $\Rt$. Conversely, by first applying a projective transformation that maps a Koebe net in $\Rt$ to tangent to half of the unit sphere in $\Rt$, and then applying a transformation that maps the unit sphere of $\Rt$ to the unit sphere in $\Rto$, we obtain a spacelike Koebe net tangent to the upper half of the unit sphere in $\Rto$. Thus, the difference between the treatment of Euclidean minimal surfaces and Lorentz maximal surfaces lies only in the handling of boundary conditions.
\end{remark}

\section{Examples of cmc surfaces in $\Rto$}
\label{sec:Rto_examples}

By Theorem \ref{Thm:Rto_boundary_to_cmc}, we can construct S$_1$-cmc surfaces in $\Rto$ from any given Neumann or Dirichlet boundary data. 

In Figure \ref{Fig:Rto_Example_1} we present two examples: an S$_1$-maximal surface with corresponding hyperbolic orthogonal circle pattern, and an
S$_1$-cmc surface with corresponding hyperbolic orthogonal ring pattern. Both surfaces are constructed using the same boundary data, specifically chosen so that the centers of the boundary circles or rings lie on a hyperbolic quadrilateral with ideal vertices.

\begin{figure}[h]
	\begin{minipage}{.42\linewidth}
		\includegraphics[width=\linewidth]{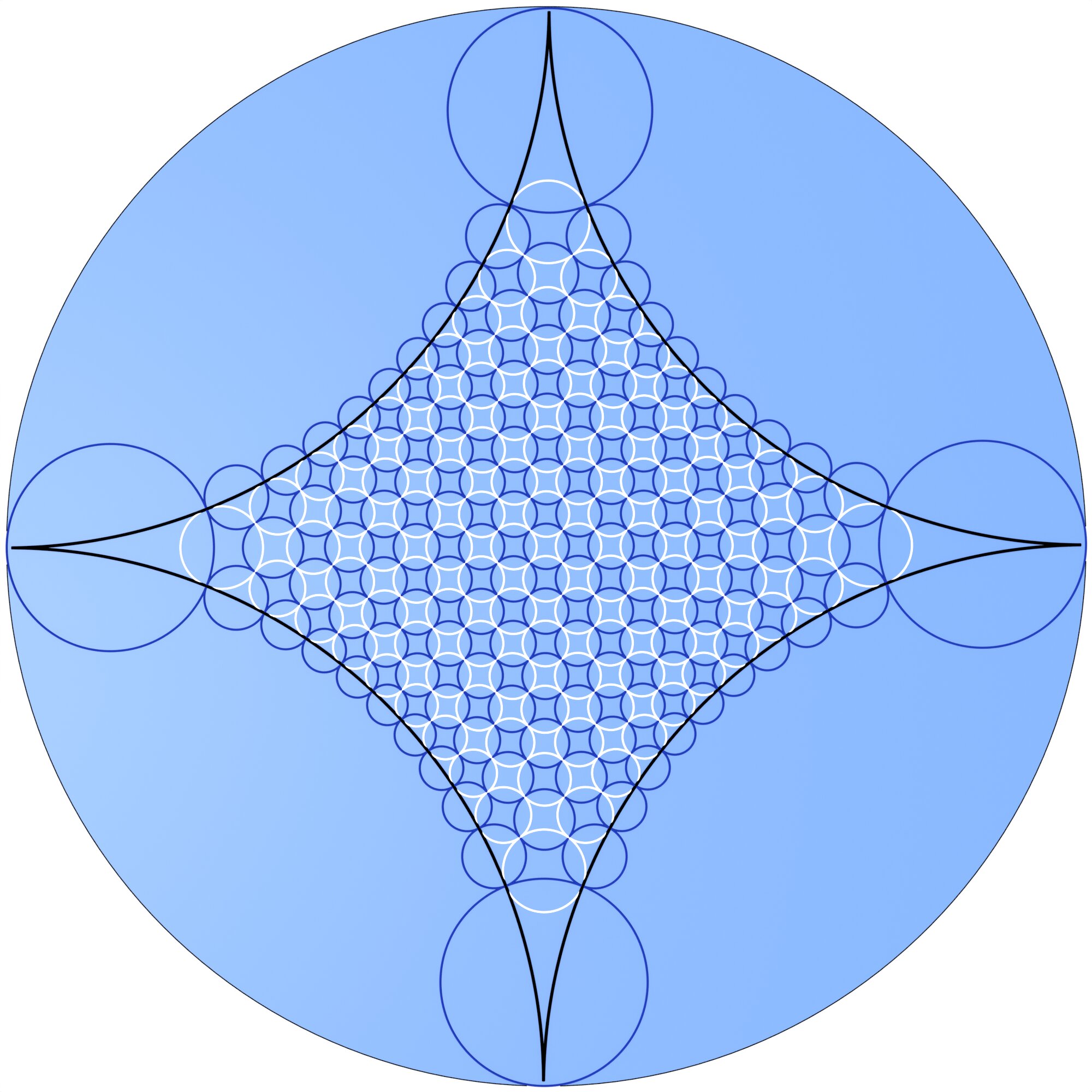}
	\end{minipage}
	\hspace{.2cm}
	\begin{minipage}{.55\linewidth}
		\includegraphics[width=\linewidth]{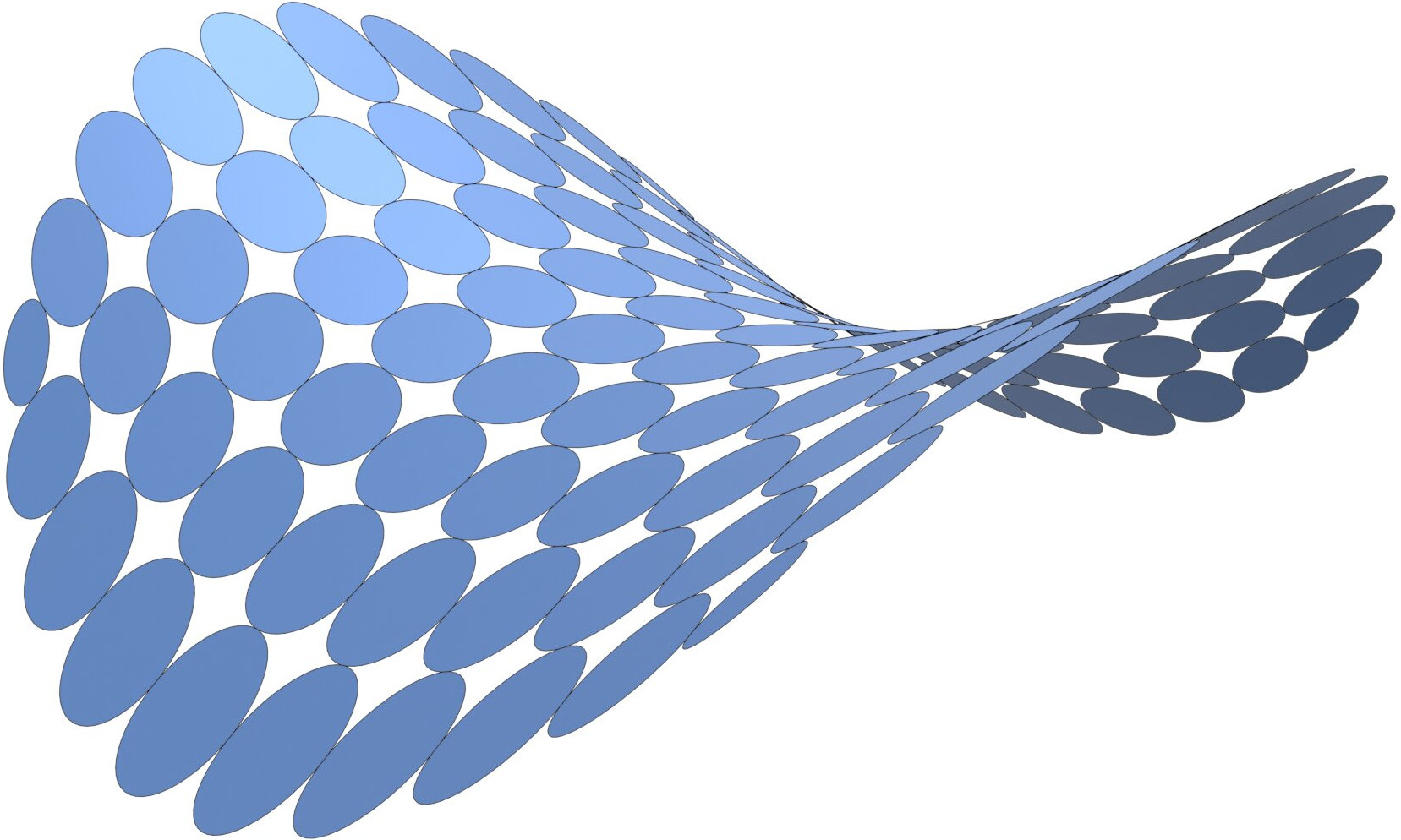}
	\end{minipage}
	
	\vspace{.5cm}
	
	\begin{minipage}{.42\linewidth}
		\includegraphics[width=\linewidth]{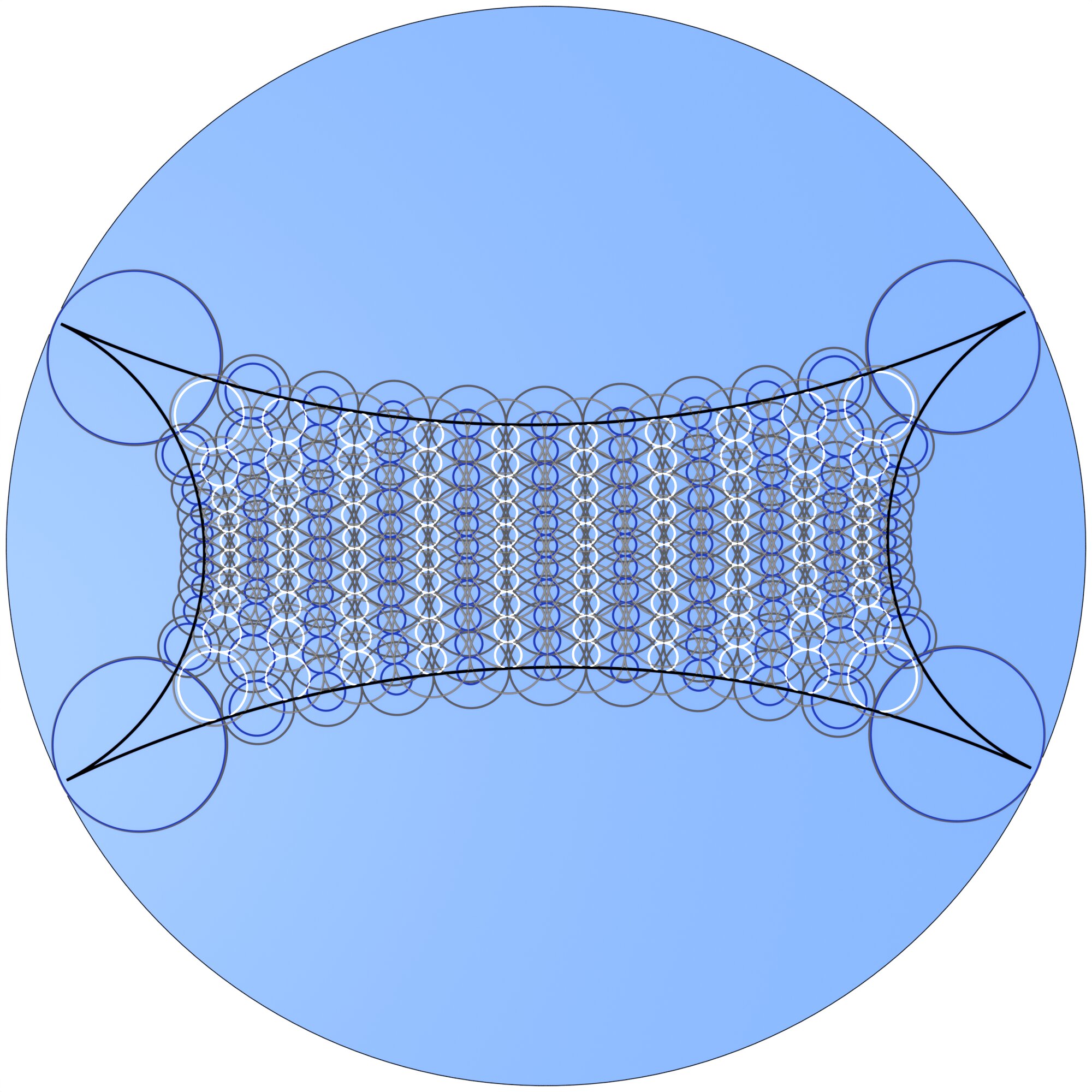}
	\end{minipage}
	\hspace{.2cm}
	\begin{minipage}{.55\linewidth}
		\includegraphics[width=\linewidth]{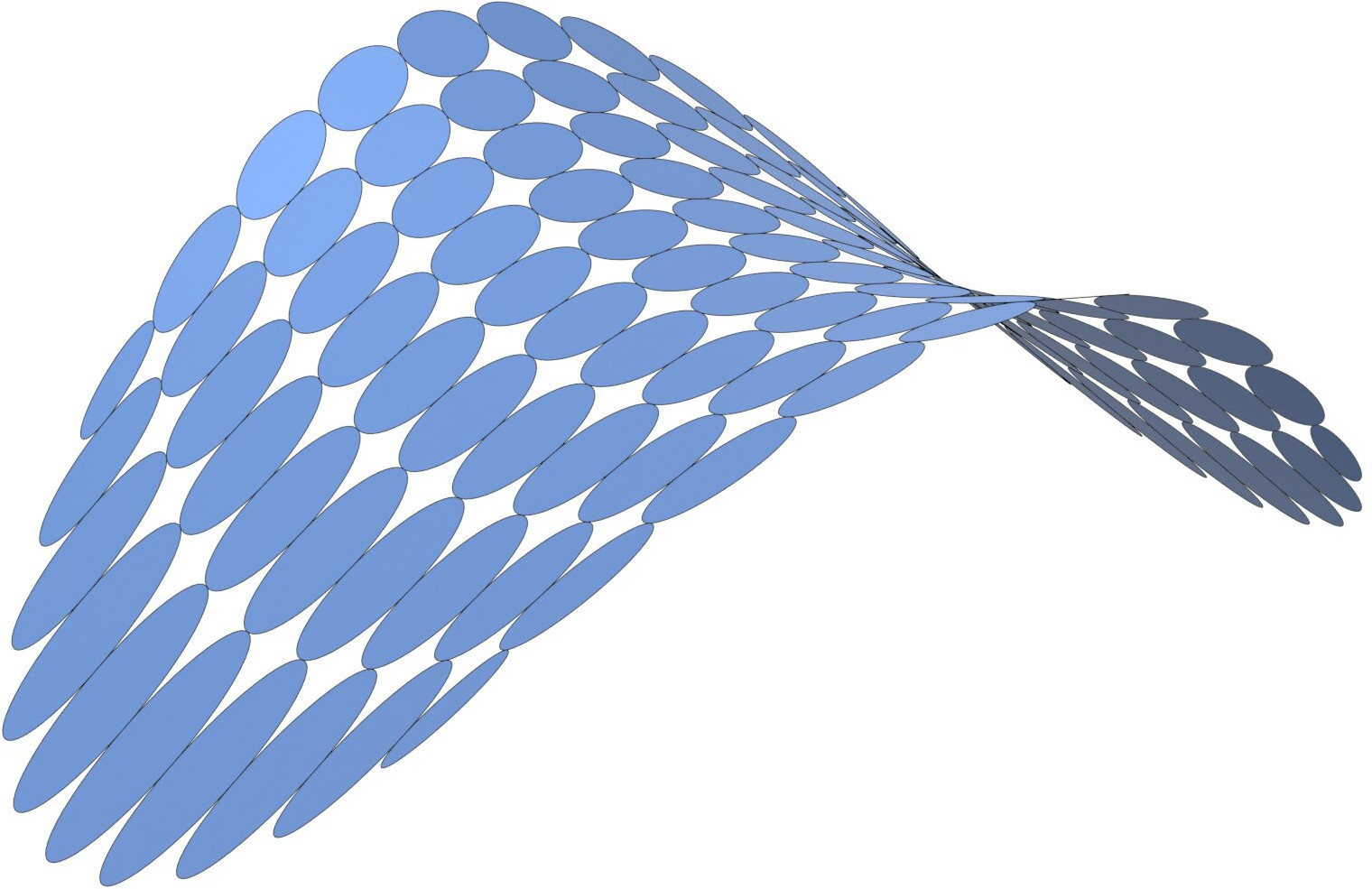}
	\end{minipage}
	
	\caption{A hyperbolic orthogonal circle pattern ($q = .99999$) and the corresponding S$_1$-minimal surface (top) and a hyperbolic orthogonal ring pattern ($q=0.99$) with corresponding S$_1$-cmc surface. The patterns are constructed using identical Neumann boundary data. The prescribed angles are approximately zero at the corners and $\pi$ at the white boundary vertices which lead to hyperbolic quadrilaterals with ideal vertices. }
	\label{Fig:Rto_Example_1}
\end{figure}

\setlength{\itemsep}{1cm}

\bibliographystyle{acm}
\bibliography{cmc_discrete}

\end{document}